\documentclass[12pt]{article}

\usepackage{amsfonts}
\usepackage{amsmath}
\usepackage{amssymb,euscript}

\newcommand{\Xcomment}[1]{}

\makeatletter
\renewcommand{\section}{\@startsection{section}{1}{0pt}%
{-3.5ex plus -1ex minus -.2ex}{2.3ex plus .2ex}%
{\normalfont\Large}}
\makeatother

 \makeatletter
\renewcommand{\subsection}{\@startsection{subsection}{2}{0pt}%
{-3.0ex plus -1ex minus -.2ex}{1.5ex plus .2ex}%
{\normalfont\normalsize\bf}}
 \makeatother

 \makeatletter
\renewcommand{\subsubsection}{\@startsection{subsubsection}{3}{0pt}%
{-3.0ex plus -1ex minus -.2ex}{-1.5ex plus -.2ex}%
{\normalfont\normalsize\bf}}
 \makeatother

\newtheorem{theorem}{Theorem}[section]
\newtheorem{lemma}[theorem]{Lemma}
\newtheorem{corollary}[theorem]{Corollary}
\newtheorem{prop}[theorem]{Proposition}

\newenvironment{proof}{\noindent{\bf Proof}~}%
{$\qed$\medskip}

\makeatletter \@addtoreset{equation}{section} \makeatother

\newenvironment{numitem}{\refstepcounter{equation}\begin{enumerate}%
\item[(\thesection.\arabic{equation})]$\quad$}{\end{enumerate}}
\newenvironment{numitem1}{\refstepcounter{equation}\begin{enumerate}%
\item[(\thesection.\arabic{equation})]}{\end{enumerate}}

\newcommand{\refeq}[1]{(\ref{eq:#1})}  % reference to equation

\def\qed{ \ \vrule width.1cm height.3cm depth0cm}

\def\tilde{\widetilde}
\def\hat{\widehat}
\def\eps{\varepsilon}

\def\Rset{{\mathbb R}}
\def\Zset{{\mathbb Z}}

\def\Bscr{\mathcal{B}}

\def\Escr{\mathcal{E}}
\def\Fscr{\mathcal{F}}
\def\Gscr{\mathcal{G}}

\def\Kscr{\mathcal{K}}

\def\Mscr{\mathcal{M}}

\def\frakB{{B}}
\def\frakC{{C}}

\def\Kmn{\mathcal{K}^{(-n)}}
\def\Kmone{\mathcal{K}^{(-1)}}

\def\prv{\check{v}}

\def\Kup{K^{\uparrow}}
\def\Klow{K^{\downarrow}}
\def\Piup{\Pi^{\uparrow}}
\def\Pilow{\Pi^{\downarrow}}
\def\vup{v^{\uparrow}}
\def\vlow{v^{\downarrow}}

\def\aup{a^{\uparrow}}
\def\bup{b^{\uparrow}}

\def\parup{c^{\uparrow}}
\def\parlow{c^{\downarrow}}
\def\heartup{\hslash^{\uparrow}}
\def\heartlow{\hslash^{\downarrow}}
\def\parmid{c^{\uparrow\downarrow}}
\def\parupdown{c^{\uparrow\downarrow}}
\def\Kmid{K^{\uparrow\downarrow}}
\def\heartmid{\hslash^{\uparrow\downarrow}}
\def\heartupdown{\hslash^{\uparrow\downarrow}}
\def\bmid{\hslash^{\downarrow\uparrow}}
\def\Dupdown{\Delta^{\uparrow\downarrow}}

\def\parupup{c^{\uparrow\uparrow}}
\def\heartupup{\hslash^{\uparrow\uparrow}}
\def\Dupup{\Delta^{\uparrow\uparrow}}

\def\Kpu{K'^{\uparrow}}
\def\Kpd{K'^{\downarrow}}
\def\Pipu{\Pi'^{\uparrow}}
\def\Pipd{\Pi'^{\downarrow}}
\def\zu{z^{\uparrow}}
\def\zd{z^{\downarrow}}
\def\zdu{z^{\downarrow\uparrow}}
\def\zud{z^{\uparrow\downarrow}}
\def\zuu{z^{\uparrow\uparrow}}
\def\cpu{c'^{\uparrow}}
\def\cpd{c'^{\downarrow}}
\def\cpuu{c'^{\uparrow\uparrow}}
\def\apu{a'^{\uparrow}}

\def\bpu{b'^{\uparrow}}

\def\Dpu{\Delta'^{\uparrow}}
\def\Dpd{\Delta'^{\downarrow}}
\def\Dpud{\Delta'^{\uparrow\downarrow}}
\def\Dpuu{\Delta'^{\uparrow\uparrow}}

\def\hpu{\hslash'^{\uparrow}}
\def\hpd{\hslash'^{\downarrow}}
\def\hpuu{\hslash'^{\uparrow\uparrow}}

\def\hpud{\hslash'^{\uparrow\downarrow}}
\def\bfa{a}
\def\bfb{b}
\def\bfc{c}

\def\bfp{p}
\def\bfq{q}

\def\bfzero{{\bf 0}}
\def\bfone{{\bf 1}}
\def\bftwo{{\bf 2}}
\def\bfthree{{\bf 3}}
\def\bffour{{\bf 4}}

\def\tone{\tilde\bfone}
\def\ttwo{\tilde\bftwo}
\def\bone{\bar\bfone}
\def\btwo{\bar\bftwo}

\def\Pin{P^{\rm in}}
\def\Pout{P^{\rm out}}
\def\ellin{t}
\def\ellout{h}

\def\eNW{e^{\rm NW}}
\def\eSW{e^{\rm SW}}
\def\eNE{e^{\rm NE}}
\def\eSE{e^{\rm SE}}

\begin{document}

\begin{center}
{\large\bf  On the combinatorial structure of crystals of types A,\,B,\,C}%
\footnote[1]{Supported by RFBR grant 10-01-9311-CNRSL\_\,a.}
\end{center}
 \medskip

\begin{center}
{\sc Vladimir~I.~Danilov}\footnote[2] {Central Institute of Economics and
Mathematics of the RAS, 47, Nakhimovskii Prospect, 117418 Moscow, Russia;
emails: danilov@cemi.rssi.ru and koshevoy@cemi.rssi.ru.},
{\sc Alexander~V.~Karzanov}\footnote[3]{Institute for System Analysis of the
RAS, 9, Prospect 60 Let Oktyabrya, 117312 Moscow, Russia; email:
sasha@cs.isa.ru. A part of this research was done while this author was
visiting Equipe Combinatoire et Optimisation, Univ. Paris-6, and Institut
f\"ur Diskrete Mathematik, Univ. Bonn.}, \\
{\sc and Gleb~A.~Koshevoy$^2$}
\end{center}

 \bigskip
 \begin{quote}
 {\bf Abstract.}
 {\small
Regular $A_n$-, $B_n$- and $C_n$-\emph{crystals} are edge-colored directed
graphs, with ordered colors $1,2,\ldots,n$, which are related to
representations of quantized algebras $U_q(\mathfrak{sl}_{n+1})$,
$U_q(\mathfrak{sp}_{2n})$ and $U_q(\mathfrak{so}_{2n+1})$, respectively. We
develop combinatorial methods to reveal refined structural properties of such
objects.

Firstly, we study subcrystals of a regular $A_n$-crystal $K$ and characterize
pairwise intersections of maximal subcrystals with colors $1,\ldots,n-1$ and
colors $2,\ldots,n$. This leads to a recursive description of the structure of
$K$ and provides an efficient procedure of assembling $K$.

Secondly, using merely combinatorial means, we demonstrate a relationship
between regular $B_n$-crystals (resp. $C_n$-crystals) and regular
\emph{symmetric} $A_{2n-1}$-crystals (resp. $A_{2n}$-crystals).
 \smallskip

\smallskip

  }
  \end{quote}

%-------------------- SEC. 1
\section{Introduction} \label{sec:intr}

{\em Crystals} are certain ``exotic'' edge-colored graphs. This graph-theoretic
abstraction, introduced by Kashiwara~\cite{kas-90,kas-95}, has proved its
usefulness in the theory of representations of Lie algebras and their quantum
analogues. A (general) \emph{crystal} is a directed graph $K$ such that: the
edges are partitioned into $n$ subsets, or \emph{color} classes, labeled
$1,\ldots,n$, each connected monochromatic subgraph of $K$ is a finite path,
and there is an interrelation between the lengths of such paths described in
terms of the $n\times n$ Cartan matrix $M=(m_{ij})$ related to a given Lie
algebra $\mathfrak{g}$. This interrelation is: for colors $i,j$, any edge
$(u,v)$ with color $i$ satisfies $(h_j(u)-t_j(u))-(h_j(v)-t_j(v))=m_{ij}$,
where for a vertex $v'$, $h_j(v')$ (resp. $t_j(v')$) denotes the length of the
maximal path colored $j$ that begins (resp. ends) at $v'$. Throughout we
assume, w.l.o.g., that any crystal in question is (weakly) connected, and call
an edge with color $i$ an $i$-\emph{edge}. Depending on Cartan matrices,
several {\em types} of crystals are distinguished.

Of most interest are crystals of representations, or {\em regular} crystals.
They are associated to elements of a certain basis of the highest weight
integrable modules (representations) over a quantized algebra
$U_q(\mathfrak{g})$. There are known ``global'' models to characterize the
regular crystals for a variety of types: generalized Young
tableaux~\cite{KN-94}, Lusztig's canonical bases~\cite{Lusztig}, Littelmann's
path model~\cite{Lit-95,Littl}, and some others.

This paper continues our combinatorial study of crystals begun
in~\cite{A2,cross,B2} and considers $n$-colored regular crystals of three
types: A,\,B,\,C, where the number $n$ of colors is arbitrary. Recall that {\em
type~A} (concerning $\mathfrak{g}=\mathfrak{sl}_{n+1}$) is related to the
Cartan matric $M$ with: $m_{ij}=-1$ if $|i-j|=1$, $m_{ij}=0$ if $|i-j|>1$, and
$m_{ii}=2$. For {\em type~B} (concerning $\mathfrak{g}=\mathfrak{sp}_{2n}$),
the matrix is obtained from the above $M$ by replacing $m_{n-1,n}$ by $-2$. And
for {\em type~C} (concerning $\mathfrak{g}=\mathfrak{so}_{2n+1}$), one should
replace $m_{n,n-1}$ by $-2$. We will refer to a regular $n$-colored crystal of
type A (B,\, C) as an $A_n$-{\em crystal} (resp. $B_n$-, $C_n$-{\em crystal})
and omit the index $n$ when the number of colors is not specified.

It is known that the (finite) regular crystals $K$ of these types have the
following properties. (i) $K$ is acyclic (i.e. without directed cycles) and has
exactly one zero-indegree vertex, called the {\em source}, and exactly one
zero-outdegree vertex, called the {\em sink} of $K$. (ii) For any
$I\subseteq\{1,\ldots,n\}$, each (inclusion-wise) maximal connected subgraph of
$K$ whose edges have colors from $I$ is a regular crystal related to the
corresponding $I\times I$ submatrix of the Cartan matrix of $K$. Throughout,
speaking of a subcrystal of $K$, we will always mean a subgraph of this kind.

Two-colored subcrystals are of most importance, due to the result
in~\cite{KKM-92} that for a crystal with exactly one zero-indegree vertex, the
regularity of all such subcrystals implies the regularity of the whole crystal.
Let $K'$ be a two-colored subcrystal with colors $i,j$ in $K$. Then for type~A,
~$K'$ is the Cartesian product of a path with color $i$ and a path with color
$j$ (forming an $A_1\times A_1$-crystal) when $|i-j|>1$, and an $A_2$-crystal
when $|i-j|=1$. For type~B, the only difference is that $K'$ is a $B_2$-crystal
when $(i,j)=(n-1,n)$, and the corresponding submatrix is viewed as
$\binom{\;\;2\;-2}{-1\;\;2}$. And for type~C, ~$K'$ with $(i,j)=(n-1,n)$ is
again a $B_2$-crystal but the corresponding submatrix is now
$\binom{\;\;2\;-1}{-2\;\;2}$. The A-crystals belong to the group of {\em
simply-laced} crystals (defined by the requirement that each two-colored
subcrystal is of type $A_1\times A_1$ or $A_2$), and the B- and C-crystals
belong to the group of {\em doubly-laced} ones (where each two-colored
subcrystal is of type $A_1\times A_1$ or $A_2$ or $B_2$); cf.,
e.g.,~\cite{Stem}.

Throughout the paper we are going to deal with regular crystals only, and for
this reason the adjective ``regular'' will usually be omitted. It should be
noted that even in case of $A_2$- and $B_2$-crystals, the corresponding
specifications of ``global'' models from~\cite{KN-94,Lit-95,Littl,Lusztig} are
rather intricate to work with directly. Fortunatelly, in the last decade there
appeared more explicit and enlightening ways to define these crystals, via
``local'' graph-theoretic axioms or by use of direct combinatorial
constructions. In case of $A_2$-crystals, a short list of ``local'' defining
axioms is pointed out by Stembridge~\cite{Stem} and an explicit construction is
given in~\cite{A2}. According to that construction, any $A_2$-crystal can be
obtained from an $A_1\times A_1$-crystal by replacing each monochromatic path
of the latter by a graph viewed as a triangular half of a directed square grid.
In case of $B_2$-crystals, both ``local'' axioms and a direct combinatorial
construction are given in~\cite{B2}. It is shown there that a $B_2$-crystal can
be obtained from an $A_2$-crystal by replacing each monochromatic path by a
certain quadrangular part of a square grid. Also~\cite{B2} describes an
alternative combinatorial construction for $B_2$-crystals, the so-called {\em
worm model}. This model will be extensively used in this paper. (For some other
results on $B_2$-crystals, see~\cite{Stern}.)

An important fact is that for any $n$-tuple $c=(c_1,\ldots,c_n)$ of nonnegative
integers, there exists exactly one $A_n$-crystal $K$ such that each $c_i$ is
equal to the length of the maximal path with color $i$ beginning at the source
(for a short combinatorial proof, see~\cite[Sec.~2]{cross}). A similar property
takes place for $B_n$- and $C_n$-crystals. We denote a crystal $K$ (of a given
type) determined by $c$ in this way by $K(c)$, and refer to $c$ as the {\em
parameter} of this crystal.

When $n>2$, the combinatorial structure of $A_n$-crystals becomes rather
complicated, even for $n=3$. Attempting to learn more about this structure, we
elaborated in~\cite{cross} a new combinatorial construction, the so-called {\em
crossing model} (which is a refinement of the Gelfand-Tsetlin pattern
model~\cite{GT-50}). This powerful tool has helped us to reveal more structural
features of an $A_n$-crystal $K=K(c)$. In particular, $K$ has the so-called
{\em principal lattice}, a set $\Pi$ of vertices with the following nice
properties:

(P1) $\Pi$ contains the source and sink of $K$, and the vertices $v\in \Pi$
correspond to the elements of the integer box $\Bscr(c):=\{a\in\Zset^n\colon
0\le a\le c\}$; we write $v=\prv[a]$;

(P2) For any $a,a'\in\Bscr(c)$ with $a\le a'$, the \emph{interval} of $K$ from
$\prv[a]$ to $\prv[a']$ (i.e. the subgraph of $K$ formed by the vertices and
edges contained in (directed) paths from $\prv[a]$ to $\prv[a']$) is isomorphic
to the $A_n$-crystal $K(a'-a)$, and its principal lattice consists of the
principal vertices $\prv[a'']$ of $K$ with $a\le a''\le a'$;

(P3) The set $\Kscr^{(-n)}$ of $(n-1)$-colored subcrystals $K'$ of $K$ having
colors $1,\ldots,n-1$ is bijective to $\Pi$; more precisely, $K'\cap\Pi$
consists of a single vertex (called the \emph{heart} of $K'$ w.r.t. $K$); and
similarly for the set $\Kscr^{(-1)}$ of subcrystals of $K$ with colors
$2,\ldots,n$.

(A sort of principal lattice can be introduced for B- and C-crystals as well;
it satisfies~(P1) and (P2) but not (P3); see Remark~5 in the end of
Section~\ref{sec:proofB3-B4}.)
\smallskip

For $a\in\Bscr(c)$, let $\Kup[a]$ (resp. $\Klow[a]$) denote the subcrystal in
$\Kscr^{(-n)}$ (resp. in $\Kscr^{(-1)}$) that contains the principal vertex
$\prv[a]$; we call it the {\em upper} (resp. {\em lower}) subcrystal at $a$. It
is shown in~\cite{cross} that the parameter of this subcrystal is expressed by
a linear function of $c$ and $a$, and that the number of upper (lower)
subcrystals with a fixed parameter $c'$ is expressed by a piece-wise linear
function of $c$ and $c'$.
\smallskip

In this paper, we further use the crossing model, aiming to obtain a refined
description of the structure of an $A_n$-crystal $K$. We study the
intersections of subcrystals $\Kup[a]$ and $\Klow[b]$ for all $a,b\in\Bscr(c)$.
This intersection may be empty or consist of one or more subcrystals with
colors $2,\ldots,n-1$, called \emph{middle} subcrystals of $K$. Each of these
middle subcrystals $\tilde K$ is therefore a lower subcrystal of $\Kup[a]$ and
an upper subcrystal of $\Klow[b]$; so $\tilde K$ has a unique vertex $z$ in the
principal lattice $\Piup$ of the former, and a unique vertex $z'$ in the
principal lattice $\Pilow$ of the latter. Our main result on A-crystals
(Theorem~\ref{tm:mainA}) and its consequences give explicit relations between
$a$, $b$, the locus of $z$ in $\Piup$, the locus of $z'$ in $\Pilow$, and the
parameters of $K$ and $\tilde K$.

This gives rise to a recursive procedure of assembling of the $A_n$-crystal
$K(c)$. More precisely, suppose that the $(n-1)$-colored crystals $\Kup[a]$ and
$\Klow[b]$ for all $a,b\in\Bscr(c)$ are already constructed. Then we can
combine these subcrystals to obtain the desired crystal $K(c)$, by properly
identifying the corresponding middle subcrystals (if any) for each pair
$\Kup[a],\Klow[b]$. This recursive method is implemented as an efficient
algorithm which, given a parameter $c\in\Zset_+^n$, outputs the crystal $K(c)$.
The running time of the algorithm and the needed space are bounded by
$Cn^2|K(c)|$, where $C$ is a constant and $|K(c)|$ is the size of $K(c)$. (It
may be of practical use for small $n$ and $c$; in general, an $A_n$-crystal has
``dimension'' $\frac{n(n+1)}{2}$ and its size grows sharply by increasing $c$.)

The second part of the paper is devoted to $n$-colored (regular) B-crystals.
With the help of Theorem~\ref{tm:mainA}, we explain, using merely combinatorial
means, that any B-crystal can be extracted from a symmetric A-crystal. More
precisely, given $c\in\Zset_+^n$, define the $(2n-1)$-tuple $c'$ by
$c'_i=c'_{2n-i}:=c_i$ for $i=1,\ldots,n$. The $A_{2n-1}$-crystal $K=K(c')$ has
a canonical involution $\sigma$ on the vertices under which the image
$(\sigma(u),\sigma(v))$ of an $i$-edge $(u,v)$ is a $(2n-i)$-edge. We say that
$K$ is \emph{symmetric} and that a vertex $v$ with $\sigma(v)=v$ is {\em
self-complementary}; let $S$ be the set of such vertices. The {\em symmetric
extract} from $K$ is the $n$-colored graph $\tilde K$ whose vertex set is $S$
and whose edges are defined as follows: (i) the edges of $\tilde K$ colored $n$
are exactly the $n$-edges of $K$ connecting elements of $S$, and (ii) for
$i<n$, vertices $u,v\in S$ are connected in $\tilde K$ by edge $(u,v)$ colored
$i$ if $K$ contains a 2-edge path from $u$ to $v$ whose edges are colored $i$
and $(2n-i)$. We prove that $\tilde K$ is isomorphic to the $B_n$-crystal with
the parameter $c$. The crucial part of the proof is a verification in case
$n=2$.

In the final part, we explain a similar fact for C-crystals; now the
$C_n$-crystals are extracted from symmetric $A_{2n}$-crystals.

It should be noted that such a way of constructing B- and C-crystals from
corresponding symmetric A-crystals has been known; this can be concluded from
the work of Naito and Sagaki~\cite{NS} where the argument relies on a
sophisticated path model. Our goal is to give alternative proofs which are
direct and purely combinatorial. We take advantages from rather transparent
axiomatics and constructions for crystals of types A,B,C, and appeal to
structural results from Section~\ref{sec:ass_A}.
\smallskip

This paper is organized as follows. Section~\ref{sec:prelim} is devoted to
basic definitions and backgrounds. Here we recall ``local'' axioms and the
crossing model for A-crystals, and review needed results on the principal
lattice $\Pi$ of an $A_n$-crystal and relations between $\Pi$ and the
$(n-1)$-colored subcrystals from~\cite{cross}. Section~\ref{sec:ass_A} gives a
recursive description of the structure of an $A_n$-crystal $K$ and the
algorithm of assembling $K$; here we rely on the main structural result
(Theorem~\ref{tm:mainA}) proved in the next Section~\ref{sec:proof}. The
devised assembling method is illustrated in Section~\ref{sec:illustr} for two
special cases of A-crystals: for an arbitrary $A_2$-crystal (in which case the
method can be compared with the explicit combinatorial construction
in~\cite{A2}), and for the particular $A_3$-crystal $K(1,1,1)$. The rest of the
paper is devoted to B- and C-crystals. Our combinatorial proof of the theorem
that the $B_n$-crystals are exactly the extracts from symmetric
$A_{2n-1}$-crystals is given in Sections~\ref{sec:Bn}--\ref{sec:proofB3-B4}.
Here Section~\ref{sec:Bn} reduces the task to $n=2$, Section~\ref{sec:worm}
recalls the worm model from~\cite{B2}, and the crucial
Section~\ref{sec:proofB3-B4} gives a proof for $n=2$, relying on the
construction of $B_2$-crystals via the worm model. An important step in the
proof consists in representing the self-complementary vertices of a symmetric
$A_3$-crystal as integer points of a certain 4-dimensional polytope (in
Theorem~\ref{tm:S-constr}). Arguing in a similar fashion, Section~\ref{sec:A2n}
gives a combinatorial proof of the theorem that the extracts from symmetric
$A_{2n}$-crystals are $C_n$-crystals. Most technical claims used in
Sections~\ref{sec:proofB3-B4} and~\ref{sec:A2n} are proved in the Appendix.

%-------------------- SEC. 2
\section{Preliminaries} \label{sec:prelim}

In this section we recall definitions and some basic properties of (regular)
crystals of types A,\,B and C, referring to them as A-, B- and C-{\em
crystals}, respectively, and review results from~\cite{cross} that will be
important for further purposes.

An $n$-{\em colored crystal} is a certain directed graph $K$ whose edge set
$E(K)$ is partitioned into $n$ subsets $E_1,\ldots,E_n$, denoted as
$K=(V(K),E_1\sqcup\ldots\sqcup E_n)$.  We assume that $K$ is (weakly) {\em
connected}, i.e. it is not the disjoint union of two nonempty graphs. We say
that an edge $e\in E_i$ has {\em color} $i$, or is an $i$-{\em edge}. When
speaking of a {\em subcrystal} $K'$ of $K$, we always mean that $K'$ is
inclusion-wise {\em maximal} among the connected subgraphs having the same set
of colors as $K'$.

%---------------SSEC 2.1
 \subsection{Crystals of type A} \label{ssec:typeA}

Stembridge~\cite{Stem} pointed out a list of ``local'' graph-theoretic axioms
for the regular simply-laced crystals. The A-crystals form a subclass of those
and are defined by axioms (A1)--(A5) below; we give the axiomatics in a
slightly different, but equivalent, form compared with~\cite{Stem}. Let $K$ be
an $n$-edge-colored graph as before.

Unless explicitly stated otherwise, by a {\em path} we mean a simple finite
directed path, i.e. a sequence of the form $(v_0,e_1,v_1,\ldots,e_k,v_k)$,
where $v_0,v_1,\ldots,v_k$ are distinct vertices and each $e_i$ is an edge from
$v_{i-1}$ to $v_i$ (admitting $k=0$).

The first axiom concerns the
structure of monochromatic subgraphs of $K$.
\begin{itemize}
\item[(A1)] For $i=1,\ldots,n$, each connected subgraph of $(V(K),E_i)$ is
a path.
  \end{itemize}

So each vertex of $K$ has at most one incoming $i$-edge and at most one
outgoing $i$-edge, and therefore one can associate to the set $E_i$ a {\em
partial invertible operator} $F_i$ acting on vertices: $(u,v)$ is an $i$-edge
if and only if $F_i$ {\em acts} at $u$ and $F_i(u)=v$ (or $u=F_i^{-1}(v)$,
where $F_i^{-1}$ is the partial operator inverse to $F_i$). Since $K$ is
connected, one can use the operator notation to express any vertex via another
one. For example, the expression $F_1^{-1}F_3^2F_2(v)$ determines the vertex
$w$ obtained from a vertex $v$ by traversing 2-edge $(v,v')$, followed by
traversing 3-edges $(v',u)$ and $(u,u')$, followed by traversing 1-edge
$(w,u')$ in backward direction. Emphasize that every time we use such an
operator expression in what follows, this automatically says that all
corresponding edges do exist in $K$.

We refer to a monochromatic path with color $i$ on the edges as an $i$-{\em
path}, and to a maximal $i$-path as an $i$-{\em line} (the latter is an
$A_1$-subcrystal of $K$). The $i$-line passing through a given vertex $v$
(possibly consisting of the only vertex $v$) is denoted by $P_i(v)$, its part
from the first vertex to $v$ by $\Pin_i(v)$, and its part from $v$ to the last
vertex by $\Pout_i(v)$ (the \emph{tail} and \emph{head} parts of $P$ w.r.t.
$v$). The lengths (i.e. the numbers of edges) of $\Pin_i(v)$ and $\Pout_i(v)$
are denoted by $\ellin_i(v)$ and $\ellout_i(v)$, respectively.

Axioms (A2)--(A5) concern interrelations of different colors $i,j$. They say
that each component of the two-colored graph $(V(K),E_i\sqcup E_j)$ forms an
$A_2$-crystal when colors $i,j$ are {\em neighboring}, which means that
$|i-j|=1$, and forms an $A_1\times A_1$-crystal otherwise.

When an edge of a color $i$ is traversed, the head and tail part lengths of
lines of another color $j$ behave as follows:
 \begin{itemize}
 \item[(A2)]
For different colors $i,j$ and for an edge $(u,v)$ with color $i$, one
holds $\ellin_j(v)\le\ellin_j(u)$ and $\ellout_j(v)\ge\ellout_j(u)$. The
value $(\ellout_j(u)-\ellin_j(u))-(\ellout_j(v)-\ellin_j(v))$ is the
constant $m_{ij}$ equal to $-1$ if $|i-j|=1$, and 0 otherwise.
Furthermore, $h_j$ is convex on each $i$-path, in the sense that if
$(u,v),(v,w)$ are consecutive $i$-edges, then $h_j(u)+h_j(w)\ge 2h_j(v)$.
  \end{itemize}
These constants $m_{ij}$ are just the coefficients of the Cartan $n\times n$
matrix $M$ related to the crystal type A and the number $n$ of colors. Each
diagonal entry $m_{ii}$ equals 2, which agrees with the trivial relation
$(\ellout_i(u)-\ellin_i(u))-(\ellout_i(v)-\ellin_i(v))=2$ for an $i$-edge
$(u,v)$.

It follows that for neighboring colors $i,j$, each $i$-line $P$ contains a
unique vertex $r$ such that: when traversing any edge $e$ of $P$ before $r$
(i.e. $e\in\Pin_i(r)$), the tail length $\ellin_j$ decreases by 1 while the
head length $\ellout_j$ does not change, and when traversing any edge of $P$
after $r$, $\ellin_j$ does not change while $\ellout_j$ increases by 1. This
$r$ is called the {\em critical} vertex for $P,i,j$. To each $i$-edge $e=(u,v)$
we associate {\em label} $\ell_{j}(e):=\ellout_j(v)-\ellout_j(u)$; then
$\ell_j(e)\in\{0,1\}$ and $t_j(v)=t_j(u)-1+\ell_j(e)$. Emphasize that the
critical vertices on an $i$-line $P$ w.r.t. its neighboring colors $j=i-1$ and
$j=i+1$ may be different (and so are the edge labels on $P$).

Two operators $F=F_i^{\alpha}$ and $F'=F_j^\beta$, where $\alpha,\beta\in
\{1,-1\}$, are said to {\em commute} at a vertex $v$ if each of $F,F'$ acts at
$v$ (i.e. corresponding $i$-edge and $j$-edge incident with $v$ exist) and
$FF'(v)=F'F(v)$. The third axiom indicates situations when such operators
commute for neighboring $i,j$.

\begin{itemize}
\item[(A3)] Let $|i-j|=1$. (a) If a vertex $u$ has outgoing $i$-edge $(u,v)$
and outgoing $j$-edge $(u,v')$ and if $\ell_{j}(u,v)=0$, then
$\ell_{i}(u,v')=1$ and $F_i,F_j$ commute at $v$.
Symmetrically: (b) if a vertex $v$ has incoming $i$-edge $(u,v)$ and
incoming $j$-edge $(u',v)$ and if $\ell_{j}(u,v)=1$, then
$\ell_{i}(u',v)=0$ and $F_i^{-1},F_j^{-1}$ commute at $v$.
(See the picture.)
  \end{itemize}
 \begin{center}
  \unitlength=1mm
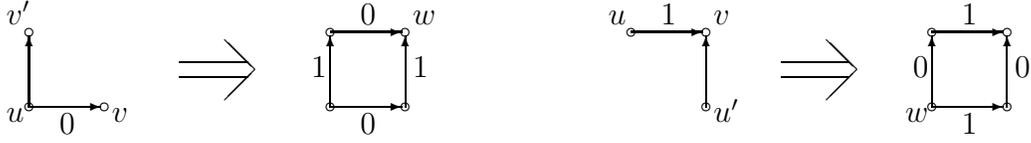
  \begin{picture}(140,20)
\put(5,5){\circle{1.0}}
\put(15,5){\circle{1.0}}
\put(45,5){\circle{1.0}}
\put(55,5){\circle{1.0}}
\put(95,5){\circle{1.0}}
\put(125,5){\circle{1.0}}
\put(135,5){\circle{1.0}}
\put(5,15){\circle{1.0}}
\put(45,15){\circle{1.0}}
\put(55,15){\circle{1.0}}
\put(85,15){\circle{1.0}}
\put(95,15){\circle{1.0}}
\put(125,15){\circle{1.0}}
\put(135,15){\circle{1.0}}
\put(5,5){\vector(1,0){9.5}}
\put(45,5){\vector(1,0){9.5}}
\put(125,5){\vector(1,0){9.5}}
\put(45,15){\vector(1,0){9.5}}
\put(85,15){\vector(1,0){9.5}}
\put(125,15){\vector(1,0){9.5}}
\put(5,5){\vector(0,1){9.5}}
\put(45,5){\vector(0,1){9.5}}
\put(55,5){\vector(0,1){9.5}}
\put(95,5){\vector(0,1){9.5}}
\put(125,5){\vector(0,1){9.5}}
\put(135,5){\vector(0,1){9.5}}
%                                  arrows
\put(25,9){\line(1,0){9}}
\put(25,11){\line(1,0){9}}
\put(105,9){\line(1,0){9}}
\put(105,11){\line(1,0){9}}
\put(31,6){\line(1,1){4}}
\put(31,14){\line(1,-1){4}}
\put(111,6){\line(1,1){4}}
\put(111,14){\line(1,-1){4}}
\put(2,3){$u$}
\put(2,16){$v'$}
\put(16,3){$v$}
\put(56,16){$w$}
\put(82,16){$u$}
\put(96,16){$v$}
\put(96,3){$u'$}
\put(121.5,3){$w$}
\put(9,1.5){0}
\put(49,1.5){0}
\put(129,1.5){1}
\put(49,16){0}
\put(89,16){1}
\put(129,16){1}
\put(42.5,9){1}
\put(56,9){1}
\put(122.5,9){0}
\put(136,9){0}
  \end{picture}
 \end{center}

Using this axiom, one easily shows that if four vertices are connected by two
$i$-edges $e,e'$ and two $j$-edges $\tilde e,\tilde e'$ (forming a ``square''),
then $\ell_{j}(e)=\ell_{j}(e')\ne \ell_{i}(\tilde e)=\ell_{i}(\tilde e')$ (as
illustrated in the picture). Another important consequence of (A3) is that for
neighboring colors $i,j$, if $v$ is the critical vertex on an $i$-line w.r.t.
color $j$, then $v$ is also the critical vertex on the $j$-line passing $v$
w.r.t. color $i$, i.e. we can speak of common critical vertices for the pair
$\{i,j\}$.

The fourth axiom points out situations when, for neighboring $i,j$, the
operators $F_i,F_j$ and their inverse ones  ``remotely commute'' (they are
said to satisfy the ``Verma relation of degree 4'').

\begin{itemize}
\item[(A4)] Let $|i-j|=1$.
(i) If a vertex $u$ has outgoing edges with color $i$ and color $j$ and if each
edge is labeled 1 w.r.t. the other color, then $F_iF_j^2F_i(u)=F_jF_i^2F_j(u)$.
Symmetrically: (ii) if $v$ has incoming edges with color $i$ and color $j$ and
if both are labeled 0, then $F_i^{-1}(F_j^{-1})^2
F_i^{-1}(v)=F_j^{-1}(F_i^{-1})^2 F_j^{-1}(v)$. (See the picture.)
  \end{itemize}
 \begin{center}
  \unitlength=1mm
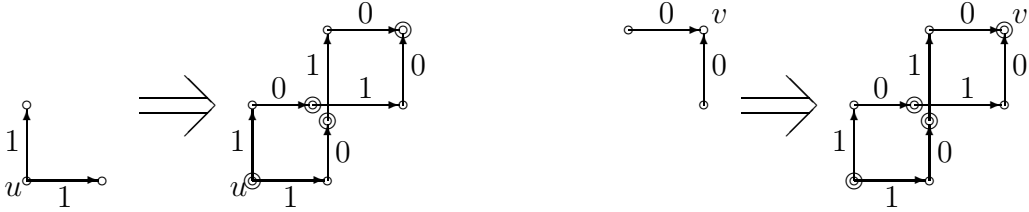
  \begin{picture}(147,30)
\put(5,5){\circle{1.0}}
\put(15,5){\circle{1.0}}
\put(5,15){\circle{1.0}}
\put(2,3){$u$}
\put(5,5){\vector(1,0){9.5}}
\put(5,5){\vector(0,1){9.5}}
\put(9,1.5){1}
\put(2,9){1}
%                                  arrow
\put(20,14){\line(1,0){9}}
\put(20,16){\line(1,0){9}}
\put(26,11){\line(1,1){4}}
\put(26,19){\line(1,-1){4}}
\put(35,5){\circle{1.0}}
\put(45,5){\circle{1.0}}
\put(35,15){\circle{1.0}}
\put(45,13){\circle{1.0}}
\put(43,15){\circle{1.0}}
\put(55,15){\circle{1.0}}
\put(45,25){\circle{1.0}}
\put(55,25){\circle{1.0}}
\put(35,5){\circle{2.0}}
\put(45,13){\circle{2.0}}
\put(43,15){\circle{2.0}}
\put(55,25){\circle{2.0}}
\put(35,5){\vector(1,0){9.5}}
\put(35,15){\vector(1,0){7.5}}
\put(43,15){\vector(1,0){11.5}}
\put(45,25){\vector(1,0){9.5}}
\put(35,5){\vector(0,1){9.5}}
\put(45,5){\vector(0,1){7.5}}
\put(45,13){\vector(0,1){11.5}}
\put(55,15){\vector(0,1){9.5}}
\put(32,3){$u$}
\put(39,1.5){1}
\put(32,9){1}
\put(46,7.5){0}
\put(37.5,16){0}
\put(49,16){1}
\put(42,19){1}
\put(56,19){0}
\put(49,26){0}
%                                right
\put(85,25){\circle{1.0}}
\put(95,15){\circle{1.0}}
\put(95,25){\circle{1.0}}
\put(85,25){\vector(1,0){9.5}}
\put(95,15){\vector(0,1){9.5}}
\put(96,26){$v$}
\put(89,26){0}
\put(96,19){0}
%                                  arrow
\put(100,14){\line(1,0){9}}
\put(100,16){\line(1,0){9}}
\put(106,11){\line(1,1){4}}
\put(106,19){\line(1,-1){4}}
\put(115,5){\circle{1.0}}
\put(125,5){\circle{1.0}}
\put(115,15){\circle{1.0}}
\put(125,13){\circle{1.0}}
\put(123,15){\circle{1.0}}
\put(135,15){\circle{1.0}}
\put(125,25){\circle{1.0}}
\put(135,25){\circle{1.0}}
\put(115,5){\circle{2.0}}
\put(125,13){\circle{2.0}}
\put(123,15){\circle{2.0}}
\put(135,25){\circle{2.0}}
\put(115,5){\vector(1,0){9.5}}
\put(115,15){\vector(1,0){7.5}}
\put(123,15){\vector(1,0){11.5}}
\put(125,25){\vector(1,0){9.5}}
\put(115,5){\vector(0,1){9.5}}
\put(125,5){\vector(0,1){7.5}}
\put(125,13){\vector(0,1){11.5}}
\put(135,15){\vector(0,1){9.5}}
\put(136,26){$v$}
\put(119,1.5){1}
\put(112,9){1}
\put(126,7.5){0}
\put(117.5,16){0}
\put(129,16){1}
\put(122,19){1}
\put(136,19){0}
\put(129,26){0}
  \end{picture}
 \end{center}

Again, one shows that the label w.r.t. $i,j$ of each of the eight involved
edges is determined uniquely, just as indicated in the above picture (where the
bigger circles indicate critical vertices).

The final axiom concerns non-neighboring colors.

 \begin{itemize}
 \item[(A5)]
Let $|i-j|\ge 2$. Then for any $F\in\{F_i,F_i^{-1}\}$ and $F'\in
\{F_j,F_j^{-1}\}$, the operators $F,F'$ commute at each vertex where both
act.
  \end{itemize}

This is equivalent to saying that each component of the two-colored subgraph
$(V(K),E_i\sqcup E_j)$ is the Cartesian product of an $i$-path $P$ and a
$j$-path $P'$, or that each subcrystal of $K$ with non-neighboring colors $i,j$
is an $A_1\times A_1$-{\em crystal}.

 \medskip
One shows that any $A_n$-crystal $K$ is finite and has exactly one
zero-indegree vertex $s_K$ and one zero-outdegree vertex $t_K$, called the {\em
source} and {\em sink} of $K$, respectively. Furthermore, the $A_n$-crystals
$K$ admit a nice parameterization: the lengths $h_1(s_K),\ldots,h_n(s_K)$ of
monochromatic paths determine $K$, and for each tuple $c=(c_1,\ldots,c_n)$ of
nonnegative integers, there exists a (unique) $A_n$-crystal $K$ such that
$c_i=h_i(s_K)$ for $i=1,\ldots,n$. (See~\cite{Stem} and~\cite{cross}.) We call
$c$ the {\em parameter} of $K$ and denote $K$ by $K(c)$.

%---------------SSEC 2.2
 \subsection{Crystals of types B and C} \label{ssec:typeBC}

These crystals are defined via the types of their two-colored subcrystals,
exhibited in axioms~(BC1)--(BC3). The difference between B- and C-crystals
concerns only specifications of axiom~(BC3) given in (BC4) and (BC4$'$). As
before, $K=(V(K),E_1\sqcup\ldots\sqcup E_n)$ is a connected $n$-colored graph.

 \begin{itemize}
 \item[(BC1)]
$K$ satisfies~(A1) and~(A5).
  \Xcomment{
i.e. each maximal monochromatic subgraph in it is an $A_1$-crystal (a path),
and for colors $i,j$ with $|i-j|\ge 2$, each component of $(V(K),E_i\sqcup
E_j)$ is an $A_1\times A_1$-crystal (the Cartesian product of two paths).
  }
  \end{itemize}

  \begin{itemize}
  \item[(BC2)]
For colors $i,j<n$ with $|i-j|=1$, each component of $(V(K),E_i\sqcup E_j)$ is
an $A_2$-crystal, i.e. it satisfies~(A2)--(A4).
  \end{itemize}

  \begin{itemize}
  \item[(BC3)]
Each component of $(V(K),E_{n-1}\sqcup E_n)$ is isomorphic to a $B_2$-crystal.
  \end{itemize}

There are several ways to define $B_2$-crystals. Based on Littlemann's path
model~\cite{Littl}, it is shown in~\cite{B2} that a $B_2$-crystal can be
equivalently defined in three ways: (i) via an explicit combinatorial
construction, (ii) via a graphical {\em worm model}, which represents each
vertex of the crystal as a certain pair of line-segments in a rectangle, and
(iii) via a list of 14 local or ``almost local'' axioms. Compared with the
$A_2$ case, this list is big enough and less convenient to handle practically.
In contrast, the worm model has a rather compact description, reviewed in
Section~\ref{sec:worm}, and we will appeal just to this model in our
examination of two-colored symmetric extracts from  corresponding A-crystals
(in Sections~\ref{sec:proofB3-B4} and~\ref{sec:A2n}).

For a $B_2$-crystal, with colors $i$ and $j$ say, the coefficients $m_{ij}$ and
$m_{ji}$ are different and take values $-1$ and $-2$ (where, as before,
$m_{pq}=(h_q(u)-t_q(u))-(h_q(v)-t_q(v))$ for an edge $(u,v)$ of color $p$). The
difference between B and C types is the following:
    \begin{itemize}
  \item[(BC4)]
~For $B_n$-crystals, $m_{n-1,n}=-2$ and $m_{n,n-1}=-1$.
  \end{itemize}
    \begin{itemize}
  \item[(BC4$'$)]
~For $C_n$-crystals, $m_{n-1,n}=-1$ and $m_{n,n-1}=-2$.
  \end{itemize}

The Cartan matrices for types A,\,B,\,C and $n=4$ are illustrated in the
picture where the coefficient in each empty cell is zero.
  \begin{center}
  \unitlength=1mm
  \begin{picture}(150,32)
    \begin{picture}(45,32)(0,0)
  \put(10,0){\line(0,1){32}}
  \put(18,0){\line(0,1){32}}
  \put(26,0){\line(0,1){32}}
  \put(34,0){\line(0,1){32}}
  \put(42,0){\line(0,1){32}}
  \put(10,0){\line(1,0){32}}
  \put(10,8){\line(1,0){32}}
  \put(10,16){\line(1,0){32}}
  \put(10,24){\line(1,0){32}}
  \put(10,32){\line(1,0){32}}
  \put(3,15){A:}
  \put(13,27){2}
  \put(21,19){2}
  \put(29,11){2}
  \put(37,3){2}
  \put(12,19){--1}
  \put(20,11){--1}
  \put(20,27){--1}
  \put(28,3){--1}
  \put(28,19){--1}
  \put(36,11){--1}
    \end{picture}
    \begin{picture}(45,32)(-5,0)
  \put(10,0){\line(0,1){32}}
  \put(18,0){\line(0,1){32}}
  \put(26,0){\line(0,1){32}}
  \put(34,0){\line(0,1){32}}
  \put(42,0){\line(0,1){32}}
  \put(10,0){\line(1,0){32}}
  \put(10,8){\line(1,0){32}}
  \put(10,16){\line(1,0){32}}
  \put(10,24){\line(1,0){32}}
  \put(10,32){\line(1,0){32}}
  \put(3,15){B:}
  \put(13,27){2}
  \put(21,19){2}
  \put(29,11){2}
  \put(37,3){2}
  \put(12,19){--1}
  \put(20,11){--1}
  \put(20,27){--1}
  \put(28,3){--1}
  \put(28,19){--1}
  \put(36,11){--2}
    \end{picture}
    \begin{picture}(45,32)(-10,0)
  \put(10,0){\line(0,1){32}}
  \put(18,0){\line(0,1){32}}
  \put(26,0){\line(0,1){32}}
  \put(34,0){\line(0,1){32}}
  \put(42,0){\line(0,1){32}}
  \put(10,0){\line(1,0){32}}
  \put(10,8){\line(1,0){32}}
  \put(10,16){\line(1,0){32}}
  \put(10,24){\line(1,0){32}}
  \put(10,32){\line(1,0){32}}
  \put(3,15){C:}
  \put(13,27){2}
  \put(21,19){2}
  \put(29,11){2}
  \put(37,3){2}
  \put(12,19){--1}
  \put(20,11){--1}
  \put(20,27){--1}
  \put(28,3){--2}
  \put(28,19){--1}
  \put(36,11){--1}
    \end{picture}
  \end{picture}
 \end{center}

 \medskip
Using arguments as in~\cite{cross,Stem} for A-crystals, one can show that any
B-crystal $K$ is finite, has exactly one source $s=s_K$ (and one sink), and is
determined by the lengths $h_1(s_K),\ldots,h_n(s_K)$. Also a $B_n$-crystal $K$
with $h(s_K)=c$ exists for any $c\in\Zset_+^n$, and similarly for
$C_n$-crystals (this is explained in~\cite{B2} for $n=2$, and follows from
reasonings in Sections~\ref{sec:Bn} and~\ref{sec:A2n} for $n>2$). This gives a
parametrization of B-crystals similar to that for A-crystals.

%-------------------- SSEC. 2.3
\subsection{The crossing model for $A_n$-crystals} \label{ssec:cross}

Following~\cite{cross}, the {\em crossing model} $\Mscr_n(c)$ generating the
$A_n$-crystal $K=K(c)$ with a parameter $c=(c_1,\ldots,c_n)\in \Zset_+^n$
consists of three ingredients:

(i) a directed graph $G_n=G=(V(G),E(G))$ depending on $n$, called the {\em
supporting graph} of the model;

(ii) a set $\Fscr=\Fscr(c)$ of {\em feasible} functions on $V(G)$;

(iii) a set $\Escr=\Escr(c)$ of transformations $f\mapsto f'$ of feasible
functions, called {\em moves} in the model.

 \smallskip
To explain the construction of the supporting graph $G$, we first introduce
another directed graph $\Gscr=\Gscr_n$ that we call the {\em proto-graph} of
$G$. Its node set consists of elements $V_i(j)$ for all $i,j\in\{1,\ldots,n\}$
such that $j\le i$. (We use the term ``node'' for vertices in the crossing
model, to avoid a possible mess between these and vertices of crystals.) Its
edges are all possible pairs of the form $(V_i(j),V_{i-1}(j))$ ({\em ascending}
edges) or $(V_i(j),V_{i+1}(j+1))$ ({\em descending} edges). We say that the
nodes $V_i(1),\ldots,V_i(i)$ form $i$-th {\em level} of $\Gscr$ and order them
as indicated (by increasing $j$). We visualize $\Gscr$ by drawing it on the
plane so that the nodes of the same level lie in a horizontal line, the
ascending edges point North-East, and the descending edges point South-East.
See the picture where $n=4$.
  \begin{center}
  \unitlength=1mm
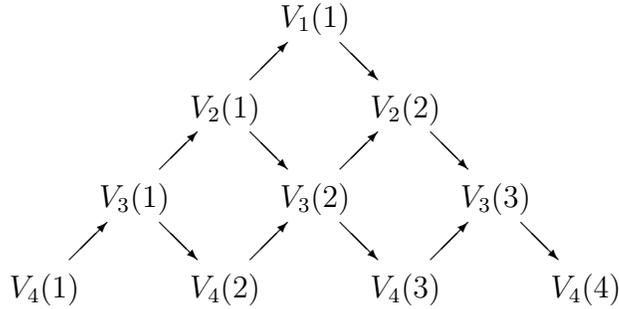
  \begin{picture}(80,40)
   \put(0,0){$V_4(1)$}
   \put(24,0){$V_4(2)$}
   \put(48,0){$V_4(3)$}
   \put(72,0){$V_4(4)$}
   \put(12,12){$V_3(1)$}
   \put(36,12){$V_3(2)$}
   \put(60,12){$V_3(3)$}
   \put(24,24){$V_2(1)$}
   \put(48,24){$V_2(2)$}
   \put(36,36){$V_1(1)$}
  \put(8,5){\vector(1,1){5}}
  \put(32,5){\vector(1,1){5}}
  \put(56,5){\vector(1,1){5}}
  \put(20,17){\vector(1,1){5}}
  \put(44,17){\vector(1,1){5}}
  \put(32,29){\vector(1,1){5}}
  \put(20,10){\vector(1,-1){5}}
  \put(44,10){\vector(1,-1){5}}
  \put(68,10){\vector(1,-1){5}}
  \put(32,22){\vector(1,-1){5}}
  \put(56,22){\vector(1,-1){5}}
  \put(44,34){\vector(1,-1){5}}
  \end{picture}
 \end{center}

The supporting graph $G$ is produced by replicating elements of $\Gscr$ as
follows. Each node $V_i(j)$ generates $n-i+1$ nodes of $G$, denoted as
$v_i^k(j)$ for $k=i-j+1,\ldots n-j+1$, which are ordered by increasing $k$ (and
accordingly follow from left to right in the visualization). We identify
$V_i(j)$ with the set of these nodes and call it a {\em multinode} of $G$. Each
edge of $\Gscr$ generates a set of edges of $G$ (a {\em multi-edge}) connecting
elements with equal upper indices. More precisely, $(V_i(j),V_{i-1}(j))$
produces $n-i+1$ ascending edges $(v_i^k(j),v_{i-1}^k(j))$ for $k=i-j+1,\ldots,
n-j+1$, and $(V_i(j),V_{i+1}(j+1))$ produces $n-i$ descending edges
$(v_i^k(j),v_{i+1}^k(j+1))$ for $k=i-j+1,\ldots,n-j$.

The resulting $G$ is the disjoint union of $n$ directed graphs
$G^1,\ldots,G^n$, where each $G^k$ contains all vertices of the form
$v_i^k(j)$. Also $G^k$ is isomorphic to the Cartesian product of two paths,
with the lengths $k-1$ and $n-k$. For example, for $n=4$, the graph $G$ is
viewed as
 \begin{center}
  \unitlength=1mm
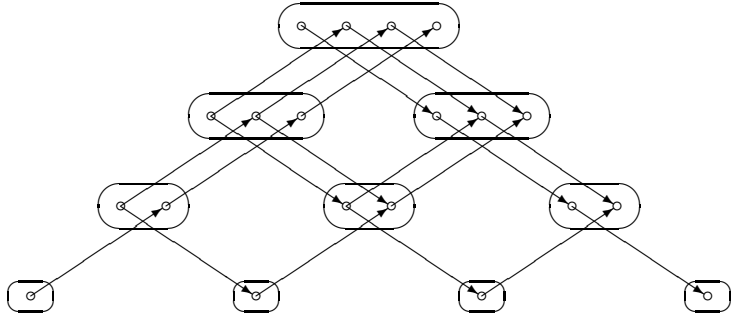
  \begin{picture}(90,39)
\put(0,0){\begin{picture}(54,36)%
\put(0,0){\circle{1.0}}
\put(18,12){\circle{1.0}}
\put(36,24){\circle{1.0}}
\put(54,36){\circle{1.0}}
\put(0,0){\vector(3,2){17.5}}
\put(18,12){\vector(3,2){17.5}}
\put(36,24){\vector(3,2){17.5}}
  \end{picture}}
\put(12,0){\begin{picture}(54,36)%
\put(0,12){\circle{1.0}}
\put(18,0){\circle{1.0}}
\put(18,24){\circle{1.0}}
\put(36,12){\circle{1.0}}
\put(36,36){\circle{1.0}}
\put(54,24){\circle{1.0}}
\put(0,12){\vector(3,2){17.5}}
\put(0,12){\vector(3,-2){17.5}}
\put(18,0){\vector(3,2){17.5}}
\put(18,24){\vector(3,2){17.5}}
\put(18,24){\vector(3,-2){17.5}}
\put(36,12){\vector(3,2){17.5}}
\put(36,36){\vector(3,-2){17.5}}
  \end{picture}}
\put(24,0){\begin{picture}(54,36)%
\put(0,24){\circle{1.0}}
\put(18,12){\circle{1.0}}
\put(18,36){\circle{1.0}}
\put(36,0){\circle{1.0}}
\put(36,24){\circle{1.0}}
\put(54,12){\circle{1.0}}
\put(0,24){\vector(3,2){17.5}}
\put(0,24){\vector(3,-2){17.5}}
\put(18,12){\vector(3,2){17.5}}
\put(18,12){\vector(3,-2){17.5}}
\put(18,36){\vector(3,-2){17.5}}
\put(36,0){\vector(3,2){17.5}}
\put(36,24){\vector(3,-2){17.5}}
  \end{picture}}
\put(36,0){\begin{picture}(54,36)%
\put(0,36){\circle{1.0}}
\put(18,24){\circle{1.0}}
\put(36,12){\circle{1.0}}
\put(54,0){\circle{1.0}}
\put(0,36){\vector(3,-2){17.5}}
\put(18,24){\vector(3,-2){17.5}}
\put(36,12){\vector(3,-2){17.5}}
  \end{picture}}
\put(45,36){\oval(24,6)}
\put(30,24){\oval(18,6)}
\put(60,24){\oval(18,6)}
\put(15,12){\oval(12,6)}
\put(45,12){\oval(12,6)}
\put(75,12){\oval(12,6)}
\put(0,0){\oval(6,4)}
\put(30,0){\oval(6,4)}
\put(60,0){\oval(6,4)}
\put(90,0){\oval(6,4)}
\end{picture}
 \end{center}
 \noindent
(where the multinodes are surrounded by ovals) and its components
$G^1,G^2,G^3,G^4$ are viewed as
 \begin{center}
  \unitlength=1mm
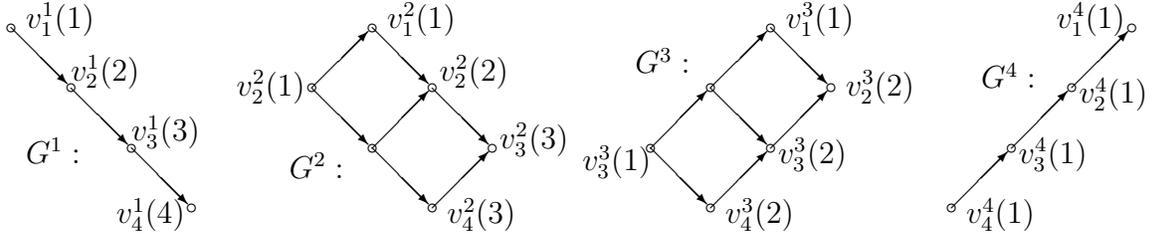
  \begin{picture}(150,27)
\put(0,24){\circle{1.0}}
\put(8,16){\circle{1.0}}
\put(16,8){\circle{1.0}}
\put(24,0){\circle{1.0}}
\put(0,24){\vector(1,-1){7.5}}
\put(8,16){\vector(1,-1){7.5}}
\put(16,8){\vector(1,-1){7.5}}
\put(2,6){$G^1:$}
\put(2,24){$v_1^1(1)$}
 \put(8,17){$v_2^1(2)$}
 \put(16,9){$v_3^1(3)$}
\put(14,-2){$v_4^1(4)$}
\put(40,16){\circle{1.0}}
\put(48,8){\circle{1.0}}
\put(48,24){\circle{1.0}}
\put(56,0){\circle{1.0}}
\put(56,16){\circle{1.0}}
\put(64,8){\circle{1.0}}
\put(40,16){\vector(1,1){7.5}}
\put(40,16){\vector(1,-1){7.5}}
\put(48,8){\vector(1,1){7.5}}
\put(48,8){\vector(1,-1){7.5}}
\put(48,24){\vector(1,-1){7.5}}
\put(56,0){\vector(1,1){7.5}}
\put(56,16){\vector(1,-1){7.5}}
\put(37,4){$G^2:$}
 \put(30,15){$v_2^2(1)$}
 \put(65,8){$v_3^2(3)$}
 \put(50,24){$v_1^2(1)$}
 \put(57,17){$v_2^2(2)$}
 \put(58,-2){$v_4^2(3)$}
\put(85,8){\circle{1.0}}
\put(93,0){\circle{1.0}}
\put(93,16){\circle{1.0}}
\put(101,8){\circle{1.0}}
\put(101,24){\circle{1.0}}
\put(109,16){\circle{1.0}}
\put(85,8){\vector(1,1){7.5}}
\put(85,8){\vector(1,-1){7.5}}
\put(93,0){\vector(1,1){7.5}}
\put(93,16){\vector(1,1){7.5}}
\put(93,16){\vector(1,-1){7.5}}
\put(101,8){\vector(1,1){7.5}}
\put(101,24){\vector(1,-1){7.5}}
\put(83,18){$G^3:$}
 \put(76,5){$v_3^3(1)$}
\put(95,-2){$v_4^3(2)$}
 \put(102,6){$v_3^3(2)$}
 \put(103,24){$v_1^3(1)$}
\put(111,15){$v_2^3(2)$}

\put(125,0){\circle{1.0}}
\put(133,8){\circle{1.0}}
\put(141,16){\circle{1.0}}
\put(149,24){\circle{1.0}}
\put(125,0){\vector(1,1){7.5}}
\put(133,8){\vector(1,1){7.5}}
\put(141,16){\vector(1,1){7.5}}
\put(129,16){$G^4:$}
 \put(127,-2){$v_4^4(1)$}
 \put(134,6){$v_3^4(1)$}
 \put(142,14){$v_2^4(1)$}
 \put(139,24){$v_1^4(1)$}
\end{picture}
 \end{center}

So each node $v=v_i^k(j)$ of $G$ has at most four incident edges, namely,
$(v_{i-1}^k(j-1),v)$, $(v_{i+1}^k(j),v)$, $(v,v_{i-1}^k(j))$,
$(v,v_{i+1}^k(j+1))$; we refer to them, when exist, as the NW-, SW-, NE-,
and SE-{\em edges}, and denote by $\eNW(v),\eSW(v),\eNE(v),\eSE(v)$,
respectively.

  \medskip
By a {\em feasible} function in the model (with a given $c$) we mean a function
$f:V(G)\to \Zset_+$ satisfying the following three conditions, where for an
edge $e=(u,v)$, $\;\partial f(e)$ denotes the increment $f(u)-f(v)$ of $f$ on
$e$, and $e$ is called {\em tight} for $f$, or $f$-{\em tight}, if $\partial
f(e)=0$:
  \begin{numitem1}
 \begin{itemize}
 \item[(i)] $f$ is {\em monotone} on the edges, in the sense that
$\partial f(e)\ge 0$ for all $e\in E(G)$;
 \item[(ii)] $0\le f(v)\le c_k$ for each $v\in V(G^k)$, $k=1,\ldots,n$;
 \item[(iii)] each multinode $V_i(j)$ contains a node $v$ with the following
property: the edge $\eSE(u)$ is tight for each node $u\in V_i(j)$
preceding $v$, and $\eSW(u')$ is tight for each node $u'\in V_i(j)$
succeeding $v$.
  \end{itemize}
 \label{eq:feas}
  \end{numitem1}

  \noindent
The {\em first} node $v=v_i^k(j)$ (i.e. with $k$ minimum) satisfying the
property in (iii) is called the {\em switch-node} of the multinode $V_i(j)$.
These nodes play an important role in our transformations of feasible functions
in the model.

To describe the rule of transforming $f\in\Fscr(c)$, we first extend each $G^k$
by adding extra nodes and edges (following~\cite{cross} and aiming to slightly
simplify the description). In the extended directed graph $\bar G^k$, the node
set consists of elements $v^k_i(j)$ for all $i=0,\ldots,n+1$ and $j=0,\ldots,n$
such that $j\le i$. The edge set of $\bar G^k$ consists of all possible pairs
of the form $(v^k_i(j),v^k_{i-1}(j))$ or $(v^k_i(j),v^k_{i+1}(j+1))$. Then all
$\bar G^k$ are isomorphic. The disjoint union of these $\bar G^k$ gives the
{\em extended supporting graph} $\bar G$.

Each feasible function on $V(G)$ is extended to the extra nodes $v=v^k_i(j)$ as
follows: $f(v):=c_k$ if there is a path from $v$ to a node of $G^k$, and
$f(v):=0$ otherwise (one may say that $v$ lies on the left of $G^k$ in the
former case, and on the right of $G^k$ in the latter case). In particular, each
edge $e$ of $\bar G$ not incident with a node of $G$ is tight, i.e. $\partial
f(e)=0$ (extending $\partial f$ to the extra edges). For a node $v=v_i^k(j)$
with $1\le j\le i\le n$, define the value $\eps(v)=\eps_f(v)$ by
  \begin{equation}    \label{eq:epsv}
\eps(v):= \partial f(\eNW(v))-\partial f(\eSE(u))\quad
  (=\partial f(\eSW(v))-\partial f(\eNE(u)),
   \end{equation}
where $u:=v_i^k(j-1)$. For a multinode $V_i(j)$ (and the given $f$),
define the numbers
  \begin{equation} \label{eq:xiij}
  \eps_i(j):=\sum(\eps(v)\colon v\in V_i(j))
  \end{equation}
and
  \begin{equation} \label{eq:barxiij}
  \tilde\eps_i(j):=\max\{0,\min\{\eps_i(p)+\eps_i(p+1)+\ldots+\eps_i(j)
                \colon 1\le p\le j\}\}.
  \end{equation}
We call $\eps(v)$, $\eps_i(j)$ and $\tilde\eps_i(j)$ the {\em slack} at a node
$v$, the {\em total slack} at a multinode $V_i(j)$ and the {\em reduced slack}
at $V_i(j)$, respectively. (We define the slacks $\eps,\tilde\eps$ in a
slightly different way than in~\cite{cross}, which however does not affect the
definitions of active multinodes and switch-nodes below.)

Now we are ready to define the transformations of $f$ (or the moves from $f$).
At most $n$ transformations $\phi_1,\ldots,\phi_n$ are possible. Each $\phi_i$
changes $f$ within level $i$ and is applicable when this level contains a
multinode $V_i(j')$ with $\tilde\eps_i(j')>0$. In this case we take the
multinode $V_i(j)$ such that
  \begin{equation} \label{eq:activej}
  \tilde\eps_i(j)>0\quad \mbox{and}\quad \tilde\eps_i(q)=0\;\;
                \mbox{for $q=j+1,\ldots,i$},
  \end{equation}
referring to it as the {\em active} multinode for the given $f$ and $i$,
and increase $f$ by 1 at the
{\em switch-node} in $V_i(j)$, preserving $f$ on the other nodes of $G$. It is
shown~\cite{cross} that the resulting function $\phi_i(f)$ is again
feasible.

So the model generates the $n$-colored directed graph $\Kscr(c)=(\Fscr,
\Escr_1\sqcup\ldots\sqcup\Escr_n)$, where each color class $\Escr_i$ is formed
by the edges $(f,\phi_i(f))$ for all feasible functions $f$ to which the
operator $\phi_i$ is applicable. This graph is just an $A_n$-crystal.

  \begin{theorem} {\rm \cite[Th.~5.1]{cross}}  \label{tm:cross}
For each $n$ and $c\in\Zset_+^n$, the $n$-colored graph $\Kscr(c)$ is exactly
the $A_n$-crystal $K(c)$.
  \end{theorem}

%---------------------------  SSEC. 2.4
 \subsection{Principal lattice and $(n-1)$-colored subcrystals of an $A_n$-crystal}
          \label{ssec:pr_lat}

Based on the crossing model, \cite{cross} reveals some important ingredients
and relations for an $A_n$-crystal $K=K(\bfc)$. One of them is the so-called
principal lattice, which is defined as follows.

Let $\bfa\in\Zset_+^n$ and $\bfa\le \bfc$. One easily checks that the function
on the vertices of the supporting graph $G$ that takes the constant value $a_k$
within each subgraph $G^k$ of $G$, $k=1,\ldots,n$, is feasible. We denote this
function and the vertex of $K$ corresponding to it by $f[\bfa]$ and
$\prv[\bfa]$, respectively, and call them {\em principal}. So the set of
principal vertices is bijective to the integer box
$\Bscr(\bfc):=\{a\in\Zset^n\colon \bfzero\le\bfa\le\bfc\}$; this set is called
the {\em principal lattice} of $K$ and denoted by $\Pi=\Pi(\bfc)$. When it is
not confusing, the term ``principal lattice'' may also be applied to
$\Bscr(\bfc)$.

The following properties of the principal lattice will be essentially used
later.

 \begin{prop} {\rm\cite[Expression~(6.4)]{cross}} \label{pr:fund_string}
Let $\bfa\in \Bscr(\bfc)$, $k\in\{1,\ldots,n\}$, and $\bfa':=\bfa+1_k$ (where
$1_k$ is $i$-th unit base vector in $\Rset^n$). The principal vertex
$\prv[\bfa']$ is obtained from $\prv[\bfa]$ by applying the operator string
   \begin{equation} \label{eq:string}
S_{n,k}:=w_{n,k,n-k+1}\cdots w_{n,k,2}w_{n,k,1},
   \end{equation}
where for $j=1,\ldots,n-k+1$, the substring $w_{n,k,j}$ is defined as
  $$
  w_{n,k,j}:=F_jF_{j+1}\cdots F_{j+k-1}.
  $$
When acting on $\Pi$, any two (applicable) strings $S_{n,k},S_{n,k'}$ commute.
In particular, any principal vertex $\prv[\bfa]$ is expressed via the source
$s_K=\prv[\bfzero]$ as
   \begin{equation} \label{eq:prin_str}
   \prv[\bfa]=S_{n,n}^{a_n}S_{n,n-1}^{a_{n-1}}\cdots S_{n,1}^{a_1}(s_K).
   \end{equation}
   \end{prop}

  \begin{prop} {\rm\cite[Prop.~6.1]{cross}} \label{pr:int_prlat}
For $c',c''\in Z_+^n$ with $c'\le c''\le c$, let $K(c'\colon\!c'')$ be the
subgraph of $K(c)$ formed by the vertices and edges contained in (directed)
paths from $\prv[c']$ to $\prv[c'']$ (the \emph{interval} of $K(c)$ from
$\prv[c']$ to $\prv[c'']$). Then $K(c'\colon\!c'')$ is isomorphic to the
$A_n$-crystal $K(c''-c')$, and the principal lattice of $K'$ consists of the
principal vertices $\prv[a]$ of $K(c)$ with $c'\le a\le c''$.
  \end{prop}

Let $\Kmn(c)$ denote the set of subcrystals with colors $1,\ldots,n-1$, and
$\Kmone$ the set of subcrystals with colors $2,\ldots,n$ in $K$ (recall that a
subcrystal is assumed to be connected and maximal).

  \begin{prop} {\rm\cite[Prop.~7.1]{cross}} \label{pr:prlat-subcryst}
Each subcrystal in $\Kmn$ (in $\Kmone$) contains precisely one principal
vertex. This gives a bijection between $\Kmn$ and $\Pi$ (resp., between
$\Kmone$ and $\Pi$).
  \end{prop}

We refer to the members of $\Kmn$ and $\Kmone$ as {\em upper} and {\em lower}
($(n-1)$-colored) subcrystals of $K$, respectively. For $\bfa\in \Bscr(\bfc)$,
the upper subcrystal containing the vertex $\prv[\bfa]$ is denoted by
$\Kup[\bfa]$. This subcrystal has its own principal lattice of dimension $n-1$,
which is denoted by $\Piup[\bfa]$. We say that the coordinate tuple $a$ is the
{\em locus} of $\Kup[\bfa]$ (and of $\Piup[\bfa]$) in $\Pi$. Analogously, for
$\bfb\in \Bscr(\bfc)$, the lower subcrystal containing $\prv[\bfb]$ is denoted
by $\Klow[\bfb]$, and its principal lattice by $\Pilow[\bfb]$; we say that
$\bfb$ is the locus of $\Klow[b]$ (and of $\Pilow[b]$) in $\Pi$. It turns out
that the parameters of upper and lower subcrystals can be expressed explicitly,
as follows.

  \begin{prop} {\rm\cite[Props.~7.2,7.3]{cross}} \label{pr:par-subcryst}
For $\bfa\in \Bscr(\bfc)$, the upper subcrystal $\Kup[\bfa]$ is isomorphic to
the $A_{n-1}$-crystal $K(\parup)$, where $\parup$ is the tuple
$(\parup_1,\ldots,\parup_{n-1})$ defined by
  \begin{equation} \label{eq:par_up}
  \parup_i:=c_i-a_i+a_{i+1}, \qquad i=1,\ldots,n-1.
  \end{equation}
The principal vertex $\prv[\bfa]$ is contained in the upper lattice
$\Piup[\bfa]$ and its coordinate $\heartup=(\heartup_1,\ldots,\heartup_{n-1})$
in $\Piup[\bfa]$ satisfies
  \begin{equation}  \label{eq:heart_up}
  \heartup_i=a_{i+1}, \qquad i=1,\ldots,n-1.
  \end{equation}

Symmetrically, for $\bfb\in \Bscr(\bfc)$, the lower subcrystal $\Klow[\bfb]$ is
isomorphic to the $A_{n-1}$-crystal $K(\parlow)$ with colors $2,\ldots,n$,
where $\parlow$ is defined by
  \begin{equation} \label{eq:par_low}
  \parlow_i:=c_i-b_i+b_{i-1}, \qquad i=2,\ldots,n.
  \end{equation}
The principal vertex $\prv[\bfb]$ is contained in the lower lattice $\Pilow[b]$
and its coordinate $\heartlow=(\heartlow_2,\ldots,\heartlow_n)$ in $\Pilow[b]$
satisfies
  \begin{equation}  \label{eq:heart_low}
  \heartlow_i=b_{i-1}, \qquad i=2,\ldots,n.
  \end{equation}
  \end{prop}

We call $\prv[\bfa]$ the {\em heart} of $\Kup[\bfa]$ w.r.t. $K$, and similarly
for lower subcrystals.

%---------------------------  SEC. 3

\section{Assembling an $A_n$-crystal} \label{sec:ass_A}

As mentioned in the Introduction, the structure of an $A_n$-crystal $K=K(\bfc)$
will be described in a recursive manner. The idea is as follows. We know that
$K$ contains $|\Pi|=(c_1+1)\times\ldots\times (c_n+1)$ upper subcrystals (with
colors $1,\ldots,n-1$) and $|\Pi|$ lower subcrystals (with colors
$2,\ldots,n$). Moreover, the parameters of these subcrystals are expressed
explicitly by~\refeq{par_up} and~\refeq{par_low}. So we may assume by recursion
that the set $\Kscr^{(-n)}$ of upper subcrystals and the set $\Kscr^{(-1)}$ of
lower subcrystals are available (already constructed). In order to assemble
$K$, it suffices to characterize, in appropriate terms, the intersection
$\Kup[\bfa]\cap \Klow[\bfb]$ for all pairs $\bfa,\bfb\in \Bscr(\bfc)$ (the
intersection may either be empty, or consist of one or more $(n-2)$-colored
subcrystals with colors $2,\ldots,n-1$ in $K$). We give an appropriate
characterization in Theorem~\ref{tm:mainA} below.

To state it, we need additional terminology and notation. Consider a subcrystal
$\Kup[\bfa]$, and let $\parup,\heartup$ be defined as
in~\refeq{par_up},\refeq{heart_up}. For $\bfp=(p_1,\ldots,p_{n-1})\in
\Bscr(\parup)$, the vertex in the upper lattice $\Piup[\bfa]$ having the
coordinate $\bfp$ is denoted by $\vup[\bfa,\bfp]$. We call the vector
$\Delta:=\bfp-\heartup$ the {\em deviation} of $\vup[\bfa,\bfp]$ from the heart
$\prv[\bfa]$ in $\Piup[\bfa]$, and will use the alternative notation
$\vup[a|\Delta]$ for this vertex. In particular,
$\prv[a]=\vup[\bfa,\heartup]=\vup[\bfa|\,0]$.

Similarly, for a lower subcrystal $\Klow[\bfb]$, let $\parlow,\heartlow$ be as
in~\refeq{par_low},\refeq{heart_low}. For $\bfq=(q_2,\ldots,q_n)\in
\Bscr(\parlow)$, the vertex with the coordinate $\bfq$ in $\Pilow[\bfb]$ is
denoted by $\vlow[\bfa,\bfq]$. Its deviation is $\nabla:=\bfq-\heartlow$, and
we may alternatively denote this vertex by $\vlow[b|\nabla]$.

We call an $(n-2)$-colored subcrystals with colors $2,\ldots,n-1$ in $K$ a {\em
middle subcrystals} and denote the set of these by $\Kscr^{(-1,-n)}$. Each
middle crystal $\Kmid$ is a {\em lower} subcrystal of some upper subcrystal
$K'=\Kup[a]$ of $K$. By Proposition~\ref{pr:prlat-subcryst} applied to $K'$,
~$\Kmid$ has a unique vertex $\vup[a|\Delta]$ in the lattice $\Piup[\bfa]$. So
each $\Kmid$ can be encoded by a pair $(a,\Delta)$ formed by a locus $a\in
\Bscr(c)$ and a deviation $\Delta$ in $\Piup[a]$. At the same time, $\Kmid$ is
an \emph{upper} subcrystal of some lower subcrystal $\Klow[b]$ of $K$ and has a
unique vertex $\vlow[b|\nabla]$ in $\Pilow[b]$. Therefore, the members of
$\Kscr^{(-1,-n)}$ determine a bijection
  $$
  \zeta:(a,\Delta)\mapsto (b,\nabla)
  $$
between all pairs $(a,\Delta)$ concerning upper subcrystals and all pairs
$(b,\nabla)$ concerning lower subcrystals.

The map $\zeta$ is expressed explicitly in the following theorem. Here for a
tuple $\rho=(\rho_i\colon i\in I)$, we denote by $\rho^+$ ($\rho^-$) the tuple
with the entries $\rho_i^+:=\max\{0,\rho_i\}$ (resp.
$\rho_i^-:=\min\{0,\rho_i\}$), $i\in I$.

 \begin{theorem} \label{tm:mainA}
Let $a\in \Bscr(c)$ and let $\Delta=(\Delta_1,\ldots,\Delta_{n-1})$ be a
deviation in $\Piup[a]$. Let $(b,\nabla)=\zeta(a,\Delta)$. Then $b$ satisfies

  \begin{equation} \label{eq:bi}
  b_i=a_i+\Delta^+_i+\Delta^-_{i-1}, \qquad i=1,\ldots,n,
  \end{equation}
letting $\Delta_0=\Delta_n:=0$, and $\nabla$ satisfies
  \begin{equation} \label{eq:nablai}
  \nabla_i=-\Delta_{i-1}, \qquad i=2,\ldots,n.
 \end{equation}
 \end{theorem}

A proof of this theorem will be given in the next section.

Based on Theorem~\ref{tm:mainA}, the crystal $K(c)$ is assembled as follows. By
recursion we assume that all upper and lower subcrystals are already
constructed. We also assume that for each upper subcrystal $\Kup[a]=K(\parup)$,
its principal lattice is distinguished by use of the corresponding injective
map $\sigma:\Bscr(\parup)\to V(K(\parup))$, and similarly for the lower
subcrystals. We delete the edges with color 1 in each $K(\parup)$ and extract
the components of the resulting graphs, forming a list $\tilde\Kscr$ of all
middle subcrystals of $K(c)$. Each $\Kmid \in\tilde\Kscr$ is encoded by a
corresponding pair $(a,\Delta)$, where $\bfa\in\Bscr(\bfc)$ and the deviation
$\Delta$ in $\Piup[a]$ is determined by use of $\sigma$ as above. Acting
similarly for the lower subcrystals $K(\parlow)$ (by deleting the edges with
color $n$ there), we obtain an isomorphic list of middle subcrystals, each of
which being encoded by a corresponding pair $(b,\nabla)$, where $b\in\Bscr(c)$
and $\nabla$ is a deviation in $\Pilow(b)$. Relations~\refeq{bi} and
\refeq{nablai} indicate how to identify each member of the first list with its
counterpart in the second one. Now restoring the deleted edges with colors 1
and $n$, we obtain the desired crystal $K(c)$. The corresponding map
$\Bscr(c)\to V(K(c))$ is constructed easily (e.g., by use of operator strings
as in Proposition~\ref{pr:fund_string}).

  \smallskip
We conclude this section with several remarks.

 \smallskip
 \noindent
{\bf Remark 1.} ~\refeq{bi} and~\refeq{nablai} lead to the following expression
of $a$ via $b$ and $\nabla$:
  \begin{equation} \label{eq:ai}
  a_i=b_i+\nabla^+_i+\nabla^-_{i+1}, \qquad i=1,\ldots,n,
  \end{equation}
letting $\nabla_1=\nabla_{n+1}:=0$.  This will be used, in particular, in the
Appendix.

 \medskip
 \noindent
{\bf Remark 2.} For each $a\in \Bscr(c)$ and each vertex $v=\vup[a,p]$ in the
upper lattice $\Piup[a]$, one can express the parameter $\parmid= (\parmid_2,
\ldots,\parmid_{n-1})$ of the middle subcrystal $\Kmid$ containing $v$, as well
as the coordinate $\heartmid=(\heartmid_2,\ldots, \heartmid_{n-1})$ of its
heart w.r.t. $\Kup[a]$ in the principal lattice of $\Kmid$. Indeed, since
$\Kmid$ is a lower subcrystal of $\Kup[a]$, one can apply relations as
in~\refeq{par_low},\refeq{heart_low}. Denoting the parameter of $\Kup[a]$ by
$\parup$ and the coordinate of its heart in $\Piup[a]$ by $\heartup$, letting
$\Delta:=p-\heartup$, and using~\refeq{par_up},\refeq{heart_up}, we have:
  \begin{gather}
  \parmid_i=\parup_i-p_i+p_{i-1}=(c_i-a_i+a_{i+1})-(a_{i+1}+\Delta_i)
     +(a_i+\Delta_{i-1}) \label{eq:par_mid} \\
    =c_i-\Delta_i+\Delta_{i-1}, \qquad i=2,\ldots,n-1; \nonumber
    \end{gather}
  \begin{equation} \label{eq:heart_mid}
  \heartmid_i=p_{i-1}=\heartup_{i-1}+\Delta_{i-1}=a_i+\Delta_{i-1},\qquad
          i=2,\ldots,n-1.
   \end{equation}

Symmetrically, if $\Kmid$ is contained in $\Klow[b]$ and is related to a
deviation $\nabla$ in $\Pilow[b]$, then
   \begin{eqnarray}
   \parmid_i &=& c_i-\nabla_i+\nabla_{i+1},
                     \label{eq:par_mid2} \\
   \bmid_i &=& b_i+\nabla_{i+1},  \qquad i=2,\ldots,n-1,
                      \label{eq:heart_mid2}
  \end{eqnarray}
where $\bmid$ is the coordinate of the heart of $\Kmid$ w.r.t. $\Klow[b]$ in
the principal lattice of $\Kmid$ (note that $\bmid$ may differ from
$\heartmid$). We will use formulas~\refeq{par_mid}--\refeq{heart_mid2} in
subsequent sections.

 \medskip
 \noindent
{\bf Remark 3.} A straightforward implementation of the above recursive method
of constructing $K=K(c)$ takes $O(2^{q(n)}N)$ time and space, where $q(n)$ is a
polynomial in $n$ and $N$ is the number of vertices of $K$. Here the factor
$2^{q(n)}$ appears because the total number of vertices in the upper and lower
subcrystals is $2N$ (implying that there appear $4N$ vertices in total on the
previous step of the recursion, and so on). Therefore, such an implementation
has polynomial complexity of the size of the output for each fixed $n$, but not
in general. However, many intermediate subcrystals arising during the recursive
process are repeated, and we can use this fact to improve the implementation.
More precisely, the colors occurring in each intermediate subcrystal in the
process form an interval of the ordered set $(1,\ldots,n)$. We call a
subcrystal of this sort a {\em color-interval subcrystal}, or a {\em
CI-subcrystal}, of $K$. In fact, {\em every} CI-subcrystal of $K$ appears in
the process. Since the number of intervals is $\frac{n(n+1)}{2}$ and the
CI-subcrystals concerning one and the same interval are pairwise disjoint, the
total number of vertices of all CI-subcrystals of $K$ is $O(n^2 N)$. It is not
difficult to implement the recursive process in such a way that each
CI-subcrystal $K'$ is explicitly constructed only once. (For this purpose, one
can use pointers from the vertices of $K'$ to its source $s_{K'}$ and
characterize $K'$ by its color-interval and $s_{K'}$.) As a result, we obtain
the following
  \begin{prop} \label{pr:efficient}
Let $n\in \Zset_+$ and $c\in\Zset_+^n$. The $A_n$-crystal $K(c)$ and all its
CI-subcrystals can be constructed in $O(q'(n)|V(K(c))|)$ time and space, where
$q'(n)$ is a polynomial in $n$.
  \end{prop}

 \noindent
{\bf Remark 4.} Relation~\refeq{bi} shows that the intersection of $\Kup[a]$
and $\Klow[b]$ may consist of many middle subcrystals. Indeed, if $\Delta_i>0$
and $\Delta_{i-1}<0$ for some $i$, then $b$ does not change by simultaneously
decreasing $\Delta_i$ by 1 and increasing $\Delta_{i-1}$ by 1. The number of
common middle subcrystals of $\Kup[a]$ and $\Klow[b]$ for arbitrary $a,b\in
\Bscr(c)$ can be expressed by an explicit piecewise linear formula,
using~\refeq{bi} and the box constraints $-a_{i+1}\le\Delta_i\le c_i-a_i$,
$i=1,\ldots,n-1$, on the deviations $\Delta$ in $\Piup[a]$ (which follow
from~\refeq{par_up},\refeq{heart_up}).

%---------------------------  SEC. 4
 \section{Proof of Theorem~\ref{tm:mainA}} \label{sec:proof}

Let $a,\Delta,b,\nabla$ be as in the hypotheses of this theorem. First we prove
relation~\refeq{nablai} in the assumption that~\refeq{bi} is valid.

 \medskip
 \noindent
{\bf Proof of~\refeq{nablai}.} The middle subcrystal $\Kmid$ determined by
$(a,\Delta)$ is the same as the one determined by $(b,\nabla)$. The parameter
$\parmid$ of $\Kmid$ is expressed simultaneously by~\refeq{par_mid} and
by~\refeq{par_mid2}. Then $c_i-\Delta_i+\Delta_{i-1}=c_i-\nabla_i+\nabla_{i+1}$
for $i=2,\ldots,n-1$. Therefore,
   \begin{equation}  \label{eq:Delta-nabla}
  \Delta_1+\nabla_2=\Delta_2+\nabla_3=\ldots=
         \Delta_{n-1}+\nabla_n=:\alpha.
   \end{equation}

In order to obtain~\refeq{nablai}, one has to show that $\alpha=0$. We argue as
follows. Renumber the colors $1,\ldots,n$ as $n,\ldots,1$, respectively; this
yields the crystal $\hat K= K(\hat c)$ symmetric to $K(c)$. Then $\Klow[b]$
turns into the upper subcrystal $\hat K^\uparrow[\hat b]$ of $\hat K$, where
$(\hat b_1,\ldots,\hat b_n)=(b_n,\ldots,b_1)$. Also the deviation $\nabla$ in
$\Pilow[b]$ turns into the deviation $\hat\nabla=(\hat\nabla_1,
\ldots,\hat\nabla_{n-1})=(\nabla_n,\ldots,\nabla_2)$ in the principal lattice
of $\hat K^\uparrow[\hat b]$. Applying relations as in~\refeq{bi} to $(\hat
b,\hat\nabla)$, we have
  \begin{equation} \label{eq:bar_nabla}
  \hat a_i=\hat b_i+\hat\nabla^+_i+\hat\nabla^-_{i-1}=
    b_{n-i+1}+\nabla^+_{n-i+1}+\nabla^-_{n-i+2}, \qquad i=1,\ldots,n,
   \end{equation}
where $\hat a_i:=a_{n-i+1}$ and $\hat\nabla^+_n:=\hat\nabla^-_0:=0$.
On the other hand,~\refeq{bi} for $(a,\Delta)$ gives
  \begin{equation} \label{eq:nab}
  \hat b_i=b_{n-i+1}=a_{n-i+1}+\Delta^+_{n-i+1}+\Delta^-_{n-i},
            \qquad i=1,\ldots,n.
  \end{equation}

Relations~\refeq{bar_nabla} and~\refeq{nab} imply
  $$
  a_{n-i+1}=(a_{n-i+1}+\Delta^+_{n-i+1}+\Delta^-_{n-i})
    +\nabla^+_{n-i+1}+\nabla^-_{n-i+2},
    $$
whence
   $$
 \Delta^+_{n-i+1}+\Delta^-_{n-i}
    +\nabla^+_{n-i+1}+\nabla^-_{n-i+2}=0, \qquad i=1,\ldots,n.
  $$
Adding up the latter equalities, we obtain
  $$
  (\Delta_1+\ldots+\Delta_{n-1})+(\nabla_2+\ldots+\nabla_n)=0.
  $$
This and~\refeq{Delta-nabla} imply $(n-1)\alpha=0$. Hence $\alpha=0$, as
required. \hfill \qed

 \medskip
 \noindent
{\bf Proof of~\refeq{bi}.} This proof is rather technical and essentially uses
the crossing model.

For a feasible function $f\in\Fscr(c)$ and its corresponding vertex $v$ in
$K=K(c)$, we may denote $v$ as $v_f$, and $f$ as $f_v$. The following
observation from the crossing model will be of use:
  \begin{numitem1}
if a vertex $v\in V(K)$ belongs to $\Kup[a]$ and to $\Klow[b]$, then the
tuples $a$ and $b$ are expressed via the values of $f=f_v$ in levels $n$
and 1 as follows:

$a_k=f(v_n^k(n-k+1))$\; and\; $b_k=f(v_1^k(1))$\; for $k=1,\ldots,n$.
  \label{eq:fn1}
  \end{numitem1}
Indeed, the principal vertex $\prv[a]$ is reachable from $v$ by applying
operators $F_i$ or $F_i^{-1}$ with $i\ne n$. The corresponding moves in the
crossing model do not change $f$ within level $n$. Similarly, $\prv[b]$ is
reachable from $v$ by applying operators $F_i$ or $F_i^{-1}$ with $i\ne 1$, and
the corresponding moves in the crossing model do not change $f$ within level 1.
Also the relations in~\refeq{fn1} are valid for the principal function
$f=f_{\prv[a]}$.

Next we introduce special functions on the node set $V(G)$ of the supporting
graph $G=G_n$. Consider a component $G^k=(V^k,E^k)$ of $G$. It is a rectangular
grid (rotated by 45$^\circ$ in the visualization of $G$), and its vertex set is
  $$
  V^k=\{v_i^k(j)\colon j=1,\ldots,n-k+1,\;\; i=j,\ldots,j+k-1\}.
  $$
To represent it in a more convenient form, introduce the variable $m:=i-j+1$
and rename $v_i^k(j)$ as $u^k_{i-j+1}(j)$, or as $u_{i-j+1}(j)$ (when no
confusion can arise). Then
  $$
  V^k=\{u_m(j)\colon j=1,\ldots,n-k+1,\;\; m=1,\ldots,k\},
  $$
the (descending) SE-edges in $G^k$ are of the form $(u_m(j),u_m(j+1))$, and the
(ascending) NE-edges are of the form $(u_m(j),u_{m-1}(j))$. We specify the
following subsets of $V^k$:

(i) the {\em SW-side} $P=P^k:=\{u_k(1),\ldots,u_k(n-k+1)\}$;

(ii) the {\em right rectangle} $R=R^k:=\{u_m(j)\colon 1\le m\le k-1,\;
1\le j\le n-k+1\}$;

(iii) the {\em left rectangle} $L=L^k:=\{u_m(j)\colon 1\le m\le k,\;
1\le j\le n-k\}$.

\noindent Denote the characteristic functions (in $\Rset^{V^k}$) of $P, R, L$
as $\pi^k$, $\rho^k,\lambda^k$, respectively.

Return to $a\in\Bscr(c)$ and a deviation $\Delta$ in $\Piup[a]$. Associate to
$(a,\Delta)$ the functions
  \begin{equation} \label{eq:fkaD}
  f^k_{a,\Delta}:=a_k\pi^k+(a_k+\Delta^-_{k-1})\rho^k+\Delta^+_k\lambda^k
  \end{equation}
on $V^k$ for $k=1,\ldots,n$ (see Fig.~\ref{fig:partition}),  and their direct
sum
  $$
  f_{a,\Delta}:= f^1_{a,\Delta}\oplus\ldots\oplus f^n_{a,\Delta}
  $$
(the function on $V(G)$ whose restriction to each $V^k$ is $f^k_{a,\Delta}$).
  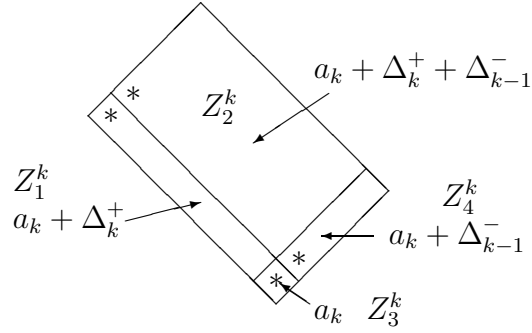
\begin{figure}[hbt]                  % Fig.1
  \begin{center}
  \unitlength=1mm
  \begin{picture}(70,40)(0,6)
  \put(10,30){\line(1,1){15}}
  \put(32,8){\line(1,1){15}}
  \put(35,5){\line(1,1){15}}
  \put(10,30){\line(1,-1){25}}
  \put(13,33){\line(1,-1){25}}
  \put(25,45){\line(1,-1){25}}
  \put(40,35){$a_k+\Delta^+_k+\Delta^-_{k-1}$}
  \put(12,29){$\ast$}
  \put(15,32){$\ast$}
  \put(34,7){$\ast$}
  \put(37,10){$\ast$}
  \put(0,20){$Z_1^k$}
  \put(0,15){$a_k+\Delta^+_k$}
  \put(40,3){$a_k$}
  \put(47,3){$Z^k_3$}
  \put(50,13){$a_k+\Delta^-_{k-1}$}
  \put(57,18){$Z^k_4$}
  \put(25,30){$Z^k_2$}
  \put(42,33){\vector(-3,-2){10}}
  \put(15,16){\vector(4,1){10}}
  \put(39.5,5){\vector(-3,2){4}}
  \put(48,14){\vector(-1,0){7}}
 \end{picture}
 \end{center}
  \caption{The partition of $V^k$.}
  \label{fig:partition}
  \end{figure}

In view of~\refeq{fn1}, $f=f_{a,\Delta}$ takes the values in levels $n$ and 1
as required in~\refeq{bi} (with $k$ in place of $i$), namely,
$f(v^k_n(n-k+1))=a_k$ and $f(v^k_1(1))=a_k+\Delta^+_k+\Delta^-_{k-1}$ for
$k=1,\ldots,n$. Therefore, to obtain~\refeq{bi} it suffices to show the
following
  \begin{lemma} \label{lm:faDelta}
(i) The function $f=f_{a,\Delta}$ is feasible. (ii) The vertex $v_f$ is
the vertex of $\Piup[a]$ having the deviation $\Delta$.
  \end{lemma}

 \noindent
{\bf Proof}~ First we prove statement~(i). Let $k\in\{1,\ldots,n\}$. We
partition $V^k$ into four subsets (rectangular {\em pieces}):
  $$
  Z^k_1:=P^k\cap L^k;\quad
  Z^k_2:=L^k\setminus P^k;\quad
  Z^k_3:=\{u_k^k(n-k+1)\};\quad
  Z^k_4:=R^k\setminus L^k
  $$
\noindent (where $Z^k_2=Z^k_4=\emptyset$ when $k=1$, and
$Z^k_1=Z^k_2=\emptyset$ when $k=n$). By~\refeq{fkaD},
  \begin{numitem1}
$f$ takes a constant value within each piece $Z^k_q$, namely: $a_k+\Delta^+_k$
on $Z^k_1$; \\
$a_k+\Delta^+_k+\Delta^-_{k-1}$ on $Z^k_2$; $a_k$ on $Z^k_3$; and
$a_k+\Delta^-_{k-1}$ on $Z^k_4$
  \label{eq:valuesZ}
  \end{numitem1}
(as illustrated in Fig.~\ref{fig:partition}). Also each edge of $G^k$
connecting different pieces goes either from $Z^k_1$ to $Z^k_2\cup Z^k_3$ or
from $Z^k_2\cup Z^k_3$ to $Z^k_4$. This and~\refeq{valuesZ} imply that
$\partial f(e)\ge 0$ for each edge $e\in E^k$, whence $f$
satisfies~\refeq{feas}(i).

The deviation $\Delta$ is bounded as $-\heartup\le\Delta\le \parup-\heartup$,
where $\parup$ is the parameter of the subcrystal $\Kup[a]$ and $\heartup$ is
the coordinate of its heart $\prv[a]$ in $\Piup[a]$. Expressions~\refeq{par_up}
and~\refeq{heart_up} for $\parup$ and $\heartup$ give
  \begin{equation} \label{eq:boundsDelta}
  -a_{k+1}\le\Delta_k\le c_k-a_k\quad \mbox{and}\quad
  -a_k\le\Delta_{k-1}\le c_{k-1}-a_{k-1}.
  \end{equation}
The inequalities $\Delta_k\le c_k-a_k$ and $a_k\le c_k$ imply
$a_k+\Delta^+_k\le c_k$. And the inequalities $-a_k\le\Delta_{k-1}$ and $a_k\ge
0$ imply $a_k+\Delta^-_{k-1}\ge 0$. Then, in view of~\refeq{valuesZ}, we obtain
$0\le f(v)\le c_k$ for each node $v$ of $G^k$, yielding~\refeq{feas}(ii).

To verify the switch condition~\refeq{feas}(iii), consider a multinode
$V_i(j)$ with $i<n$. It consists of $n-i+1$ nodes $v^k_i(j)$, where
$i-j+1\le k\le n-j+1$.

Let $i\le n-2$. Suppose that there is a node $v=v_i^k(j)$ whose SW-edge
$e=(u,v)$ exists and is not $f$-tight. This is possible only if $u\in Z^k_1$
and $v\in Z^k_2$. Then $k$ is determined as $k=i-j+2$, i.e. $v$ is the second
node in $V_i(j)$. We observe that: (a) for the first node $v_i^{k-1}(j)$ of
$V_i(j)$, both ends of its SE-edge $e'$ belong to the piece $Z_1^{k-1}$, whence
$e'$ is $f$-tight; and (b) for any node $v_i^{k'}(j)$ with $k'>k$ in $V_i(j)$,
both ends of its SW-edge $e''$ belong either to $Z_2^{k'}$ or to $Z_4^{k'}$.
Therefore, the node $v$ satisfies the requirement in~\refeq{feas}(iii) for
$V_i(j)$.

Now let $i=n-1$. Then $V_i(j)$ consists of two nodes $v=v^{n-j}_{n-1}(j)$ and
$v'=v^{n-j+1}_{n-1}(j)$. Put $k:=n-j$. Then the edge $e=\eSE(v)$ goes from
$Z^k_1$ to $Z^k_3=\{u^k_k(n-k+1)\}$, and the edge $e'=\eSW(v')$ goes from
$Z^{k+1}_3=\{u^{k+1}_{k+1}(n-k)\}$ to $Z^{k+1}_4$. By~\refeq{valuesZ}, we have
$\partial f(e)=(a_k+\Delta_k^+)-a_k =\Delta^+_k$ and $\partial
f(e')=a_{k+1}-(a_{k+1}+\Delta_k^-)=-\Delta^-_k$. Since at least one of
$\Delta^+_k,\Delta^-_k$ is zero, we conclude that at least one of $e,e'$ is
tight. So~\refeq{feas}(iii) is valid again.

 \medskip
Next we start proving statement~(ii) in the lemma. We use induction on the
value
  $$
  \eta(\Delta):=\Delta_1+\ldots+\Delta_{n-1}.
  $$

In view of~\refeq{boundsDelta}, $\eta(\Delta)\ge -a_2-\ldots-a_n$. Let this
hold with equality. Then $\Delta_k=-a_{k+1}\le 0$ for $k=1,\ldots,n-1$, and
by~\refeq{valuesZ}, $f=f_{a,\Delta}$ takes the following values within each
$V^k$: $f(v)=a_k$ if $v\in P^k$, and $f(v)=0$ if $v\in V^k-P^k$. This $f$ is
the minimal feasible function whose values in level $n$ match $a$, and
therefore, $v_f$ is the source of $\Kup[a]$. Then $v_f$ is the minimal vertex
$\prv[a,\bfzero]$ in $\Piup[a]$, and its deviation in $\Piup[a]$ is just
$\Delta$, as required. This gives the base of the induction.

Now consider an arbitrary $\Delta$ satisfying~\refeq{boundsDelta}. Let $k$ be
such that $\Delta_k<c_k-a_k$ (if any) and define $\Delta'_k:=\Delta_k+1$ and
$\Delta'_i:=\Delta_i$ for $i\ne k$. Then $\eta(\Delta)<\eta(\Delta')$. We
assume by induction that claim~(ii) is valid for $f_{a,\Delta}$, and our aim is
to show validity of~(ii) for $f_{a,\Delta'}$.

In what follows $f$ stands for the initial function $f_{a,\Delta}$.

Let $v'$ be the vertex with the deviation $\Delta'$ in $\Piup[a]$. Both $v_f$
and $v'$ are principal vertices of the subcrystal $\Kup[a]$ and the coordinate
of $v'$ in $\Piup[a]$ is obtained from the one of $v_f$ by increasing its
$k$-th entry by 1. According to Proposition~\ref{pr:fund_string} (with $n$
replaced by $n-1$), $v'$ is obtained from $v_f$ by applying the operator string
  $$
  S_{n-1,k}=w_{n-1,k,n-k}\cdots w_{n-1,k,1},
  $$
where $w_{n-1,k,j}=F_j\cdots F_{j+k-1}$ (cf.~\refeq{string}). In light of this,
we have to show that
  \begin{numitem1}
when (the sequence of moves corresponding to) $S_{n-1,k}$ is applied to $f$,
the resulting feasible function is exactly $f_{a,\Delta'}$.
  \label{eq:applW}
  \end{numitem1}

For convenience $m$-th term, from left to right, in the substring $w_{n-1,k,j}$
(i.e. the operator $F_{j+m-1}$) will be denoted by $\phi(j,m)$, $m=1,\ldots,k$.
So $w_{n-1,k,j}=\phi(j,1)\phi(j,2)\ldots \phi(j,k)$.

We distinguish between two cases: $\Delta\ge 0$ and $\Delta<0$.

 \medskip
{\bf Case 1}: $\Delta_k\ge 0$. An essential fact is that the number $k(n-k)$ of
operators in $S_{n-1,k}$ is equal to the number of nodes in the left rectangle
$L^k$ of $G^k$, and moreover, the substrings in $S_{n-1,k}$ one-to-one
correspond to the NE-paths in $L^k$. More precisely, the level of each node
$u_m(j)$ of $L^k$ is equal to the ``color'' of the operator $\phi(j,m)$
(indeed, $u_m(j)=v^k_{j+m-1}(j)$ and $\phi(j,m)=F_{j+m-1}$).

Let $f^{j,m}$ denote the current function on $V(G)$ just before the application
of $\phi(j,m)$ (when the process starts with $f=f_{a,\Delta}$). Also we write
$(j',m')\prec (j,m)$ if $j'<j$ or if $j'=j$ and $m'>m$. We assert that
  \begin{numitem1}
for each $m$, the application of $\phi(j,m)$ to $f^{j,m}$ increases the value
at the node $u^k_m(j)$ by 1; equivalently:
$f^{j,m}(u^k_{m'}(j'))=f(u^k_{m'}(j'))+1$ if $(j',m')\prec (j,m)$, and
$f^{j,m}(v)=f(v)$ for the other nodes $v$ of $G$,
  \label{eq:phijm}
 \end{numitem1}
whence~\refeq{applW} will immediately follow.

In order to show~\refeq{phijm}, we first examine tight edges and the slacks
$\eps(v)$ of the nodes $v$ in levels $<n$ for the initial function $f$. One can
observe from~\refeq{valuesZ} that
  \begin{numitem1}
for $k'=1,\ldots,n$, each node $v$ of the subgraph $G^{k'}$ has at least one
entering edge (i.e. $\eSW(v)$ or $\eNW(v)$) which is $f$-tight, except,
possibly, for the nodes $v^{k'}_{k'}(1)$, $v^{k'}_{k'-1}(1)$,
$v^{k'}_{n}(n-k'+1)$, $v^{k'}_{n-1}(n-k'+1)$ (indicated by stars in
Fig.~\ref{fig:partition}).
  \label{eq:tightenter}
  \end{numitem1}

 \noindent
{\bf Claim.} {\em For $k'=1,\ldots,n$ and a node $v$ of \,$G^{k'}$ in a level
$<n$,

(a) if $v\ne v^{k'}_{k'}(1),v^{k'}_{k'-1}(1)$, then $\eps(v)=0$;

(b) if $v= v^{k'}_{k'}(1)$, then $\eps(v)=c_{k'}-a_{k'}-\Delta^+_{k'}\ge
0$;

(c) if $v=v^{k'}_{k'-1}(1)$, then $\eps(v)=-\Delta^-_{k'}\ge 0$. }

 \medskip
 \noindent
{\bf Proof of Claim.} Let $v=v^{k'}_i(j)$ and $i<n$. By~\refeq{epsv}, the slack
$\eps(v)$ is equal to $f(w)+f(z)-f(u)-f(v)$, where $w:=v^{k'}_{i-1}(j-1)$,
$z:=v^{k'}_{i+1}(j)$, $u:=v^{k'}_{i}(j-1)$ (these vertices belong to the
extended graph $\bar G^{k'}$). We consider possible cases and
use~\refeq{valuesZ}.

(i) If $w,z,u$ are in $G^{k'}$, then $\partial f(w,v)=\partial f(u,z)$.

(ii) If both $v,w$ are in the piece $Z^{k'}_1$ of $G^{k'}$, then $f(w)=f(v)$
and $f(u)=f(z)=c_{k'}$.

(iii) If $j=1$ and $i\le k'-2$, then $f(v)=f(z)$ and $f(u)=f(w)=c_{k'}$. So in
these cases we have $\eps(v)=0$, yielding~(a).

(iv) Let $v=v^{k'}_{k'}(1)$. Then $f(v)=a_{k'}+\Delta^+_{k'}$ and
$f(u)=f(w)=f(z)=c_{k'}$. This gives $\eps(v)=c_{k'}-a_{k'}-\Delta^+_{k'}$,
yielding~(b).

(v) Let $v=v^{k'}_{k'-1}(1)$. Then $f(v)=a_{k'}+\Delta^+_{k'}+\Delta^-_{k'-1}$,
$f(z)=a_{k'}+\Delta^+_{k'}$ and $f(u)=f(w)=c_{k'}$. This gives
$\eps(v)=-\Delta^-_{k'-1}$, yielding~(c). \hfill \qed

 \medskip
This Claim and the relations $\Delta^-_k=0$ and $\Delta_k<c_k-a_k$ enable us to
estimate the total slacks $\eps_i(j)$ for $f$ at the multinodes $V_i(j)$ with
$i<n$:
  \begin{numitem1}
(i) the edge $\eSW(v^{k+1}_k(1))$ is $f$-tight, $\eps(v^k_k(1))>0$, and
$\eps(v)=0$ for the other \\
 \hphantom{(i)} nodes $v$ in $V_k(1)$; so $\eps_k(1)>0$;

(ii) if $i\ne k,n$, then $\eps(v_i^i(1)),\eps(v_i^{i+1}(1))\ge 0$ and
$\eps(v)=0$ for the other nodes \\
 \hphantom{(ii)} $v$ in $V_i(1)$; so $\eps_i(1)\ge 0$;

(iii) if $i\ne n$ and $j>1$, then $\eps(v)=0$ for all nodes $v$ in $V_i(j)$; so
$\eps_i(j)=0$.
  \label{eq:incase1}
  \end{numitem1}

Now we are ready to prove~\refeq{phijm}. When dealing with a current function
$f^{j,m}$ and seeking for the node at level $j+m-1$ where the operator
$\phi(j,m)$ should act to increase $f^{j,m}$, we can immediately exclude from
consideration any node $v$ that has a tight entering edge (since acting the
operator at $v$ would cause violation of the monotonicity
condition~\refeq{feas}(i)).

Due to~\refeq{tightenter} and~\refeq{incase1}(i), for the initial function
$f=f^{1,k}$, there is only one node in level $k$ that has no tight entering
edge, namely, $v^k_k(1)$. So, at the first step of the process, the first
operator $\phi(1,k)$ of $S_{n-1,k}$ acts just at $v^k_k(1)$, as required
in~\refeq{phijm}.

Next consider a step with $f':=f^{j,m}$ and $\phi(j,m)$ for $(j,m)\ne (1,k)$,
assuming that~\refeq{phijm} is valid at the previous step.

(A) Let $j=1$ (and $m<k$). For $v:=v^k_m(1)$ and $z:=v^k_{m+1}(1)$, we have
$f'(v)=f(v)\le f(z)=f'(z)-1$. So the unique edge $e=(z,v)$ entering $v$ is not
$f'$-tight. By~\refeq{tightenter}, there are at most two other nodes in level
$m$ that may have no tight entering edges for $f$ (and therefore, for $f'$),
namely, $v^m_m(1)$ and $v^{m+1}_m(1)$. Then $\phi(1,m)$ must act at $v$, as
required in~\refeq{phijm} (since the non-tightness of the SW-edge $e$ of $v$
implies that none of the nodes $v^{m'}_m(1)$ in $V_m(1)$ preceding $v$ (i.e.
with $m'<k$) can be the switch-node).

 \smallskip
(B) Let $j>1$. Comparing $f'$ with $f$ in the node $v:=u^k_m(j)=v^k_{j+m-1}(j)$
and its adjacent nodes, we observe that $v$ has no $f'$-tight entering edge and
that $\eps_{f'}(v)>0$. Also for any other node $v'$ in level $j+m-1$, one can
see that if $v'$ has a tight entering edge for $f$, then so does for $f'$, and
that $\eps_f(v')\ge \eps_{f'}(v')\ge 0$. Using this,
properties~\refeq{tightenter},\,\refeq{incase1}(iii), and
condition~\refeq{activej}, one can conclude that the total and reduced slacks
for $f'$ at the multinode $V':=V_{j+m-1}(j)$ are positive, that $V'$ is the
active multinode for $f'$ in level $j+m-1$, and that $\phi(j,m)$ can be applied
only at $v$, yielding~\refeq{phijm} again.

Thus, \refeq{applW} is valid in Case~1.

 \medskip
{\bf Case 2}: $\Delta_k<0$. We assert that in this case the string $S_{n-1,k}$
acts within the right rectangle $R^{k+1}$ of the subgraph $G^{k+1}$ (note that
$R^{k+1}$ is of size $k\times(n-k)$). More precisely,
  \begin{numitem1}
each operator $\phi(j,m)$ modifies the current function by increasing its
value at the node $u^{k+1}_m(j)$ by 1.
  \label{eq:phijm2}
  \end{numitem1}
Then for the resulting function $\tilde f$ in the process, its restriction to
$V^{k+1}$ is
  $$
  a_{k+1}\pi^{k+1}+(a_{k+1}+\Delta^-_k+1)\rho^{k+1}+
         \Delta^+_{k+1}\lambda^{k+1}
  $$
(cf.~\refeq{fkaD}). Therefore, $\tilde f=f_{a,\Delta'}$ (in view of
$(\Delta')^-_k=\Delta^-_k+1$), yielding~\refeq{applW}.

To show~\refeq{phijm2}, we argue as in the previous case and
use~\refeq{tightenter} and the above Claim. Since $\Delta_k<0$, part~(i)
in~\refeq{incase1} for the initial function $f$ is modified as:
  \begin{numitem1}
for $j=1,\ldots,n-k$, the SW-edge of each node
$u^{k+1}_k(j)=v^{k+1}_{j+k-1}(j)$ is not $f$-tight, $\eps(v^{k+1}_k(1))>0$,
~$\eps(v^k_k(1))\ge 0$, and $\eps(v)=0$ for the other nodes $v$ in $V_k(1)$; so
$\eps_k(1)>0$,
  \label{eq:incase2}
  \end{numitem1}
while properties~(ii) and~(iii) preserve.

By~\refeq{tightenter} and~\refeq{incase2}, there are only two nodes in level
$k$ that have no $f$-tight entering edges, namely, $v^k_k(1)$ and
$v^{k+1}_k(1)$. Also $e=\eSW(v^{k+1}_k(1))$ is not tight. So, at the first
step, $\phi(1,k)$ must act at $v^{k+1}_k(1)$, as required in~\refeq{phijm2}
(since the non-tightness of $e$ implies that the node $v^k_k(1)$ preceding
$v^{k+1}_k(1)$ cannot be the switch-node in $V_k(1)$).

The fact that $\phi(1,m)$ with $m<k$ acts at $v^{k+1}_m(1)$ is shown by arguing
as in~(A) above. And for $j>1$, to show that $\phi(j,m)=F_{j+m-1}$ acts at
$u^{k+1}_m(j)=v^{k+1}_{j+m-1}(j)$, we argue as in~(B) above. Here, when $m=k$,
we also use the fact that the edge $\eSW(u^{k+1}_{k}(j))$ is not $f$-tight
(by~\refeq{incase2}), whence both edges entering $u^{k+1}_k(j)$ are not tight
for the current function. So~\refeq{phijm2} is always valid.

Thus, we have the desired property~\refeq{applW} in both cases~1 and~2, and
statement~(ii) in the lemma follows. \hfill \qed\qed

This completes the proof of relation~\refeq{bi} in Theorem~\ref{tm:mainA}.

%---------------SEC 5
 \section{Illustrations} \label{sec:illustr}

In this section we give two illustrations to the above assembling construction
for A-crystals. The first one specifies the interrelation between upper and
lower subcrystals in an arbitrary $A_2$-crystal, which can be compared with the
explicit construction (the so-called ``sail model'') for $A_2$-crystals
in~\cite{A2}. The second one visualizes the subcrystals structure for one
instance of $A_3$-crystals, namely, $K(1,1,1)$.

%---------------SSEC 5.1
\subsection{$A_2$-crystals} \label{ssec:A2}

The subcrystals structure becomes simpler when we deal with an $A_2$-crystal
$K=K(c_1,c_2)$. In this case the roles of upper, lower, and middle subcrystals
are played by 1-paths, 2-paths, and vertices of $K$, respectively, where by an
$i$-path we mean a maximal path of color $i$.

Consider an upper subcrystal in $K$. This is a 1-path $P=(v_0,v_1,\ldots,v_p)$
containing exactly one principal vertex $\prv[a]$ of $K$ (the heart of $P$);
here $v_i$ stands for $i$-th vertex in $P$, $a=(a_1,a_2)\in\Zset_+^2$ and $a\le
c$. Let $\prv[a]=v_h$. Formulas~\refeq{par_up} and~\refeq{heart_up} give
  \begin{equation} \label{eq:1path}
  |P|=p=c_1-a_1+a_2 \qquad\mbox{and} \qquad h=a_2.
  \end{equation}
Fix a vertex $v=v_i$ of $P$. It belongs to some 2-path (lower subcrystal)
$Q=(u_1,u_2,\ldots,u_q)$. Let $v=u_j$ and let $\prv[b]=u_{\bar h}$ be the
principal vertex of $K$ occurring in $Q$ (the heart of $Q$). The vertex $v$
forms a middle subcrystal of $K$; its deviations from the heart of $P$ and from
the heart of $Q$ are equal to $i-h=:\delta$ and $j-\bar h=:\bar\delta$,
respectively. By~\refeq{nablai} in Theorem~\ref{tm:mainA}, we have
$\bar\delta=-\delta$. Then we can compute the coordinates $b$ by use
of~\refeq{bi} and, further, apply~\refeq{par_low} and~\refeq{heart_low} to
compute the length of $Q$ and the locus of its heart. This gives:
  \begin{numitem1} \label{eq:2path}
 \begin{itemize}
\item[(i)] if $\delta\ge 0$ (i.e. $a_2\le i\le c_1-a_1+a_2$), then
$b_1=a_1+\delta=a_1+i-a_2$, ~$b_2=a_2$, ~$|Q|=c_2-b_2+b_1=c_2-2a_2+a_1+i$, and
$|Q|-\bar h=|Q|-b_1=c_2-a_2$;
  \item[(ii)] if $\delta\le 0$ (i.e. $0\le i\le a_2$), then $b_1=a_1$,
~$b_2=a_2+\delta=a_2+(i-a_2)=i$, ~$|Q|=c_2-b_2+b_1=c_2-i+a_1$, and $\bar
h=b_1=a_1$.
  \end{itemize}
  \end{numitem1}

Using~\refeq{1path} and~\refeq{2path}, one can enumerate the sets of 1-paths
and 2-paths and properly intersect corresponding pairs, obtaining the
$A_2$-crystal $K(c)$. It is rather routine to check that the resulting graph
coincides with the one generated by the \emph{sail model} from~\cite{A2}. Next
we outline that construction (it will be used in Section~\ref{ssec:C1}).

Given $c\in\Zset_+^2$, the $A_2$-crystal $K(c)$ is produced from two particular
two-colored graphs $R$ and $L$, called the {\em right sail} of size $c_1$ and
the {\em left sail} of size $c_2$, respectively. The vertices of $R$ correspond
to the vectors $(i,j)\in\Zset^2$ such that $0\le j\le i\le c_1$, and the
vertices of $L$ to the vectors $(i,j)\in\Zset^2$ such that $0\le i\le j\le
c_2$. In both $R,L$, the edges of color 1 are all possible pairs of the form
$((i,j),(i+1,j))$, and the edges of color 2 are all possible pairs of the form
$((i,j),(i,j+1))$. (Observe that both $R$ and $L$ satisfy axioms (A1)-(A4), $R$
is isomorphic to $K(c_1,0)$, ~$L$ is isomorphic to $K(0,c_2)$, and their
critical vertices are the ``diagonal vertices'' $(i,i)$.)

In order to produce $K(c)$, take $c_2$ disjoint copies $R_1,\ldots,R_{c_2}$ of
$R$ and $c_1$ disjoint copies $L_1,\ldots,L_{c_1}$ of $L$, referring to $R_j$
as $j$-th right sail, and to $L_i$ as $i$-th left sail. Let $D(R_j)$ and
$D(L_i)$ denote the sets of diagonal vertices in $R_j$ and $L_i$, respectively.
For all $i=1,\ldots,c_1$ and $j=1,\ldots,c_2$, we identify the diagonal
vertices $(i,i)\in D(R_j)$ and $(j,j)\in D(L_i)$. The resulting graph is just
the desired $K(c)$. The edge colors of $K(c)$ are inherited from $L$ and $R$.
One checks that $K(c)$ has $(c_1+1)\times (c_2+1)$ critical vertices; they are
exactly those induced by the diagonal vertices of the sails. The principal
lattice of $K(c)$ is just constituted by the critical vertices.

The case $(c_1,c_2)=(1,2)$ is drawn in the picture; here the critical
(principal) vertices are indicated by circles, 1-edges by horizontal arrows,
and 2-edges by vertical arrows.
 \begin{center}
  \unitlength=1mm
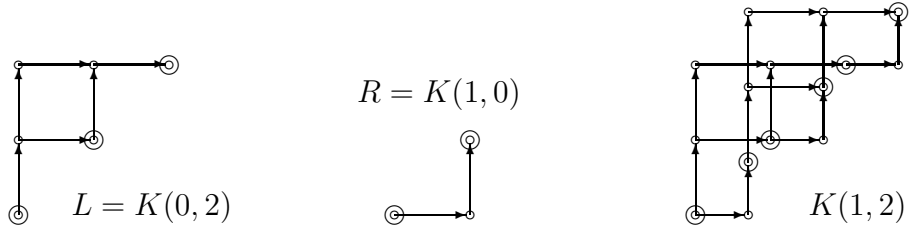
  \begin{picture}(125,35)
\put(5,5){\circle{1.0}}           % left
\put(5,15){\circle{1.0}} \put(5,25){\circle{1.0}} \put(15,15){\circle{1.0}}
\put(15,25){\circle{1.0}} \put(25,25){\circle{1.0}} \put(5,5){\circle{2.5}}
\put(15,15){\circle{2.5}} \put(25,25){\circle{2.5}}
\put(5,15){\vector(1,0){9.5}} \put(5,25){\vector(1,0){9.5}}
\put(15,25){\vector(1,0){9.5}} \put(5,5){\vector(0,1){9.5}}
\put(5,15){\vector(0,1){9.5}} \put(15,15){\vector(0,1){9.5}}
\put(12,5){$L=K(0,2)$}
 %                                   % middle
\put(55,5){\circle{1.0}} \put(65,5){\circle{1.0}} \put(65,15){\circle{1.0}}
\put(55,5){\circle{2.5}} \put(65,15){\circle{2.5}}
\put(55,5){\vector(1,0){9.5}} \put(65,5){\vector(0,1){9.5}}
 \put(50,20){$R=K(1,0)$}
 %                                  % right
\put(95,5){\circle{1.0}} \put(95,15){\circle{1.0}} \put(95,25){\circle{1.0}}
\put(105,15){\circle{1.0}} \put(105,25){\circle{1.0}}
\put(115,25){\circle{1.0}} \put(95,5){\circle{2.5}} \put(105,15){\circle{2.5}}
\put(115,25){\circle{2.5}} \put(95,15){\vector(1,0){9.5}}
\put(95,25){\vector(1,0){9.5}} \put(105,25){\vector(1,0){9.5}}
\put(95,5){\vector(0,1){9.5}} \put(95,15){\vector(0,1){9.5}}
\put(105,15){\vector(0,1){9.5}}
\put(102,12){\circle{1.0}} \put(102,22){\circle{1.0}}
\put(102,32){\circle{1.0}} \put(112,22){\circle{1.0}}
\put(112,32){\circle{1.0}} \put(122,32){\circle{1.0}}
\put(102,12){\circle{2.5}} \put(112,22){\circle{2.5}}
\put(122,32){\circle{2.5}} \put(102,22){\vector(1,0){9.5}}
\put(102,32){\vector(1,0){9.5}} \put(112,32){\vector(1,0){9.5}}
\put(102,12){\vector(0,1){9.5}} \put(102,22){\vector(0,1){9.5}}
\put(112,22){\vector(0,1){9.5}}
\put(102,5){\circle{1.0}} \put(95,5){\vector(1,0){6.5}}
\put(102,5){\vector(0,1){6.5}} \put(112,15){\circle{1.0}}
\put(105,15){\vector(1,0){6.5}} \put(112,15){\vector(0,1){6.5}}
\put(122,25){\circle{1.0}} \put(115,25){\vector(1,0){6.5}}
\put(122,25){\vector(0,1){6.5}}
 \put(110,5){$K(1,2)$}
  \end{picture}
 \end{center}

In particular, the sail model shows that the numbers of edges of each color in
an $A_2$-crystal are the same. This implies a similar property for any
$A_n$-crystal (and moreover, for crystals of classical simply-laced types).

%---------------SSEC 5.2
\subsection{$A_3$-crystal $K(1,1,1)$} \label{ssec:K111}

Next we illustrate the $A_3$-crystal $K=K(1,1,1)$. It has 64 vertices and 102
edges, and drawing it in full would take too much space; for this reason, we
describe it in fragments, namely, by demonstrating all of its upper and lower
subcrystals. We abbreviate notation $\prv[(i,j,k)]$ for principal vertices to
$(i,j,k)$ for short. So the principal lattice consists of eight vertices
$(0,0,0),\ldots,(1,1,1)$, as drawn in the picture (where the arrows indicate
moves by principal operator strings $S_{3,k}$ as in~\refeq{string}):
  \begin{center}
  \unitlength=1mm
  \begin{picture}(80,42)
    \put(5,5){\begin{picture}(80,35)
  \put(0,0){\circle*{1.2}}
  \put(30,0){\circle*{1.2}}
  \put(0,20){\circle*{1.2}}
  \put(10,10){\circle*{1.2}}
  \put(40,10){\circle*{1.2}}
  \put(10,30){\circle*{1.2}}
  \put(30,20){\circle*{1.2}}
  \put(40,30){\circle*{1.2}}
  \put(0,0){\vector(1,0){29}}
  \put(10,10){\vector(1,0){29}}
  \put(0,20){\vector(1,0){29}}
  \put(10,30){\vector(1,0){29}}
  \put(0,0){\vector(0,1){19.3}}
  \put(30,0){\vector(0,1){19.3}}
  \put(10,10){\vector(0,1){19.3}}
  \put(40,10){\vector(0,1){19.3}}
  \put(0,0){\vector(1,1){9.5}}
  \put(30,0){\vector(1,1){9.5}}
  \put(0,20){\vector(1,1){9.5}}
  \put(30,20){\vector(1,1){9.5}}
  \put(-14,-4){$s$=(0,0,0)}
  \put(28,-4){(1,0,0)}
  \put(-12,21){(0,1,0)}
  \put(11,11){(0,0,1)}
  \put(19,21){(1,1,0)}
  \put(41,8){(1,0,1)}
  \put(4,31.5){(0,1,1)}
  \put(33,31.5){(1,1,1)=$t$}
  \put(10,-4){$S_{3,1}$}
  \put(-7,9){$S_{3,2}$}
  \put(6,3){$S_{3,3}$}
  \put(60,20){$S_{3,1}=F_3F_2F_1$}
  \put(60,13){$S_{3,2}=F_2F_3F_1F_2$}
  \put(60,6){$S_{3,3}=F_1F_2F_3$}
   \end{picture}}
%%%
 \end{picture}
 \end{center}

Thus, $K$ has eight upper subcrystals $\Kup[i,j,k]$ and eight lower subcrystals
$\Klow[i,j,k]$ (writing $K^\bullet[i,j,k]$ for $K^\bullet[(i,j,k)]$); they are
drawn in Figures~\ref{fig:upper111} and~\ref{fig:lower111}. Here the directions
of edges of colors 1,2,3 are as indicated in the upper left corner. In each
subcrystal we indicate its critical vertices by black circles, and the unique
principal vertex of $K$ occurring in it (the heart) by a big white circle. $K$
has 30 middle subcrystals (paths of color 2), which are labeled as $A,\ldots,Z,
\Gamma,\Delta,\Phi,\Psi$ (note that $B,F,G,N,P,T,V,\Phi$ consist of single
vertices).

  \begin{figure}[hbt]                  % Fig.2
  \begin{center}
  \unitlength=1mm
  \begin{picture}(145,137)
                        %              (upper left)
   \put(3,105){\begin{picture}(20,30)
  \put(0,10){\vector(4,-1){12}}
  \put(0,10){\vector(0,1){12}}
  \put(13,6){1}
  \put(1,20){2}
   \end{picture}}
%%
                         %              (upper 000)
   \put(30,103){\begin{picture}(40,30)
  \put(0,0){\circle*{1.5}}
  \put(12,-3){\circle{1.2}}
  \put(0,12){\circle{1.2}}
  \put(24,6){\circle{1.2}}
  \put(8,10){\circle*{1.5}}
  \put(12,21){\circle{1.2}}
  \put(12,13){\circle*{1.5}}
  \put(24,18){\circle*{1.5}}
  \put(0,0){\circle{2.5}}
  \put(0,0){\line(4,-1){12}}
  \put(0,12){\line(4,-1){24}}
  \put(12,21){\line(4,-1){12}}
  \put(0,0){\line(0,1){12}}
  \put(12,-3){\line(0,1){24}}
  \put(24,6){\line(0,1){12}}
  \put(-10,-4.5){(0,0,0)}
  \put(-3.5,5.5){$A$}
  \put(6,6){$B$}
  \put(13,2){$C$}
  \put(25,12){$D$}
  \put(-4,25){$\Kup[0,0,0]\simeq K(1,1)$}
   \end{picture}}
%%%
                         %              (upper 100)
   \put(85,113){\begin{picture}(30,25)
  \put(0,0){\circle*{1.5}}
  \put(0,12){\circle{1.2}}
  \put(12,9){\circle*{1.5}}
  \put(0,0){\circle{2.5}}
  \put(0,12){\line(4,-1){12}}
  \put(0,0){\line(0,1){12}}
  \put(-10,-4.5){(1,0,0)}
  \put(-3.5,5.5){$E$}
  \put(11,5){$F$}
  \put(-4,16){$\Kup[1,0,0]\simeq K(0,1)$}
   \end{picture}}
%%%
                         %              (upper 010)
   \put(55,75){\begin{picture}(40,30)
  \put(0,0){\circle*{1.5}}
  \put(12,-3){\circle{1.2}}
  \put(24,-6){\circle{1.2}}
  \put(24,6){\circle{1.2}}
  \put(12,9){\circle*{1.5}}
  \put(24,18){\circle*{1.5}}
  \put(12,9){\circle{2.5}}
  \put(0,0){\line(4,-1){24}}
  \put(12,9){\line(4,-1){12}}
  \put(12,-3){\line(0,1){12}}
  \put(24,-6){\line(0,1){24}}
  \put(6,11){(0,1,0)}
  \put(-1.5,1.5){$G$}
  \put(13,2){$H$}
  \put(25,10){$I$}
  \put(-4,22){$\Kup[0,1,0]\simeq K(2,0)$}
   \end{picture}}
%%%

                         %              (upper 001)
   \put(105,75){\begin{picture}(45,45)
  \put(0,0){\circle*{1.5}}
  \put(0,12){\circle{1.2}}
  \put(0,24){\circle{1.2}}
  \put(12,9){\circle*{1.5}}
  \put(12,21){\circle{1.2}}
  \put(24,18){\circle*{1.5}}
  \put(12,9){\circle{2.5}}
  \put(0,12){\line(4,-1){21}}
  \put(0,24){\line(4,-1){33}}
  \put(0,0){\line(0,1){24}}
  \put(12,9){\line(0,1){12}}
  \put(9,7){\circle*{1.5}}
  \put(9,19){\circle{1.2}}
  \put(9,31){\circle{1.2}}
  \put(21,16){\circle*{1.5}}
  \put(21,28){\circle{1.2}}
  \put(33,25){\circle*{1.5}}
  \put(9,19){\line(4,-1){12}}
  \put(9,31){\line(4,-1){24}}
  \put(9,-2){\line(0,1){33}}
  \put(21,7){\line(0,1){21}}
  \put(9,-2){\circle{1.2}}
  \put(21,7){\circle{1.2}}
  \put(33,16){\circle{1.2}}
  \put(0,0){\line(4,-1){9}}
  \put(33,16){\line(0,1){9}}
  \put(14,-1){(0,0,1)}
 \put(17,3){\vector(-1,1){4}}
  \put(-3,16){$J$}
  \put(5.5,25){$K$}
  \put(12.5,13){$L$}
  \put(17,22){$M$}
  \put(25,18){$N$}
  \put(34,20){$O$}
  \put(0,35){$\Kup[0,0,1]\simeq K(1,2)$}
   \end{picture}}
%%%
                         %              (upper 110)
   \put(10,45){\begin{picture}(30,25)
  \put(0,0){\circle*{1.5}}
  \put(12,-3){\circle{1.2}}
  \put(12,9){\circle*{1.5}}
  \put(12,9){\circle{2.5}}
  \put(0,0){\line(4,-1){12}}
  \put(12,-3){\line(0,1){12}}
  \put(-1,9){(1,1,0)}
  \put(-1,-4){$P$}
  \put(13,3){$Q$}
  \put(-4,17){$\Kup[1,1,0]\simeq K(1,0)$}
   \end{picture}}
%%%

                         %              (upper 101)
   \put(18,0){\begin{picture}(40,30)
  \put(0,0){\circle*{1.5}}
  \put(0,12){\circle{1.2}}
  \put(0,24){\circle{1.2}}
  \put(12,9){\circle*{1.5}}
  \put(12,21){\circle{1.2}}
  \put(24,18){\circle*{1.5}}
  \put(12,9){\circle{2.5}}
  \put(0,12){\line(4,-1){12}}
  \put(0,24){\line(4,-1){24}}
  \put(0,0){\line(0,1){24}}
  \put(12,9){\line(0,1){12}}
  \put(9,4){(1,0,1)}
  \put(-3.5,16){$R$}
  \put(9,14){$S$}
  \put(25,18){$T$}
  \put(-4,27){$\Kup[1,0,1]\simeq K(0,2)$}
   \end{picture}}
%%%

                         %              (upper 011)
   \put(62,25){\begin{picture}(50,50)
  \put(0,0){\circle*{1.5}}
  \put(12,-3){\circle{1.2}}
  \put(24,-6){\circle{1.2}}
  \put(24,6){\circle{1.2}}
  \put(12,9){\circle*{1.5}}
  \put(24,18){\circle*{1.5}}
  \put(0,0){\line(4,-1){24}}
  \put(12,9){\line(4,-1){12}}
  \put(12,-3){\line(0,1){21}}
  \put(24,-6){\line(0,1){33}}
  \put(9,7){\circle*{1.5}}
  \put(21,4){\circle{1.2}}
  \put(33,1){\circle{1.2}}
  \put(33,13){\circle{1.2}}
  \put(21,16){\circle*{1.5}}
  \put(33,25){\circle*{1.5}}
  \put(21,16){\circle{2.5}}
  \put(0,9){\line(4,-1){33}}
  \put(12,18){\line(4,-1){21}}
  \put(21,4){\line(0,1){12}}
  \put(33,1){\line(0,1){24}}
  \put(0,9){\circle{1.2}}
  \put(12,18){\circle{1.2}}
  \put(24,27){\circle{1.2}}
  \put(24,27){\line(4,-1){9}}
  \put(0,0){\line(0,1){9}}
  \put(7,23){(0,1,1)}
 \put(16,22){\vector(1,-1){4}}
  \put(-3.5,3){$U$}
  \put(7,3){$V$}
  \put(7.5,12){$W$}
  \put(17.5,10){$X$}
  \put(24.5,20){$Y$}
  \put(33.5,17){$Z$}
  \put(-4,32){$\Kup[0,1,1]\simeq K(2,1)$}
   \end{picture}}
%%%
                         %              (upper 111)
   \put(110,10){\begin{picture}(40,30)
  \put(0,0){\circle*{1.5}}
  \put(12,-3){\circle{1.2}}
  \put(0,12){\circle{1.2}}
  \put(24,6){\circle{1.2}}
  \put(8,10){\circle*{1.5}}
  \put(12,21){\circle{1.2}}
  \put(12,13){\circle*{1.5}}
  \put(24,18){\circle*{1.5}}
  \put(24,18){\circle{2.5}}
  \put(0,0){\line(4,-1){12}}
  \put(0,12){\line(4,-1){24}}
  \put(12,21){\line(4,-1){12}}
  \put(0,0){\line(0,1){12}}
  \put(12,-3){\line(0,1){24}}
  \put(24,6){\line(0,1){12}}
  \put(25,18){(1,1,1)}
  \put(-4,6){$\Gamma$}
  \put(6,6){$\Phi$}
  \put(13,2){$\Delta$}
  \put(25,11){$\Psi$}
  \put(-4,25){$\Kup[1,1,1]\simeq K(1,1)$}
   \end{picture}}
%%%
 \end{picture}
 \end{center}
  \caption{The upper subcrystals in $K(1,1,1)$}
  \label{fig:upper111}
  \end{figure}
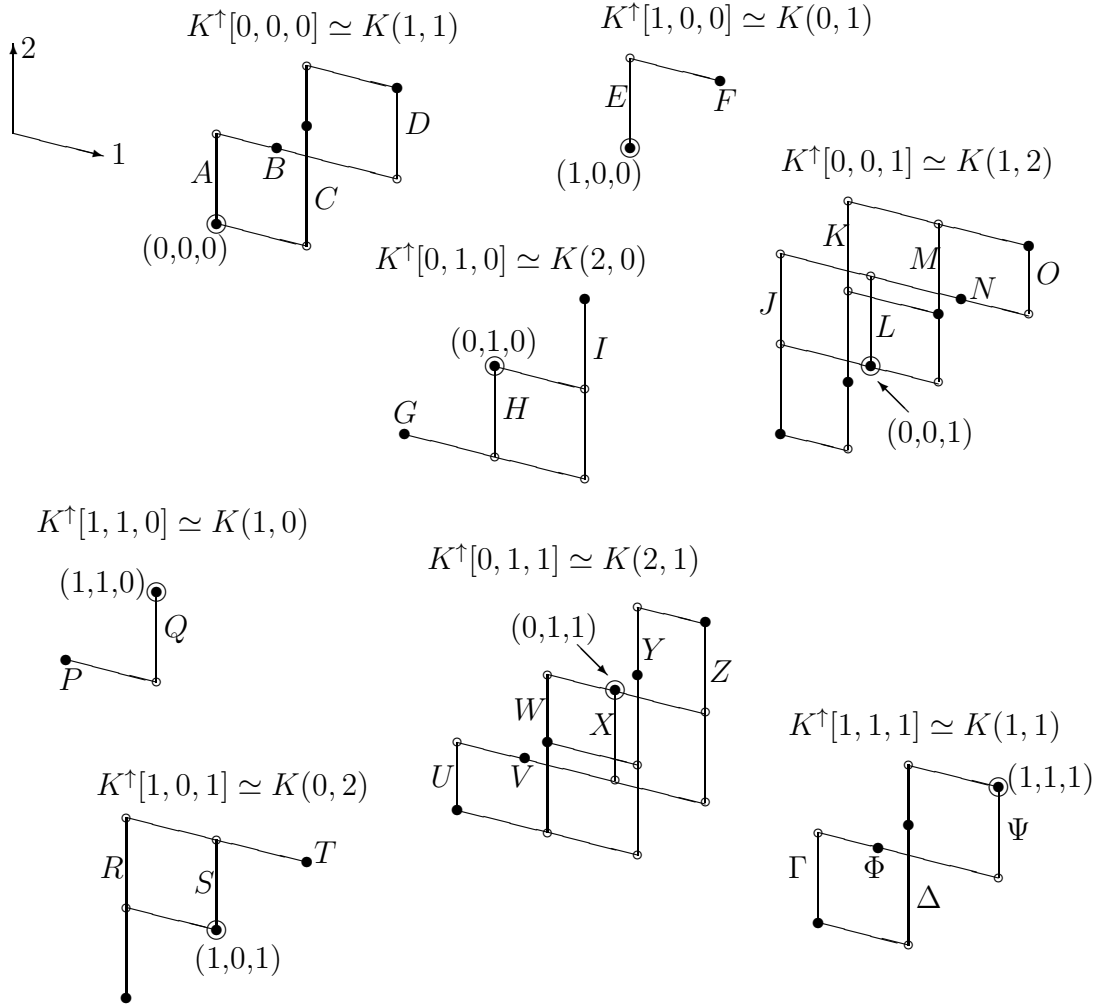

For each upper subcrystal $\Kup[i,j,k]$, its parameter $\parup$ and heart locus
$\heartup$, computed by~\refeq{par_up} and~\refeq{heart_up}, are as follows
(where $\tilde K,\tilde s,z$ denote the current subcrystal, its source, and its
heart, respectively):

$\bullet$ for $\Kup[0,0,0]$: $\parup_1=1-0+0=1$, $\parup_2=1-0+0=1$, and
$\heartup_1=\heartup_2=0$ (so $\tilde K$ is isomorphic to $K(1,1)$ and $z$
coincides with $\tilde s$);

$\bullet$ for $\Kup[1,0,0]$: $\parup_1=1-1+0=0$, $\parup_2=1-0+0=1$, and
$\heartup_1=\heartup_2=0$;

$\bullet$ for $\Kup[0,1,0]$: $\parup_1=1-0+1=2$, $\parup_2=1-1+0=0$,
$\heartup_1=1$, and $\heartup_2=0$ (so $\tilde K\simeq K(2,0)$ and $z$ is
located at $S_{2,1}(\tilde s)=F_2F_1(\tilde s)$);

$\bullet$ for $\Kup[0,0,1]$: $\parup_1=1-0+0=1$, $\parup_2=1-0+1=2$,
$\heartup_1=0$, and $\heartup_2=1$ (so $\tilde K\simeq K(1,2)$ and $z$ is
located at $S_{2,2}(\tilde s)=F_1F_2(\tilde s)$);

$\bullet$ for $\Kup[1,1,0]$: $\parup_1=1-1+1=1$, $\parup_2=1-1+0=0$,
$\heartup_1=1$, and $\heartup_2=0$;

$\bullet$ for $\Kup[1,0,1]$: $\parup_1=1-1+0=0$, $\parup_2=1-0+1=2$,
$\heartup_1=0$, and $\heartup_2=1$;

$\bullet$ for $\Kup[0,1,1]$: $\parup_1=1-0+1=2$, $\parup_2=1-1+1=1$, and
$\heartup_1=\heartup_2=1$ (so $\tilde K\simeq K(2,1)$ and $z$ is located at
$S_{2,2}S_{2,1}(\tilde s)=F_1F_2F_2F_1(\tilde s)$);

$\bullet$ for $\Kup[1,1,1]$: $\parup_1=1-1+1=1$, $\parup_2=1-1+1=1$, and
$\heartup_1=\heartup_2=1$.

  \begin{figure}[hbt]                  % Fig.3
  \begin{center}
  \unitlength=1mm
  \begin{picture}(145,148)
                        %              (upper left)
   \put(3,110){\begin{picture}(20,30)
  \put(0,10){\vector(3,1){9}}
  \put(0,10){\vector(0,1){12}}
  \put(1,20){2}
  \put(9.5,13){3}
   \end{picture}}
%%
                         %              (lower 000)
   \put(30,105){\begin{picture}(40,40)
  \put(0,0){\circle*{1.5}}
  \put(12,4){\circle{1.2}}
  \put(0,12){\circle{1.2}}
  \put(24,20){\circle{1.2}}
  \put(8,14.7){\circle*{1.5}}
  \put(12,28){\circle{1.2}}
  \put(12,20){\circle*{1.5}}
  \put(24,32){\circle*{1.5}}
  \put(0,0){\circle{2.5}}
  \put(0,0){\line(3,1){12}}
  \put(0,12){\line(3,1){24}}
  \put(12,28){\line(3,1){12}}
  \put(0,0){\line(0,1){12}}
  \put(12,4){\line(0,1){24}}
  \put(24,20){\line(0,1){12}}
  \put(-10,-4.5){(0,0,0)}
  \put(-3.5,5.5){$A$}
  \put(6,10){$G$}
  \put(13,9){$J$}
  \put(25,24.5){$U$}
  \put(-4,35){$\Klow[0,0,0]\simeq K(1,1)$}
   \end{picture}}
%%%
                         %              (lower 001)
   \put(85,120){\begin{picture}(30,25)
  \put(0,0){\circle*{1.5}}
  \put(0,12){\circle{1.2}}
  \put(12,16){\circle*{1.5}}
  \put(0,0){\circle{2.5}}
  \put(0,12){\line(3,1){12}}
  \put(0,0){\line(0,1){12}}
  \put(-10,-4.5){(0,0,1)}
  \put(-3.5,5.5){$L$}
  \put(12,11.5){$V$}
  \put(-4,20){$\Klow[0,0,1]\simeq K(0,1)$}
   \end{picture}}
%%%
                         %              (lower 010)
   \put(53,70){\begin{picture}(40,35)
  \put(0,0){\circle*{1.5}}
  \put(12,4){\circle{1.2}}
  \put(24,8){\circle{1.2}}
  \put(24,20){\circle{1.2}}
  \put(12,16){\circle*{1.5}}
  \put(24,32){\circle*{1.5}}
  \put(12,16){\circle{2.5}}
  \put(0,0){\line(3,1){24}}
  \put(12,16){\line(3,1){12}}
  \put(12,4){\line(0,1){12}}
  \put(24,8){\line(0,1){24}}
  \put(3,18){(0,1,0)}
  \put(-1.5,1.5){$B$}
  \put(12.5,10){$H$}
  \put(25,24){$W$}
  \put(-4,35){$\Klow[0,1,0]\simeq K(2,0)$}
   \end{picture}}
%%%

                         %              (lower 100)
   \put(103,75){\begin{picture}(45,50)
  \put(0,0){\circle*{1.5}}
  \put(0,12){\circle{1.2}}
  \put(0,24){\circle{1.2}}
  \put(12,16){\circle*{1.5}}
  \put(12,28){\circle{1.2}}
  \put(24,32){\circle*{1.5}}
  \put(12,16){\circle{2.5}}
  \put(0,12){\line(3,1){21}}
  \put(0,24){\line(3,1){33}}
  \put(0,0){\line(0,1){24}}
  \put(12,16){\line(0,1){12}}
  \put(9,11.7){\circle*{1.5}}
  \put(9,23.7){\circle{1.2}}
  \put(9,35.7){\circle{1.2}}
  \put(21,28.2){\circle*{1.5}}
  \put(21,40.2){\circle{1.2}}
  \put(33,44.2){\circle*{1.5}}
  \put(9,24.2){\line(3,1){12}}
  \put(9,36.2){\line(3,1){24}}
  \put(9,3.2){\line(0,1){33}}
  \put(21,19.2){\line(0,1){21}}
  \put(9,3.2){\circle{1.2}}
  \put(21,19.2){\circle{1.2}}
  \put(33,35.2){\circle{1.2}}
  \put(0,0){\line(3,1){9}}
  \put(33,35.2){\line(0,1){9}}
  \put(15,7){(1,0,0)}
 \put(17.5,10.5){\vector(-1,1){4}}
  \put(-3.5,16){$C$}
  \put(5,30){$K$}
  \put(12.5,20){$E$}
  \put(17.5,34){$R$}
  \put(24.5,29){$P$}
  \put(34,39){$\Gamma$}
  \put(0,48){$\Klow[1,0,0]\simeq K(1,2)$}
   \end{picture}}
%%%
                         %              (lower 011)
   \put(5,45){\begin{picture}(30,27)
  \put(0,0){\circle*{1.5}}
  \put(12,4){\circle{1.2}}
  \put(12,16){\circle*{1.5}}
  \put(12,16){\circle{2.5}}
  \put(0,0){\line(3,1){12}}
  \put(12,4){\line(0,1){12}}
  \put(-1,16){(0,1,1)}
  \put(-1,-4){$N$}
  \put(13,10){$X$}
  \put(-4,23){$\Klow[0,1,1]\simeq K(1,0)$}
   \end{picture}}
%%%

                         %              (lower 101)
   \put(15,0){\begin{picture}(40,30)
  \put(0,0){\circle*{1.5}}
  \put(0,12){\circle{1.2}}
  \put(0,24){\circle{1.2}}
  \put(12,16){\circle*{1.5}}
  \put(12,28){\circle{1.2}}
  \put(24,32){\circle*{1.5}}
  \put(12,16){\circle{2.5}}
  \put(0,12){\line(3,1){12}}
  \put(0,24){\line(3,1){24}}
  \put(0,0){\line(0,1){24}}
  \put(12,16){\line(0,1){12}}
  \put(12,12){(1,0,1)}
  \put(-4,16){$M$}
  \put(12.5,21){$S$}
  \put(24,28){$\Phi$}
  \put(-4,35){$\Klow[1,0,1]\simeq K(0,2)$}
   \end{picture}}
%%%

                         %              (lower 110)
   \put(62,13){\begin{picture}(50,50)
  \put(0,0){\circle*{1.5}}
  \put(12,4){\circle{1.2}}
  \put(24,8){\circle{1.2}}
  \put(24,20){\circle{1.2}}
  \put(12,16){\circle*{1.5}}
  \put(24,32){\circle*{1.5}}
  \put(0,0){\line(3,1){24}}
  \put(12,16){\line(3,1){12}}
  \put(12,4){\line(0,1){21}}
  \put(24,8){\line(0,1){33}}
  \put(9,12.1){\circle*{1.5}}
  \put(21,16.2){\circle{1.2}}
  \put(33,20.2){\circle{1.2}}
  \put(33,32.2){\circle{1.2}}
  \put(21,28.2){\circle*{1.5}}
  \put(33,44.2){\circle*{1.5}}
  \put(21,28.2){\circle{2.5}}
  \put(0,9){\line(3,1){33}}
  \put(12,25){\line(3,1){21}}
  \put(21,16.2){\line(0,1){12}}
  \put(33,20.2){\line(0,1){24}}
  \put(0,9){\circle{1.2}}
  \put(12,25){\circle{1.2}}
  \put(24,41){\circle{1.2}}
  \put(24,41){\line(3,1){9}}
  \put(0,0){\line(0,1){9}}
  \put(9,29.5){(1,1,0)}
  \put(-3.5,3){$D$}
  \put(7,8){$F$}
  \put(9.5,19){$I$}
  \put(17.5,21.5){$Q$}
  \put(24.5,35){$Y$}
  \put(33.5,38){$\Delta$}
  \put(-4,46){$\Klow[1,1,0]\simeq K(2,1)$}
   \end{picture}}
%%%
                         %              (lower 111)
   \put(110,5){\begin{picture}(40,30)
  \put(0,0){\circle*{1.5}}
  \put(12,4){\circle{1.2}}
  \put(0,12){\circle{1.2}}
  \put(24,20){\circle{1.2}}
  \put(8,14.7){\circle*{1.5}}
  \put(12,28){\circle{1.2}}
  \put(12,20){\circle*{1.5}}
  \put(24,32){\circle*{1.5}}
  \put(24,32){\circle{2.5}}
  \put(0,0){\line(3,1){12}}
  \put(0,12){\line(3,1){24}}
  \put(12,28){\line(3,1){12}}
  \put(0,0){\line(0,1){12}}
  \put(12,4){\line(0,1){24}}
  \put(24,20){\line(0,1){12}}
  \put(25,31){(1,1,1)}
  \put(-3.5,5){$O$}
  \put(7,10){$T$}
  \put(8.5,22){$Z$}
  \put(24.5,24){$\Psi$}
  \put(0,38){$\Klow[1,1,1]\simeq K(1,1)$}
   \end{picture}}
%%%
 \end{picture}
 \end{center}
  \caption{The lower subcrystals in $K(1,1,1)$}
  \label{fig:lower111}
  \end{figure}
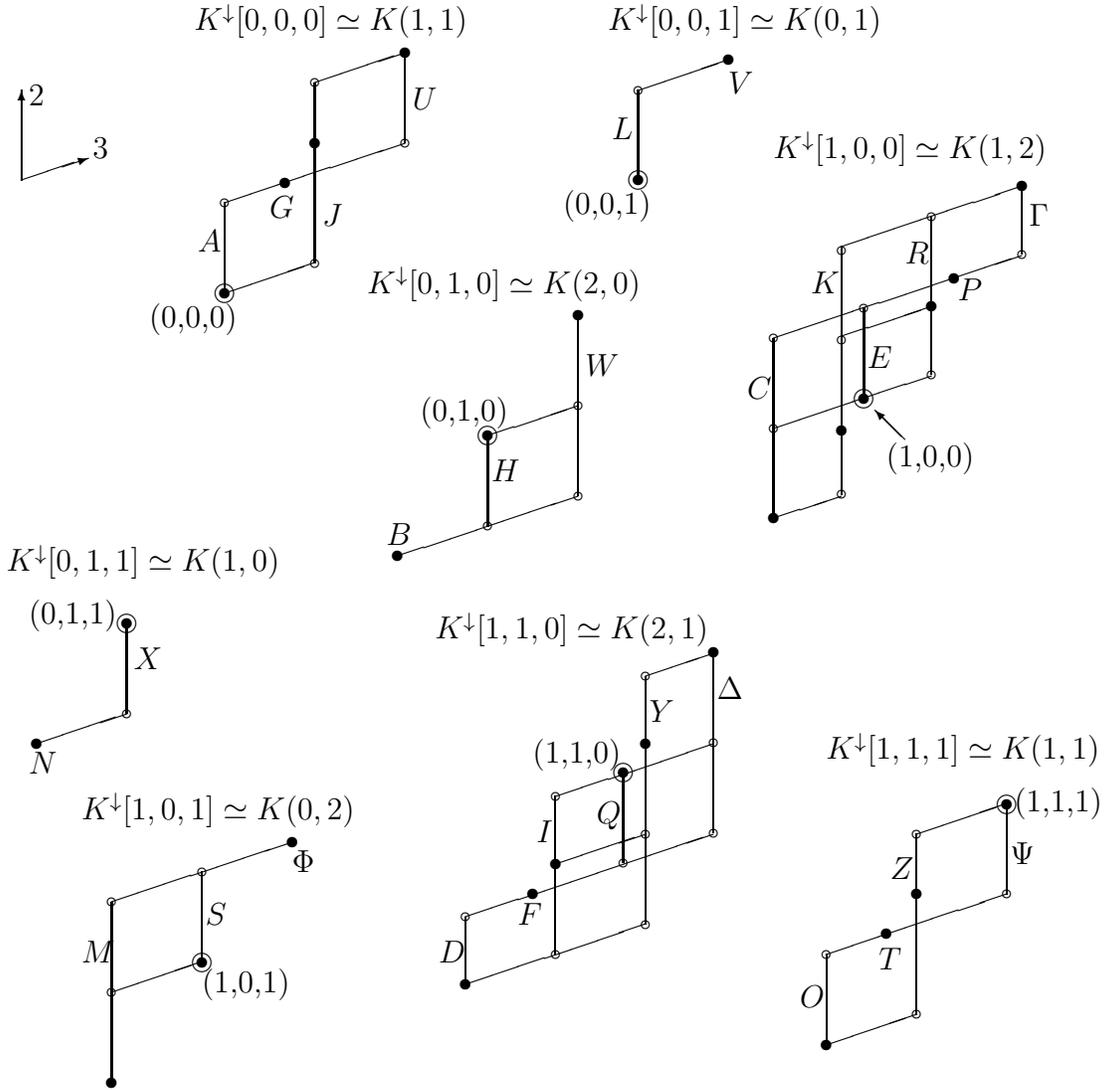

Since $K(1,1,1)$ is ``symmetric'', so are its upper and lower subcrystals, i.e.
each $\Klow[i,j,k]$ is obtained from $\Kup[k,j,i]$ by replacing color 1 by 3.
In Fig.~\ref{fig:lower111}, when writing $\Klow[i,j,k]\simeq K(\alpha,\beta)$,
the parameters $\alpha,\beta$ concern colors 3 and 2, respectively.

Now the desired $K(1,1,1)$ is assembled by gluing the fragments in
Figs.~\ref{fig:upper111},\ref{fig:lower111} along the 2-paths $A,\ldots,\Psi$.

%---------------------------  SEC. 6
 \section{Deriving $B_n$-crystals from symmetric $A_{2n-1}$-crystals} \label{sec:Bn}

We say that an $A_{2n-1}$-crystal $K=(V(K), E_1\sqcup\ldots\sqcup E_{2n-1})$
with parameter $c=(c_1,\ldots,c_{2n-1})$ is {\em symmetric} if $c_i=c_{2n-i}$
for each $i$. Equivalently: renumbering the colors $1,\ldots,2n-1$ as
$2n-1,\ldots,1$ makes the same $K$ (since any A-crystal is determined by its
parameter). Color $2n-i$ is regarded as {\em complementary} to color $i$ and
will usually be denoted with prime: we write $i'$ for $2n-i$. In particular,
$n'=n$. For an operator string $F_{i_1}^{\alpha_1}F_{i_2}^{\alpha_2}\ldots
F_{i_k}^{\alpha_k}$ (where each $F_{i_j}$ concerns the color class $E_{i_j}$
and $\alpha_{j}\in\Zset$), the complementary string is defined to be
$F_{i'_1}^{\alpha_1}F_{i'_2}^{\alpha_2}\ldots F_{i'_k}^{\alpha_k}$. This gives
a natural complementarity relation on the set of paths, including non-directed
ones, that begin at the source $s$ of $K$, which in turn yields the
complementarity bijection (involution) $\sigma:V(K)\to V(K)$. We extend
$\sigma$, in a natural way, to edges, paths and subgraphs of $K$. A vertex
$v\in V(K)$ is called {\em self-complementary} (or \emph{symmetric}) if
$v=\sigma(v)$; equivalently: for some (equivalently, \emph{any}) path $P$ from
$s$ to $v$, the complementary path $\sigma(P)$ terminates at $v$ as well. In
particular, the source and sink of $K$ are self-complementary.

Let $S$ be the set of self-complementary vertices in $K$. Clearly if a vertex
$u\in S$ has outgoing edge colored $i$, then $v$ has outgoing $i'$-edge as
well. When $i<n$, the colors $i,i'$ are not neighboring ($|i-i'|\ge 2$); so by
axiom~(A5) (from Section~\ref{ssec:typeA}), the operators $F_i$ and $F_{i'}$
commute at $v$, and the vertex $v=F_iF_{i'}(u)$ is again self-complementary. We
denote such a pair $(i,i')$ with $i<i'$ by $\bar i$. The pair $(u,v)$ of
vertices as above is regarded as an {\em edge with color} $\bar i$, or an $\bar
i$-{\em edge}; we denote the set of $\bar i$-edges by $E_{\bar i}$ and denote
the partial operator on $S$ related to $E_{\bar i}$ by $F_{\bar i}$. We also
refer to the four edges $(u,w),(w,v),(u,w'),(w',v)$ of $K$, where $w:=F_i(u)$
and $w':=F_{i'}(u)$, as the {\em underlying} edges of $(u,v)$.

As to color $n$, if $u\in S$ has outgoing $n$-edge $(u,v)$, then $v\in S$ as
well. We formally set $\bar n:=(n,n'=n)$, define $E_{\bar n}$ to be the set of
$n$-edges connecting pairs of self-complementary vertices (so $E_{\bar
n}\subseteq E_n$), and associate to $E_{\bar n}$ the partial operator $F_{\bar
n}$ on $S$.

As a result, we obtain the $n$-colored directed graph $\frakB=(S,E_{\bar
1}\sqcup\ldots\sqcup E_{\bar n})$, called the {\em symmetric extract} from $K$.
The colors in $\frakB$ are ordered as $\bar 1, \ldots, \bar n$, and different
colors $\bar i,\bar j$ are called {\em neighboring} if $|i-j|=1$.

Figure~\ref{fig:B-cryst} illustrates two ``simplest'' symmetric A-crystals for
$n=3$, namely, $K(1,0,1)$ and $K(0,1,0)$, and their symmetric extracts
$\frakB(1,0)$ and $\frakB(0,1)$.

  \begin{figure}[hbt]                  % Fig.4
  \begin{center}
  \unitlength=1mm
  \begin{picture}(140,60)
                         %              (upper left)
   \put(20,33){\begin{picture}(60,30)
 %  \put(0,0){\circle{1}}
  \put(0,0){\circle*{1}}
  \put(12,-3){\circle*{1}}
  \put(8,4){\circle*{1}}
  \put(20,1){\circle*{1}}
  \put(8,14){\circle*{1}}
  \put(18,11.5){\circle*{1}}
  \put(12,6){\circle*{1}}
  \put(19,9.5){\circle*{1}}
  \put(20,13){\circle*{1}}
  \put(20,21){\circle*{1}}
  \put(32,18){\circle*{1}}
  \put(28,25){\circle*{1}}
  \put(28,14){\circle*{1}}
  \put(32,8){\circle*{1}}
  \put(40,22){\circle*{1}}
  \put(0,0){\circle{2.5}}
  \put(20,1){\circle{2.5}}
  \put(20,13){\circle{2.5}}
  \put(20,21){\circle{2.5}}
  \put(40,22){\circle{2.5}}
  \put(0,0){\vector(4,-1){12}}
  \put(0,0){\vector(2,1){8}}
  \put(12,-3){\vector(2,1){7.2}}
  \put(8,4){\vector(4,-1){11}}
  \put(12,-3){\vector(0,1){9}}
  \put(8,4){\vector(0,1){10}}
  \put(8,14){\vector(4,-1){10}}
  \put(18,11.5){\vector(4,-1){14}}
  \put(12,6){\vector(2,1){7}}
  \put(19,9.5){\vector(2,1){9}}
  \put(20,1){\vector(0,1){12}}
  \put(20,13){\vector(0,1){7}}
  \put(20,21){\vector(2,1){8}}
  \put(20,21){\vector(4,-1){12}}
  \put(32,18){\vector(2,1){7.2}}
  \put(28,25){\vector(4,-1){11}}
  \put(28,14){\vector(0,1){11}}
  \put(32,8){\vector(0,1){10}}
  \put(-15,17){\vector(4,-1){8}}
  \put(-15,17){\vector(2,1){6}}
  \put(-15,17){\vector(0,1){8}}
  \put(-14.5,22){2}
  \put(-9,11.5){1}
  \put(-8.5,19){3}
  \put(35,0){$K(1,0,1)$}
  \put(-4,-1){$s$}
  \put(42,21){$t$}
  \put(55,15){$\Longrightarrow$}
    \end{picture}}
%%%                                       upper right
    \put(90,33){\begin{picture}(40,30)
  \put(3,1){\circle*{1}}
  \put(20,1){\circle*{1}}
  \put(20,11){\circle*{1}}
  \put(20,21){\circle*{1}}
  \put(37,21){\circle*{1}}
  \put(3,1){\circle{2.5}}
  \put(20,1){\circle{2.5}}
  \put(20,11){\circle{2.5}}
  \put(20,21){\circle{2.5}}
  \put(37,21){\circle{2.5}}
  \put(20,1){\vector(0,1){9}}
  \put(20,11){\vector(0,1){9}}
  \put(3,1){\vector(1,0){16}}
  \put(20,21){\vector(1,0){16}}
  \put(33,-1){$\frakB(1,0)$}
   \put(10,2){$\bar 1$}
   \put(21,4.5){$\bar 2$}
   \put(17,14.5){$\bar 2$}
   \put(28,16.5){$\bar 1$}
  \put(-1,0){$s$}
  \put(39,20){$t$}
   \end{picture}}
%%%                                lower left
 \put(30,0){\begin{picture}(40,25)
  \put(0,0){\circle*{1}}
  \put(12,7){\circle*{1}}
  \put(8,14){\circle*{1}}
  \put(20,11){\circle*{1}}
  \put(0,10){\circle*{1}}
  \put(20,21){\circle*{1}}
  \put(0,0){\circle{2.5}}
  \put(20,11){\circle{2.5}}
  \put(0,10){\circle{2.5}}
  \put(20,21){\circle{2.5}}
  \put(0,10){\vector(4,-1){12}}
  \put(0,10){\vector(2,1){8}}
  \put(12,7){\vector(2,1){7.2}}
  \put(8,14){\vector(4,-1){11}}
  \put(0,0){\vector(0,1){9}}
  \put(20,11){\vector(0,1){9}}
  \put(20,0){$K(0,1,0)$}
  \put(-4,-1){$s$}
  \put(22,21){$t$}
  \put(45,15){$\Longrightarrow$}
    \end{picture}}
%%%                                 lower right
    \put(95,0){\begin{picture}(30,30)
  \put(0,0){\circle*{1}}
  \put(15,10){\circle*{1}}
  \put(0,10){\circle*{1}}
  \put(15,20){\circle*{1}}
  \put(0,0){\circle{2.5}}
  \put(15,10){\circle{2.5}}
  \put(0,10){\circle{2.5}}
  \put(15,20){\circle{2.5}}
  \put(0,0){\vector(0,1){9}}
  \put(15,10){\vector(0,1){9}}
  \put(0,10){\vector(1,0){14}}
  \put(20,0){$\frakB(0,1)$}
   \put(7,6){$\bar 1$}
   \put(-3,4){$\bar 2$}
   \put(16,14){$\bar 2$}
  \put(-4,-1){$s$}
  \put(17,20){$t$}
    \end{picture}}
%%%
 \end{picture}
 \end{center}
  \caption{Creation of $B(1,0)$ and $B(0,1)$.}
  \label{fig:B-cryst}
  \end{figure}
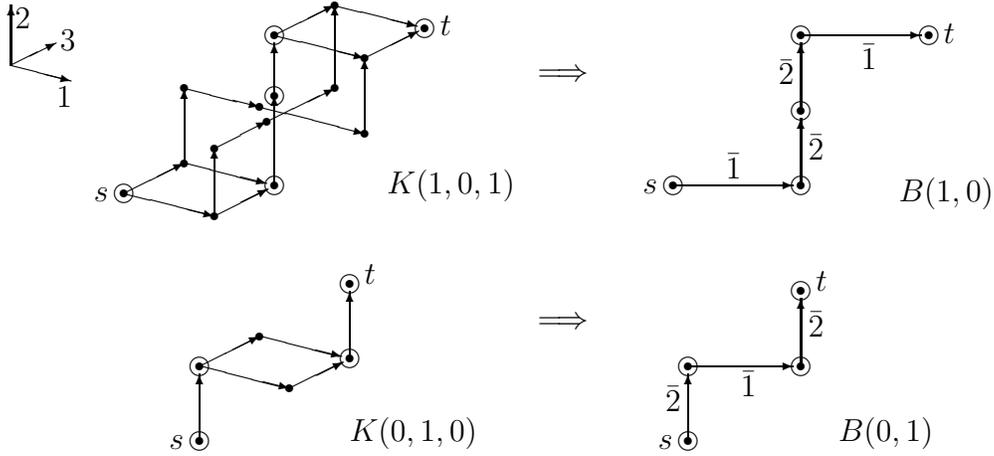

The following relation between $A$- and $B$-crystals can be concluded
from~\cite[Th.~3.2.4]{NS}.
  \begin{theorem} \label{tm:A-B}
Let $K(c)$ be a symmetric $A_{2n-1}$-crystal. Then the symmetric extract
$\frakB=(S,E_{\bar 1}\sqcup\ldots\sqcup E_{\bar n})$ from $K(c)$ is a
$B_n$-crystal.
  \end{theorem}

In Sections~\ref{sec:Bn}--\ref{sec:proofB3-B4} we give a combinatorial proof of
this theorem, based on our knowledge of the structure of $A_n$- and
$B_2$-crystals. We start with simple observations.

 \medskip
1. An important property of $K=K(c)$ is that it is \emph{graded} at each color
$i$, which means that in any closed route in $K$, the numbers of forward and
backward $i$-edges are equal (equivalently, $V(K)$ admits a map to $\Zset^n$
under which each $i$-edge corresponds to a shift by $i$-th base vector). Clearly
such a property remains valid for $\frakB$ as well. In particular, $\frakB$ is
acyclic. Also for $i=1,\ldots,n$, each vertex of $\frakB$ has at most one
outgoing $\bar i$-edge and at most one incoming $\bar i$-edge. So each component
of $(S,E_{\bar i})$ is a path (as required in (A1)). Another known fact is that
the graph $K^{rev}$ obtained from (a not necessarily symmetric A-crystal) $K$ by
reversing its edges and changing each edge color to the complementary one is
isomorphic to $K$ (this operation swaps the source and sink). This implies that
the ``reversed'' graph $\frakB^{rev}$ (with preserved edge colors) is isomorphic
to $\frakB$; this fact will be used in Section~\ref{ssec:B3}.

  \medskip
2. For a vertex $v\in S$ and color $\bar i$, let $h_{\bar i}(v)$ (resp.
$t_{\bar i}(v)$) denote the length of the maximal $\bar i$-path beginning
(resp. ending) at $v$. Then
  \begin{equation} \label{eq:barlength}
  h_{\bar i}(v)=h_i(v)=h_{i'}(v)\quad\mbox{and}\quad
          t_{\bar i}(v)=t_i(v)=t_{i'}(v).
   \end{equation}
This is trivial when $i=n$. If $i<n$, consider the component $K'$ of
$(V(K),E_i\sqcup E_{i'})$ that contains $v$. Since $|i-i'|\ge 2$, $K'$ is the
Cartesian product of an $i$-path $P$ and an $i'$-path $P'$. Since $K'$ contains
a self-complementary vertex, it easily follows that $K'=\sigma(K')$. This
implies that the lengths of $P$ and $P'$ are equal, and further, that
$V(K')\cap S$ consists of the vertices of the form $F_i^\alpha
F_{i'}^\alpha(s')$, where $s'$ is the source of $K'$ (the ``diagonal'' of
$K'$). Now~\refeq{barlength} easily follows.

  \medskip
3. Each vertex $v\in S$ is reachable by a (directed) path in $\frakB$ beginning
at the source $s$ of $K$; in particular, $\frakB$ is connected and $s$ is the
source of $\frakB$. This can be shown by induction on the length $|P|$ of a
path $P$ from $s$ to $v$ in $K$. Indeed, for such a $P$, take the complementary
path $P'$ (also going from $s$ to $v$). Let the last edge $(u,v)$ of $P$ have
color $i$; then the last edge $(u',v)$ of $P'$ has color $i'$. If $i=n$, then
$u=u'$, implying $u\in S$, and we can apply induction. And if $i\ne n$, then
$F_i^{-1}$ and $F_{i'}^{-1}$ commute at $v$ and the vertex
$w:=F_i^{-1}F_{i'}^{-1}(v)$ is self-complementary. Since $K$ is graded, the
length of a path from $s$ to $w$ in $K$ is less than $|P|$. So by induction $w$
is reachable by a path from $s$ in $\frakB$, implying a similar property for
$v$.

 \medskip
4. By~\refeq{barlength} applied to $v=s$, we have $h_{\bar i}(s)=c_i$ for each
$i=1,\ldots,n$. So one can regard the $n$-tuple $\bar c=(c_1,\ldots,c_n)$ as
the {\em parameter} of $\frakB$ and denote $\frakB$ as $\frakB(\bar c)$ (the
set of such tuples $\bar c$ is $\Zset_+^n$, and there is a unique $\frakB$ for
each $\bar c$ in our construction).

 \medskip
5. Let $i,j<n$ and $|i-j|\ge 2$. Then any two colors among $i,i',j,j'$ are not
neighboring, and therefore (by~(A5)), each subcrystal $K'$ of $K$ with these
four colors is the Cartesian product of four monochromatic paths. This implies
that if two operators among $F_{\bar i},F_{\bar i}^{-1},F_{\bar j},F_{\bar
j}^{-1}$ act at $v\in S$, then these operators commute at $v$, whence each
component of $(S,E_{\bar i}\sqcup E_{\bar j})$ is the Cartesian product of an
$\bar i$-path and a $\bar j$-path, i.e. an $A_1\times A_1$-crystal. A similar
fact is shown for $(S,E_{\bar i}\sqcup E_{\bar n})$ when $i\le n-2$. Thus,
$\frakB$ satisfies axiom~(BC1) (from Section~\ref{ssec:typeBC}).

 \medskip
It remains to verify axioms~(BC2),(BC3),(BC4) for $\frakB$. We start
with~(BC2).
  \begin{lemma} \label{lm:B2}
~Let $i,j<n$ and $|i-j|=1$. Then each component of the subgraph $\frakB':=(S,
E_{\bar i}\sqcup E_{\bar j})$ is an $A_2$-crystal.
  \end{lemma}
  \begin{proof}
To verify axiom~(A2) for $\frakB'$, consider an $\bar i$-edge $(u,v)$. Let
$x:=F_i(u)$ and $y:=F_{i'}(u)$; then $(u,x),(y,v)$ are the $i$-edges and
$(u,y),(x,v)$ are the $i'$-edges of $K$ underlying $(u,v)$. Since
$|i-j|=|i'-j'|=1$ and $|i-j'|=|i'-j|\ge 2$, we obtain (using~(A2),(A5) for
$K$):
  \begin{gather*}
  h_j(v)=h_j(x)=h_j(u)+\ell_j(u,x); \quad
  h_j(v)=h_j(y)+\ell_j(y,v)=h_j(u)+\ell_j(y,v); \\
  h_{j'}(v)=h_{j'}(x)+\ell_{j'}(x,v)=h_{j'}(u)+\ell_{j'}(x,v); \quad
  h_{j'}(v)=h_{j'}(y)=h_{j'}(u)+\ell_{j'}(u,y)
   \end{gather*}
(labels $\ell$ are defined for A-crystals in~Section~\ref{ssec:typeA}). These
and the equalities $h_j(u)=h_{j'}(u)=h_{\bar j}(u)$ and
$h_j(v)=h_{j'}(v)=h_{\bar j}(v)$ (cf.~\refeq{barlength}) imply
  $$
 \ell_j(u,x)=\ell_j(y,v)=\ell_{j'}(x,v)=\ell_{j'}(u,y)=:\alpha,
  $$
and
  $$
  h_{\bar j}(v)=h_{\bar j}(u)+\alpha.
   $$

Handling lengths $t_j,t_{j'}$ in a similar way, we obtain
$t_{\bar j}(v)=t_{\bar j}(u)+1-\alpha$.

Therefore, when traversing $(u,v)$, the lengths $h_{\bar j}$ and $t_{\bar
j}$ behave as required for A-crystals, and
  \begin{numitem1}
all underlying edges of an $\bar i$-edge $e$ have one and the same label
$\alpha$ w.r.t. their neighboring colors in $\{j,j'\}$, and $e$ inherits
just this label: $\ell_{\bar j}(e)=\alpha$.
  \label{eq:commonlab}
  \end{numitem1}

To check the convexity condition in~(A2), consider consecutive $\bar i$-edges
$(u,v),(v,w)$ and take their underlying $i$-edges $(x,v)$ and $(v,y')$, where
$x:=F_i^{-1}(v)$ and $y':=F_i(v)$. Then $\ell_j(x,v)\le \ell_j(v,y')$ (by~(A2)
for $K$) implies $\ell_{\bar j}(u,v)\le\ell_{\bar j}(v,w)$,
by~\refeq{commonlab}. Thus, $\frakB'$ satisfies axiom~(A2).

 \smallskip
Next, let $u\in S$ have outgoing $\bar i$-edge $(u,v)$ and outgoing $\bar
j$-edge $(u,v')$ in $\frakB$, and let $\ell_{\bar j}(u,v)=0$. By~(A2), $h_{\bar
j}(v)\ge h_{\bar j}(u)$; so $v$ has outgoing $\bar j$-edge $(v,w)$. Similarly,
$v'$ has outgoing $\bar i$-edge $(v',w')$. We assert that $w=w'$ (as required
in~(A3)), i.e.
  $$
  F_{\bar j}F_{\bar i}(u)=F_{\bar i}F_{\bar j}(u).
  $$

To show this, let $x:=F_{i'}(u)$ and $y:=F_j(u)$ (these vertices are not in
$S$). Then $u,x,y$ and $z:=F_j(x)$ are connected by the $i'$-edges
$(u,x),(y,z)$ and the $j$-edges $(u,y),(x,z)$ (since $|i'-j|\ge 2$).
By~\refeq{commonlab}, $\ell_{\bar j}(u,v)=0$ implies $\ell_{j'}(u,x)=0$.
Therefore, $h_{j'}(x)=h_{j'}(u)$. This together with the trivial equalities
$h_{j'}(y)=h_{j'}(u)$ and $h_{j'}(z)=h_{j'}(x)$ (as $|j-j'|\ge 2$) implies
$h_{j'}(y)=h_{j'}(z)$, which means that $\ell_{j'}(y,z)=0$. So the operators
$F_{i'},F_{j'}$ commute at $y$; note that $F_{j'}(y)=v'$. We have
  \begin{multline*}
  w=F_{j'}F_jF_{i'}F_i(u)=F_{j'}F_{i'}F_jF_i(u)=F_{j'}F_{i'}F_iF_j(u) \\
   =F_{j'}F_iF_{i'}F_j(u)=F_iF_{j'}F_{i'}F_j(u)=F_iF_{i'}F_{j'}F_j(u)=w'
   \end{multline*}
(since: $F_i,F_j$ commute at $u$; ~$F_{i'},F_{j'}$ commute at $y=F_j(u)$; and
$F_{p},F_{q}$ with $|p-q|\ge 2$ are permutable). Also the fact that
$\ell_i(u,y)=1$ (by~(A3) for $K$) implies $\ell_{\bar i}(u,v')=1$. Thus,
$\frakB'$ satisfies the part of~(A3) concerning the forward operators $F_{\bar
i},F_{\bar j}$. The claim for the backward operators $F_{\bar i}^{-1},F_{\bar
j}^{-1}$ follows by reversing the edges of $K$ and $\frakB$.

 \smallskip
Finally, instead of a direct (and tiresome) verification of axiom~(A4) for
$\frakB'$, we can appeal to the result in~\cite[Proposition~5.3]{A2} saying
that for a connected two-colored graph $K'$,~(A4) follows from~(A1),(A2),(A3)
and the condition that $K'$ has exactly one source (zero-indegree vertex).

In light of this, consider a component $\tilde \frakB$ of $\frakB'$. Let
$\tilde K$ be the component of $(V(K),E_i\sqcup E_{i'}\sqcup E_j\sqcup E_{j'})$
containing the vertices of $\tilde \frakB$. This $\tilde K$ is the Cartesian
product of two $A_2$-crystals (with colors $i,j$ and colors $i',j'$), whence
$\tilde K$ has a unique source $\tilde s$. We claim that $\tilde s$ is the
unique source of $\tilde \frakB$.

Indeed, since $\tilde \frakB$ is finite and acyclic, it has a source $s'$.
Suppose $s'\ne\tilde s$. Then $s'$ has an incoming edge $e$ of some color among
$\{i,i',j,j'\}$. Let for definiteness $e$ be an $i$-edge. Then $F_i^{-1}$ acts
at $s'$, and by the symmetry, so does $F_{i'}^{-1}$. These operators commute,
and $v:=F_{i'}^{-1}F_{i}^{-1}(s')$ is a self-complementary vertex. Then
$(v,s')$ is an edge of $\tilde \frakB$ entering $s'$, contrary to the choice of
$s'$. Thus, $s'=\tilde s$, and~(A4) for $\frakB'$ follows.

This completes the proof of the lemma. \hfill
 \end{proof}

The crucial point is to show validity of~(BC3) and~(BC4). Since these axioms
concern only colors $n-1,n,n+1$ in $K$, we may assume that $n=2$. Here we use
the simple fact that if a component $K'$ of $(V(K),E_{n-1}\sqcup E_n\sqcup
E_{n+1})$ contains a self-complementary vertex, then for each $v\in V(K')$, the
vertex $\sigma(v)$ belongs to $K'$ as well. Hence $K'$ is symmetric.

  \begin{theorem} \label{tm:B3-B4}
Let $K=(V(K),E_1\sqcup E_2\sqcup E_3)$ be a symmetric $A_3$-crystal with
parameter $c=(c_1,c_2,c_3=c_1)$. Then the symmetric extract $\frakB(\bar
c)=(S,E_{\bar 1}\sqcup E_{\bar 2})$ from $K$ is the $B_2$-crystal with
parameter $\bar c=(c_1,c_2)$ respecting the Cartan coefficients $m_{12}=-2$ and
$m_{21}=-1$.
  \end{theorem}

%---------------------------  SEC. 7
 \section{The worm model} \label{sec:worm}

Our method of proof of Theorem~\ref{tm:B3-B4} (given in the next section)
consists in showing that the graph $\frakB$ figured there is isomorphic to the
graph generated by the so-called {\em worm model} for the given parameter $\bar
c=(c_1,c_2)$. This relies on the fact that the latter graph is just the
$B_2$-crystal for $\bar c$. In this section we review the construction of
\emph{worm graphs} and operations on them given in~\cite{B2}.
\medskip

Given a parameter $\bar c=(c_1,c_2)\in\Zset_+^2$, the worm model produces a
two-colored directed graph $W=W(\bar c)$, called the {\em worm graph} for $\bar
c$. The vertices of $W$ are the admissible six-tuples $w=(x',y, x''\, ;\,
y',x,y'')$ of integers satisfying
  \begin{equation} \label{eq:wormpar}
  0\le x,x',x''\le 2c_1 \quad\mbox{and} \quad 0\le y,y',y''\le c_2.
  \end{equation}
Here the six-tuple $w$ is called {\em admissible} if the following three
conditions hold:
  \begin{numitem1}
 (i) $x'$ and $x''$ are even; \\
 (ii) $y'\le y\le y''$ and $x'\le x\le x''$; \\
 (iii) if $y'<y$ then $x'=x$, and if $y<y''$ then $x=x''$.
  \label{eq:admiss}
 \end{numitem1}
It is convenient to visualize $w$ by taking four points in the rectangle
$R(c_1,c_2):=\{(\alpha,\beta)\in\Rset^2\colon 0\le \alpha \le 2c_1,
~0\le\beta\le c_2\}$, namely:
                        $$
        X'=(x',y), \quad X''=(x'',y), \quad Y'=(x,y') \quad
        \mbox{and}\quad Y''=(x,y''),
         $$
and drawing the horizontal line-segment $X'X''$ connecting $X'$ and $X''$ and
the vertical line-segment $Y'Y''$ connecting $Y'$ and $Y''$.
Then~\refeq{admiss} is equivalent to the following:
   \begin{numitem1}
 (i) the first coordinates of the points $X'$ and $X''$ are even; \\
 (ii) the point $X''$ lies to the right of $X'$, and the point
$Y''$ lies above $Y'$; \\
 (iii) the segments $X'X''$ and $Y'Y''$ have nonempty intersection; \\
 (iv) at least one of $X'=X''$, $X'=Y''$, $Y'=Y''$, $Y'=X''$ holds.
    \label{eq:vizual}
  \end{numitem1}
Depending on the equality in~\refeq{vizual}(iv), we distinguish between four
sorts of vertices of $W$, also called {\em worms}:
  \begin{itemize}
  \item[]
  {\em V-worm} (viz. {\em vertical} worm) appears when \;\;$Y'\le Y''$ and $X'=X''$; \\
  {\em VH-worm}: \;\;$Y'\le Y''=X'\le X''$; \\
  {\em HV-worm}: \;\;$X'\le X''=Y'\le Y''$; \\
  {\em H-worm} (viz. {\em horizontal} worm): \;\;$X'\le X''$ and $Y'=Y''$. \\
  \end{itemize}
(When, e.g., $X'=X''=Y'$, we regard $w$ as a V-worm and an HV-worm
simultaneously.) These cases are illustrated (from left to right) in the
picture; hereinafter $X$ stands for the point $X'=X''$, and $Y$ for $Y'=Y''$.

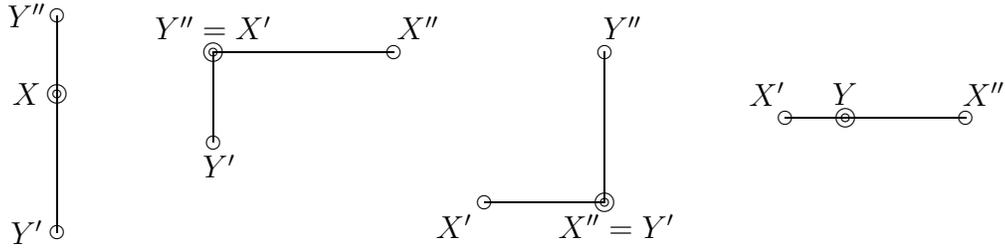
\unitlength=.800mm \special{em:linewidth 0.4pt} \linethickness{0.4pt}
 \begin{picture}(175.00,48.00)(0,5)
 \put(19.00,10.00){\circle{2.00}}
 \put(19.00,46.00){\circle{2.00}}
 \put(19.00,33.00){\circle{1.50}}
 \put(19.00,33.00){\circle{3.00}}
 \put(19,10){\line(0,1){36}}
 \put(14.00,10.00){\makebox(0,0)[cc]{$Y'$}}
 \put(14.00,46.00){\makebox(0,0)[cc]{$Y''$}}
 \put(14.00,33.00){\makebox(0,0)[cc]{$X$}}
 \put(45.00,25.00){\circle{2.00}}
 \put(45.00,40.00){\circle{1.50}}
 \put(45.00,40.00){\circle{3.00}}
 \put(75.00,40.00){\circle{2.00}}
 \put(45,25){\line(0,1){15}}
 \put(45,40){\line(1,0){30}}
 \put(46.00,21.00){\makebox(0,0)[cc]{$Y'$}}
 \put(45.00,44.00){\makebox(0,0)[cc]{$Y''=X'$}}
 \put(79.00,44.00){\makebox(0,0)[cc]{$X''$}}
 \put(90.00,15.00){\circle{2.00}}
 \put(110.00,15.00){\circle{1.50}}
 \put(110.00,15.00){\circle{3.00}}
 \put(110.00,40.00){\circle{2.00}}
 \put(90,15){\line(1,0){20}}
 \put(110,15){\line(0,1){25}}
 \put(85.00,11.00){\makebox(0,0)[cc]{$X'$}}
 \put(112.00,11.00){\makebox(0,0)[cc]{$X''=Y'$}}
 \put(113.00,44.00){\makebox(0,0)[cc]{$Y''$}}
 \put(140.00,29.00){\circle{2.00}}
 \put(150.00,29.00){\circle{1.50}}
 \put(150.00,29.00){\circle{3.00}}
 \put(170.00,29.00){\circle{2.00}}
 \put(140,29){\line(1,0){30}}
 \put(137.00,33.00){\makebox(0,0)[cc]{$X'$}}
 \put(150.00,33.00){\makebox(0,0)[cc]{$Y$}}
 \put(173.00,33.00){\makebox(0,0)[cc]{$X''$}}
     \end{picture}

 \smallskip
A worm $w$ is called {\em proper} if three points among $X',X'',Y',Y''$ are
different. When a worm degenerates into one point, we say that the worm is {\em
principal} (which matches a principal vertex in the related $B_2$-crystal). The
horizontal line-segment $X'X''$ is called the {\em horizontal limb} of $w$
(which degenerates into the single point $X$ in the V-worm case). Also when
$Y'$ (resp. $Y''$) does not lie in the line-segment $X'X''$, we say that the
vertical line-segment $Y'X'$ is the {\em lower limb} (resp. $X''Y''$ is the
{\em upper limb}) of $w$.

Next we explain the construction of edges of $W$. We denote the edge colors by
$\tilde 1$ and $\tilde 2$, and write $\tilde{\bf 1}$ and $\tilde{\bf 2}$ for
the partial operators on the worms associated to these colors, respectively.
The action of $\tilde{\bf 1}$ on a worm $v=(x',y, x''\,;\,y',x,y'')$ is as
follows:
  \begin{numitem1}
   (i) if $2x>x'+x''$ then $x'$ increases by 2; \\
   (ii) if $x=x'=x''$ and $y''>y$ then $y$ increases by 1; \\
   (iii) otherwise $x''$ increases by 2
    \label{eq:edges1}
  \end{numitem1}
(preserving the other entries). The operator does not act if the new six-tuple
would violate the boundary condition~\refeq{wormpar}. So in case of a proper
HV-worm, the point $X'$ moves by two positions to the right; in case of a
VH-worm, the point $X''$ moves by two positions to the right; in case of a
V-worm with $X\ne Y''$, the point $X$ moves by one position up. The case of
H-worms is a bit tricky: one should move (by two positions to the right) that
of the points $X',X''$ which is farther from $Y$; if they are equidistant from
$Y$, then the point $X''$ moves.

In its turn, the action of $\tilde{\bf 2}$ on $v$ is as follows:
  \begin{numitem1}
     (iv) if $2y>y'+y''$, then $y'$ increases by 1; \\
     (v) if $y''=y=y'$ and $x''>x$, then $x$ increases by 1; \\
     (vi) otherwise $y''$ increases by 1.
    \label{eq:edges2}
  \end{numitem1}
So the operator $\tilde{\bf 2}$ shifts $Y'$ ($Y''$) by one position up in the
proper VH-case (resp. in the HV-case) and shifts $Y$ by one position to the
right in the H-case with $Y\ne X''$. In the V-case, $\tilde{\bf 2}$ shifts (by
one position up) that of the points $Y',Y''$ which is farther from $X$; if they
are equidistant from $X$, then $Y''$ moves.
   \begin{theorem} {\rm \cite{B2}} \label{tm:worm-B2}
For each $\bar c\in\Zset_+^2$, the worm graph $W(\bar c)$ is isomorphic to the
$B_2$-crystal $B(\bar c)$ (where colors $\tilde 1,\tilde 2$ correspond to
$1,2$, respectively).
   \end{theorem}

 %-------------------------- SEC 8
\section{Symmetric extracts from $A_3$-crystals are $B_2$-crystals} \label{sec:proofB3-B4}

In this section we use the above worm model to prove Theorem~\ref{tm:B3-B4}
(thus completing the proof of Theorem~\ref{tm:A-B}). The proof falls into three
stages, described in Sections~\ref{ssec:B1}--\ref{ssec:B3} below.

Let $K=K(c)$ be a symmetric $A_3$-crystal with parameter $c=(c_1,c_2,c_3=c_1)$,
and $\frakB=\frakB(\bar c)$ the symmetric extract from $K$, where $\bar
c=(c_1,c_2)$. As before, $S$ denotes the set of self-complementary vertices in
$K$, or the vertices of $\frakB$. For brevity the partial operators
$F_1,F_2,F_3$ on the vertices of $K$ are denoted as $\bfone,\bftwo,\bfthree$,
and the corresponding operators $F_{\bar 1},F_{\bar 2}$ for $\frakB$ by
$\bar\bfone,\bar\bftwo$ (respectively).

 %------------------
 \subsection{Additional relations} \label{ssec:B1}

Our first goal is to establish additional facts (in Lemma~\ref{lm:relat}) about
self-complementary vertices of $K$ which will be needed to relate $\frakB$ to
the worm graph for $(c_1,c_2)$ defined in Section~\ref{sec:worm}.

In case $n=3$, relations in~\refeq{string} on the principal lattice $\Pi$ of
$K$ are specified as:
  $$
  \begin{array}{rl}
 S_{3,1}&=w_{3,1,3}w_{3,1,2}w_{3,1,1}=\bfthree\bftwo\bfone; \\
 S_{3,2}&=w_{3,2,2}w_{3,2,1}=\bftwo\bfthree\bfone\bftwo=\bftwo\bfone\bfthree\bftwo; \\
 S_{3,3}&=w_{3,3,1}=\bfone\bftwo\bfthree.
  \end{array}
  $$
These relations together with~\refeq{prin_str} and the fact that any operator
strings $S_{3,k}$ and $S_{3,k'}$ commute within $\Pi$ imply that for any
principal vertex $\prv[a=(a_1,a_2,a_3)]$ of $K$, its complementary vertex
$\sigma(\prv[a])$ is also principal and has the form $\prv[(a_3,a_2,a_1)]$. (To
see this, take the path from the source $s$ of $K$ to $\prv[a]$ corresponding
to the string $S^{a_3}_{3,3}S^{a_2}_{3,2}S^{a_1}_{3,1}$. Its complementary path
goes to $\sigma(\prv[a])$ and corresponds to the string
$S^{a_3}_{3,1}S^{a_2}_{3,2}S^{a_1}_{3,3}$.) Furthermore, the lower subcrystal
$\Klow[(a_3,a_2,a_1)]$ (with colors 2,3) is complementary to the upper
subcrystal $\Kup[a]$ (with colors 1,2).

Consider a self-complementary vertex $v\in S$ of $K$. It belongs to some upper
subcrystal $\Kup[a]$ and some lower subcrystal $\Klow[b]$. The component
(middle subcrystal) of $\Kup[a]\cap \Klow[b]$ containing $v$ is a path $P$ of
color 2. By a general fact (cf. Proposition~\ref{pr:prlat-subcryst}), $P$ has
exactly one vertex, $z$ say, in the upper lattice $\Piup[a]$; similarly, $P$
has exactly one vertex, $z'$ say, in the lower lattice $\Pilow[b]$. Let
$\Delta=(\Delta_1,\Delta_2)$ be the deviation of $z$ in $\Piup[a]$ (from the
heart $\prv[a]$ of $\Kup[a]$), and $\nabla=(\nabla_2,\nabla_3)$ the deviation
of $z'$ in $\Pilow[b]$. Using notation from Section~\ref{sec:ass_A}, we denote
$z$ and $z'$ as $\vup[a|\Delta]$ and $\vlow[b|\nabla]$, respectively. The next
lemma exhibits important features of $v$.
  \begin{lemma} \label{lm:relat}
For $a,b,\Delta,\nabla$ as above, the following properties hold:

{\rm (i)} $(b_1,b_2,b_3)=(a_3,a_2,a_1)$;

{\rm (ii)} $\Klow[b]$ is complementary to $\Kup[a]$;

{\rm(iii)} $\Delta_1=-\Delta_2=\nabla_3=-\nabla_2$;

{\rm(iv)} $\vup[a|\Delta]=\vlow[b|\nabla]$.
  \end{lemma}
  \begin{proof}
We can reach the vertex $v$ from the source $s$ of $K$ by moving along the
concatenation of three paths $P_1,P_2,P_3$, where: $P_1$ goes from $s$ to
$\prv[a]$; $P_2$ is a path from $\prv[a]$ to $\vup[a|\Delta]$ in $\Kup[a]$; and
$P_3$ is a path from $\vup[a|\Delta]$ to $v$ in $P$. (The paths $P_2,P_3$ are
not necessarily directed.) Take the complementary paths $P'_1,P'_2,P'_3$ to
$P_1,P_2,P_3$, respectively.

The vertex $v$ belongs to the lower subcrystal $\Klow[b]$. On the other hand,
$v$ belongs to the lower subcrystal $K'$ complementary to $\Kup[a]$, which is
expressed as $\Klow[(a_3,a_2,a_1)]$ (since the end of $P_1$ is the principle
vertex $\prv[a]$ of $K$ and the end of $P'_1$ is the complementary principle
vertex $\sigma(\prv[a])=\prv[(a_3,a_2,a_1)]$). This yields~(i),(ii).

By relation~\refeq{nablai} in Theorem~\ref{tm:mainA}, we have
$\nabla_2=-\Delta_1$ and $\nabla_3=-\Delta_2$. Also since the end
$\vlow[b|\nabla]$ of $P'_2$ is complementary to the end $\vup[a|\Delta]$ of
$P_2$, the deviation $\nabla$ is complementary to $\Delta$, i.e.
$\nabla_2=\Delta_2$ and $\nabla_3=\Delta_1$. This yields~(iii).

Finally, since the monochromatic paths $P_3$ and $P'_3$ have the same color 2
and end at the same vertex $v$, we have $P'_3=P_3$. Therefore, the beginning
vertices $\vup[a|\Delta]$ and $\vlow[b|\nabla]$ of these paths coincide,
yielding~(iv). \hfill
  \end{proof}

%------------------------
 \subsection{A correspondence between symmetric vertices and worms}
 \label{ssec:B2}

In this subsection we explain how to associate the elements of $S$ to the
vertices (worms) of the worm graph $W(\bar c)$.

For $v\in S$ and its corresponding $a,\Delta$ as above, we will write $a(v)$
for $a$, and $\Delta(v)$ for $\Delta$. By (iii) in Lemma~\ref{lm:relat}, the
deviation $\Delta$ is of the form
  $$
  \Delta=(\Delta_1,\Delta_2)=(\delta,-\delta)
   $$
for some $\delta=\delta(v)\in\Zset$. By (i) in that lemma, $b_1=a_3$; this
together with the equality $b_1=a_1+\Delta^+_1$ (cf.~\refeq{bi} in
Theorem~\ref{tm:mainA}) implies $a_3=a_1+\Delta^+_1$. This gives certain
constrains on $a,\delta$, namely:
  \begin{numitem1}
one always holds $a_3\ge a_1$; furthermore, if $a_3>a_1$ then $\delta=a_3-a_1$
($>0$), and if $a_3=a_1$ then $\delta\le 0$.
  \label{eq:a-delta}
  \end{numitem1}

So $a(v),\Delta(v)$ are determined by $a_1(v),a_2(v),\delta(v)$. The latter
triple together with the coordinate $\ell=\ell(v)$ of $v$ in the corresponding
middle subcrystal (path of color 2) $P=P(v)$ determine $v$ in $S$. We will
refer to the quadruple $(a_1(v),a_2(v),\delta(v),\ell(v))$ as the
\emph{description} of $v$. In the monochromatic subcrystal $P$, the principal
lattice consists of all vertices of $P$; the heart is the vertex
$\vup[a|\Delta]$, further denoted by $z=z(v)$; and $\ell$ is the length of the
subpath in $P$ from the beginning to $v$ (the \emph{tail length} for $v$).

The description $(a_1,a_2,\delta,\ell)$ of $v$ satisfies the following linear
constrains:
   \begin{gather}
   0\le a_1\le c_1;\qquad 0\le a_2\le c_2; \label{eq:constr1} \\
   -a_2,\/\/\/-c_2+a_2\le\delta\le c_1-a_1; \label{eq:constr2} \\
   0\le \ell\le c_2+2\delta. \label{eq:constr3}
   \end{gather}
Here~\refeq{constr1} is clear. When $\delta\ge 0$, the right inequality
in~\refeq{constr2} follows from $a_1+\delta=a_3\le c_1$; cf.~\refeq{a-delta}.
The upper subcrystal $\Kup[a]$ has parameter $\parup$ with
  \begin{equation} \label{eq:parup3}
  \parup_1=c_1-a_1+a_2, \qquad\parup_2=c_2-a_2+a_3
  \end{equation}
(by~\refeq{par_up}) and heart coordinate $\heartup$ with
  \begin{equation} \label{eq:heart3}
\heartup_1=a_2, \qquad \heartup_2=a_3
  \end{equation}
(by~\refeq{heart_up}). When $\delta\le 0$, we obtain
$-\delta=-\Delta_1\le\heartup_1=a_2$ and
$-\delta=\Delta_2\le\parup_2-\heartup_2= c_2-a_2$, yielding the left
inequalities in~\refeq{constr2}. Finally, for the middle subcrystal (path) $P$,
its parameter (length) $\parmid_2$ and coordinate $\heartmid_2$ of its heart
$z$ are:
  \begin{equation} \label{eq:par-centr-P}
\parmid_2=c_2-\Delta_2+\Delta_1=c_2+2\delta, \qquad \heartmid_2=a_2+\Delta_1=a_2+\delta
  \end{equation}
(by~\refeq{par_mid} and~\refeq{heart_mid}). The first relation
in~\refeq{par-centr-P} yields~\refeq{constr3}.

Conversely, let $a_1,a_2,\delta,\ell$ be integers
satisfying~\refeq{constr1}--\refeq{constr3}. Put $a_3:=a_1+\delta^+$,
~$a:=(a_1,a_2,a_3)$ and $b:=(a_3,a_2,a_1)$.
Comparing~\refeq{constr1},\refeq{constr2} with~\refeq{parup3},\refeq{heart3},
one can conclude that the upper lattice $\Piup[a]$ contains a vertex whose
deviation equals $(\delta,-\delta)$. Moreover, this vertex
$\vup[a|(\delta,-\delta)]=:z$ coincides with the vertex
$\vlow[b|(-\delta,\delta)]$ in the lower lattice $\Pilow[b]$ and is
self-complementary. Then all vertices of the middle subcrystal $P$ containing
$z$ are self-complementary. Now comparing~\refeq{constr3}
with~\refeq{par-centr-P}, we can conclude that $P$ has a vertex $v$ whose
coordinate equals $\ell$. This $v$ is just the self-complementary vertex
corresponding to $a_1,a_2,\delta,\ell$.

Thus, we obtain the following
  \begin{theorem} \label{tm:S-constr}
~For each integer solution $(a_1,a_2,\delta,\ell)$
to~\refeq{constr1}--\refeq{constr3}, there is a vertex $v\in S$ such that
$(a_1,a_2,\delta,\ell)=(a_1(v),a_2(v),\delta(v),\ell(v))$, and vice versa.
\hfill\qed
  \end{theorem}

This correspondence is crucial in our construction of worms for the elements of
$S$. It is convenient to consider the prism $\{a\in\Rset^3\colon 0\le a_1\le
a_3\le c_1, ~0\le a_2\le c_2\}$; the integer points in it are exactly the
coordinates of principal vertices $\prv[a]$ of $K$ with $a_1\le a_3$. The
ground rectangle for the worms that we construct is identified with the facet
$\Phi$ of the prism formed by the points $a$ satisfying $a_1=a_3$. We modify
the coordinates on $\Phi$ by $(a_1,a_2,a_3=a_1)\mapsto (2a_1,a_2)$; then the
first coordinate runs from $0$ to $2c_1$, and the second from $0$ to $c_2$.

Consider $v\in S$ and let $a,\delta,\ell$ stand for $a(v),\delta(v),\ell(v)$,
respectively. The desired worm $w=w(v)=(X',X'',Y',Y'')$ on $\Phi$ is assigned
by the following three rules (see Fig.~\ref{fig:wormsB} where the corresponding
worms are drawn in bold and $\delta\ne 0$). We denote the $\ell_1$-distance of
points $A,A'$ by $\Vert AA'\Vert$.
  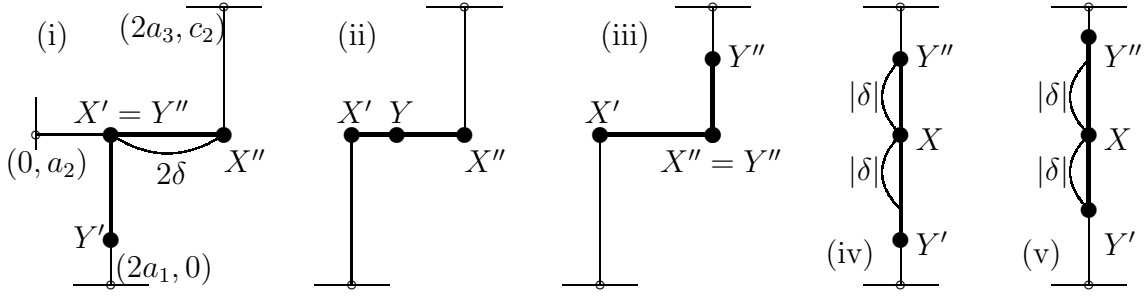
\begin{figure}[hbt]                  % Fig.5
  \begin{center}
  \unitlength=1mm
  \begin{picture}(145,35)
                         %              (i)
   \put(10,-3){\begin{picture}(25,37)
   \put(0,0){\circle{1}}
   \put(15,37){\circle{1}}
   \put(-10,20){\circle{1}}
   \put(0,20){\circle*{2}}
   \put(15,20){\circle*{2}}
   \put(0,6){\circle*{2}}
  \put(0,0){\line(0,1){20}}
  \put(15,20){\line(0,1){17}}
  \put(-5,0){\line(1,0){10}}
  \put(10,37){\line(1,0){10}}
  \put(-10,18){\line(0,1){7}}
  \put(-10,20){\line(1,0){10}}
   \put(-5,21.5){$X'=Y''$}
  \put(-5,5){$Y'$}
  \put(15,15){$X''$}
 \put(-14,15){$(0,a_2)$}
  \put(0.5,1){$(2a_1,0)$}
  \put(1,33){$(2a_3,c_2)$}
  \put(-10,32){(i)}
{\thicklines
  \put(0,20.1){\line(1,0){15}}
  \put(0.1,6){\line(0,1){14}}
  \put(0,19.9){\line(1,0){15}}
  \put(-0.1,6){\line(0,1){14}}
 }
 \qbezier(0,20)(7.5,15)(15,20)
 \put(6,13.5){$2\delta$}
     \end{picture}}
%
                         %              (ii)
   \put(42,-3){\begin{picture}(25,37)
   \put(0,0){\circle{1}}
   \put(15,37){\circle{1}}
   \put(0,20){\circle*{2}}
   \put(15,20){\circle*{2}}
   \put(6,20){\circle*{2}}
  \put(0,0){\line(0,1){20}}
  \put(15,20){\line(0,1){17}}
  \put(-5,0){\line(1,0){10}}
  \put(10,37){\line(1,0){10}}
  \put(-2,21.5){$X'$}
  \put(5,21.5){$Y$}
  \put(15,15){$X''$}
    \put(-2,32){(ii)}
{\thicklines
  \put(0,20.1){\line(1,0){15}}
  \put(0,19.9){\line(1,0){15}}
 }
    \end{picture}}
%
                         %              (iii)
   \put(75,-3){\begin{picture}(25,37)
   \put(0,0){\circle{1}}
   \put(15,37){\circle{1}}
   \put(0,20){\circle*{2}}
   \put(15,20){\circle*{2}}
   \put(15,30){\circle*{2}}
  \put(0,0){\line(0,1){20}}
  \put(15,20){\line(0,1){17}}
  \put(-5,0){\line(1,0){10}}
  \put(10,37){\line(1,0){10}}
  \put(-2,21.5){$X'$}
  \put(17,29){$Y''$}
  \put(8,15){$X''=Y''$}
  \put(0,32){(iii)}
   \put(15,20){\circle*{2}}
{\thicklines
  \put(0,20.1){\line(1,0){15}}
  \put(15.1,20){\line(0,1){10}}
  \put(0,19.9){\line(1,0){15}}
  \put(14.9,20){\line(0,1){10}}
 }
    \end{picture}}
%
                         %              (iv)
   \put(115,-3){\begin{picture}(10,37)
   \put(0,0){\circle{1}}
   \put(0,37){\circle{1}}
   \put(0,20){\circle*{2}}
   \put(0,6){\circle*{2}}
   \put(0,30){\circle*{2}}
  \put(0,0){\line(0,1){37}}
  \put(-5,0){\line(1,0){10}}
  \put(-5,37){\line(1,0){10}}
  \put(2,4){$Y'$}
  \put(2,18){$X$}
  \put(2,29){$Y''$}
  \put(-10,3){(iv)}
{\thicklines
  \put(0.1,6){\line(0,1){24}}
  \put(-0.1,6){\line(0,1){24}}
 }
 \qbezier(0,20)(-5,25)(0,30)
 \put(-7,24){$|\delta|$}
 \qbezier(0,10)(-5,15)(0,20)
 \put(-7,14){$|\delta|$}
    \end{picture}}
%
                         %              (v)
   \put(140,-3){\begin{picture}(10,37)
   \put(0,0){\circle{1}}
   \put(0,37){\circle{1}}
   \put(0,20){\circle*{2}}
   \put(0,10){\circle*{2}}
   \put(0,33){\circle*{2}}
  \put(0,0){\line(0,1){37}}
  \put(-5,0){\line(1,0){10}}
  \put(-5,37){\line(1,0){10}}
  \put(2,4){$Y'$}
  \put(2,18){$X$}
  \put(2,29){$Y''$}
  \put(-9,3){(v)}
{\thicklines
  \put(0.1,10){\line(0,1){23}}
  \put(-0.1,10){\line(0,1){23}}
 }
 \qbezier(0,20)(-5,25)(0,30)
 \put(-7,24){$|\delta|$}
 \qbezier(0,10)(-5,15)(0,20)
 \put(-7,14){$|\delta|$}
     \end{picture}}
 \end{picture}
 \end{center}
  \caption{(i) $\delta>0$, ~$\ell<a_2$; ~(ii) $\delta>0$, ~$a_2< \ell< a_2+2\delta$;
~(iii) $\delta>0$, ~$\ell>a_2+2\delta$; ~(iv) $\delta<0$, ~$\ell<a_2+\delta$;
~(v) $\delta<0$, ~$\ell>a_2+\delta$.}
  \label{fig:wormsB}
  \end{figure}
  \begin{numitem1}
The horizontal limb of $w$ connects the points $X':=(2a_1,a_2)$ and
$X'':=(2a_3,a_2)$ (degenerating into the single point $X=(2a_1,a_2)$ when
$\delta\le 0$; cf.~\refeq{a-delta}).
   \label{eq:rule1}
   \end{numitem1}
  \begin{numitem1}
Let $\delta\ge 0$. Then (see Fig.~\ref{fig:wormsB}(i),(ii),(iii)):
  \begin{itemize}
 \item[(i)]
if $\ell<a_2$, then $w$ is the VH-worm in which the lower limb connects the
points $Y':=(2a_1,\ell)$ and $Y''=X'=(2a_1,a_2)$;
 \item[(ii)]
if $a_2\le \ell\le a_2+2\delta$, then $w$ is the H-worm in which the point $Y$
is located at $(2a_1+\ell-a_2,\,a_2)$;
 \item[(iii)]
if $\ell>a_2+2\delta$, then $w$ is the HV-worm in which the upper limb connects
the points $Y'=X''=(2a_3,a_2)$ and $Y'':=(2a_3,\,\ell-2\delta)$.
  \end{itemize}
   \label{eq:rule2}
  \end{numitem1}
  \begin{numitem1}
Let $\delta\le 0$. Then $w$ is the V-worm and (see
Fig.~\ref{fig:wormsB}(iv),(v)):
  \begin{itemize}
 \item[(i)] if $\ell\le a_2+\delta$, then $Y':=(2a_1,\ell)$ and $Y'':=(2a_1,\,a_2+|\delta|)=X+(0,|\delta|)$;
 \item[(ii)] if $\ell> a_2+\delta$, then $Y':=(2a_1,\,a_2+\delta)=X+(0,\delta)$ and
 $Y'':=(2a_1,\,\ell+2|\delta|)$.
  \end{itemize}
  \label{eq:rule3}
  \end{numitem1}

First of all we have to check that each worm constructed
by~\refeq{rule1}--\refeq{rule3} is well-defined and their set is complete
(coincides with the set of worms in $W(\bar c)$).

By~\refeq{rule1}, both points $X',X''$ lie in $\Phi$ and their first
coordinates are even (as required in~\refeq{admiss}(i)). Also any horizontal
line-segment in $\Phi$ connecting points $(x',y)$ and $(x'',y)$ with $x',x''$
even and $y$ integer is present as the horizontal limb of $w(v)$ for some $v\in
S$ (namely, with $a_1(v)=x'/2$, $a_2(v)=y$, $a_3(v)=x''/2$). To verify other
properties, it is useful to partition $S$ into subsets ({\em groups})
$S(a,\delta)$, each depending on a pair $(a,\delta)$ as in~\refeq{a-delta} and
consisting of all $v\in S$ such that $a(v)=a$ and $\delta(v)=\delta$. That is,
$S(a,\delta)$ is formed by the vertices of the corresponding middle subcrystal,
denoted as $P(a,\delta)$; so $P(a,\delta)=P(v)$ for all $v\in S(a,\delta)$. We
consider two cases.
\medskip

{\em Case I}: $\delta\ge 0$. ~Define $Q=Q(a,\delta)$ to be the union of the
vertical line-segments $Y_0X'$ and $X''Y_1$ and the horizontal line-segment
$X'X''$, where $Y_0:=(2a_1,0)$ and $Y_1:=(2a_3,c_2)$. Note that the sum of
lengths of these segments is $a_2+(2a_3-2a_1)+(c_2-a_2)=c_2+2\delta$ (in view
of $a_3-a_1=\delta$); this is equal to the length $\parmid_2$ of the path
$P=P(a,\delta)$ (cf.~\refeq{par-centr-P}). Regarding $Q$ as the corresponding
path from $Y_0$ to $Y_1$, we can identify it with $P$.
Comparing~\refeq{constr3} with~\refeq{rule2}, we observe that: the first vertex
of $P$ (where $\ell=0$) is identified with $Y_0$, and the last vertex of $P$
(where $\ell=c_2+2\delta$) with $Y_1$. For the first vertex, the arising worm
$w$ has the $Y'$ point at $Y_0$ and is the largest VH-worm for $S(a,\delta)$,
whereas for the last vertex, $w$ has the $Y''$ point at $Y_1$ and is the
largest HV-worm for $S(a,\delta)$. When moving along $P$ step by step, the
current worm $w=w(v)$ evolves as follows: while $\ell(v)<a_2$, ~$w$ is a
VH-worm whose lower limb $Y'X'$ shortens by 1 at each step; while $a_2\le
\ell(v)<a_2+2\delta$, ~$w$ is an H-worm in which the point $Y$ shifts to the
right by 1 at each step; and while $a_2+2\delta\le \ell(v)<c_2+2\delta$, ~$w$
is an HV-worm whose upper limb $X''Y''$ increases by 1 at each step. This
behavior matches the action of operator $\tilde\bftwo$ on the worm graph.
\medskip

 %\noindent
{\em Case II}: $\delta\le 0$. ~Let $Y_0:=(2a_1,0)$ and $Y_1:=(2a_1,c_2)$ (cf.
the previous case) and define $\tilde Y_0:=(2a_1,a_2+\delta)$ and $\tilde
Y_1:=(2a_1,a_2+|\delta|)$. (Recall that $\delta\le 0$ implies $a_1=a_3$.)
Comparing~\refeq{constr3} with~\refeq{rule3}, we observe that: for the first
vertex of $P=P(a,\delta)$ (where $\ell=0$), the arising worm $w$ has the $Y'$
point at $Y_0$ and the $Y''$ point at $\tilde Y_1$, whereas for the last vertex
of $P$ (where $\ell=c_2+2\delta$), ~$w$ has $Y'$ at $\tilde Y_0$ and $Y''$ at
$Y_1$. When moving along $P$, the current H-worm evolves as follows: while
$\ell(v)<a_2+\delta$ (and therefore, $\Vert Y'X\Vert>|\delta|=\Vert
XY''\Vert$), the lower limb $Y'X$ shortens by 1 at each step and $Y''$ rests at
$\tilde Y_1$, and while $a_2+\delta\le \ell(v)<c_2+2\delta$ (and therefore,
$\Vert Y'X\Vert=|\delta|\le\Vert XY''\Vert$) , the upper limb $XY''$ increases
by 1 at each step and $Y'$ rests at $\tilde Y_0$. (Note that $\min\{\Vert
Y'X\Vert,\Vert XY''\Vert\}$ is invariant and equal to $|\delta|$.) Again, this
matches the action of $\tilde\bftwo$ on the worm graph.

Thus, we come to the following
  \begin{prop}  \label{pr:S-worm}
By the above construction, the correspondence $v\mapsto w(v)$ is a bijection
between the vertices of the symmetric extract $\frakB$ from $K(c)$ and the
vertices of the worm graph $W=W(\bar c_1,\bar c_2)$. Under this bijection, the
edges of second color $\bar 2$ of $\frakB$ are transferred to the edges of
second color $\tilde 2$ of $W$.
  \end{prop}

%------------------
 \subsection{Verification of edges of color 1} \label{ssec:B3}

To finish the proof of Theorem~\ref{tm:B3-B4} it remains to show that under the
above correspondence $v\mapsto w(v)$, the edges of color $\bar 1$ in the
symmetric extract $\frakB$ from $K=K(c_1,c_2,c_1)$ are transferred one-to-one
to the edges of color $\tilde 1$ in the worm graph $W(c_1,c_2)$. This involves
additional ideas and technical tools.
   \begin{prop} \label{pr:edg_color1}
For each $v\in S$, the following properties hold:
  \begin{itemize}
\item[\rm(i)] if $\tilde\bfone$ does not act at the worm $w(v)$, then
$\bar\bfone$ does not act at $v$;
 \item[\rm(ii)] if $\tilde\bfone$ acts at $w(v)$, then $\bar\bfone$ acts at $v$
and $w(\bar\bfone v)=\tilde\bfone w(v)$.
\end{itemize}
  \end{prop}

   \begin{proof}
We will use induction on the \emph{length} of $w(v)$, which is defined below.
When needed, handling one or another object related to $v$ (or another vertex
in $S$), we will include $v$ as argument in corresponding notation (on the
other hand, we often omit $v$ when it is clear from the context). We associate
to $v$ the integers $a_1(v),\,a_2(v),\,a_3(v),\,\delta(v),\,\ell(v)$ as before.
Considering the worm $w(v)$ in the form of six-tuple $(x'(v),y(v),x''(v);
y'(v),x(v),y''(v))$ as defined in Section~\ref{sec:worm}, we introduce the
following values:
   \begin{gather}
p_1(v):=x'(v)/2,\qquad p_2(v):=y'(v), \qquad q_1(v):= x''(v)/2, \label{eq:ppqq} \\
 q_2(v):=y''(v) \quad \mbox{and} \quad \eta(v):= q_1(v)+q_2(v)-p_1(v)-p_2(v).
 \nonumber
  \end{gather}
Then $p_i\le q_i$, $i=1,2$, and the corresponding points $X'(v),\, X''(v),\,
Y'(v),\, Y''(v)$ are located within the rectangle $R_v:=
\{(\alpha,\beta)\in\Rset^2 \colon 2p_1(v)\le\alpha\le 2q_1(v),\, p_2\le\beta\le
q_2(v)\}$. Moreover, at least one of $X'(v),\,Y'(v)$ lies at the south-west
corner $(2p_1(v),p_2(v))$ of $R_v$ and at least one of $X''(v),\,Y''(v)$ lies
at the north-east corner $(2q_1(v),q_2(v))$; one may say that $w(v)$
\emph{spans} $R_v$. We also call $R_v$ the \emph{domain} of $w(v)$. It
degenerates into a horizontal segment (a vertical segment, a single point) when
$w(v)$ is an H-worm (resp. a V-worm, a principal point).

In addition, extending $p(v)$ and $q(v)$ to the self-complementary triples
$\hat p=\hat p(v):=(p_1(v),p_2(v),p_1(v))$ and $\hat q=\hat
q(v):=(q_1(v),q_2(v),q_1(v))$, we consider the self-complementary vertices
$\prv[\hat p]$ and $\prv[\hat q]$ in the principal lattice $\Pi$ of $K$ and
define the graph $B_v$ to be the interval of $\frakB$ from $\prv[\hat p]$ to
$\prv[\hat q]$. We will use the following easy corollary from
Proposition~\ref{pr:int_prlat}.
   \begin{corollary} \label{cor:Bv}
~$B_v$ is isomorphic to $\frakB(q(v)-p(v))$ (the symmetric extract from $K(\hat
q(v)-\hat p(v))$. \hfill\qed
 \end{corollary}

The number $\eta(v)$ in~\refeq{ppqq} is just what we call the \emph{length} of
$w(v)$. We assume by induction that the required  properties~(i),(ii) are valid
for each $v'\in S$ with $\eta(v')<\eta(v)$. When $\tilde\bfone$ acts at $w(v)$,
we denote by $u$ the element of $S$ such that $w(u)=\tilde\bfone w(v)$
(existing by Proposition~\ref{pr:S-worm}). (So our goal in this case is to show
that $u=\bar\bfone v$.) From the description of the worm model one can see that
for the domain $R_u$, only three situations are possible: (a) $R_u\subset R_v$
(where the inclusion is strict), (b) $R_u\supset R_v$, and (c) $R_u=R_v$. The
first case is easy. \medskip

   \noindent\textbf{Claim 1}
~{\em {\rm(i)} If $R_u\subset R_v$, then $\bone$ acts at $v$ and $u=\bar\bfone
v$. ~{\rm(ii)} Suppose $\bone$ acts at $v$ and let $v':=\bone v$. If
$\eta(v')<\eta(v)$, then $\tone$ acts at $w(v)$ and $u=v'$.}
\medskip

 \begin{proof}
~(i) Clearly $R_u\subset R_v$ implies $\eta(u)<\eta(v)$. Applying the induction
to the vertex $u$ in the reversed graph $\frakB^{rev}$ (which is isomorphic to
$\frakB$), one can conclude that the operator reverse to $\bar\bfone$ transfers
$u$ to $v$. Since $\bar\bfone$ is invertible, $u=\bar\bfone v$.

Part (ii) is proved in a similar way. \hfill
 \end{proof}

The situation $R_u\supseteq R_v$ is less trivial. We examine three cases.
\medskip

 \noindent
\underline{\em Case 1}: ~$w(v)$ is an H-worm, i.e. $p_2=q_2=:r$ (hereinafter
$p_i$ stands for $p_i(v)$, and $q_i$ for $q_i(v)$). Then $q_1-p_1=\delta\ge 0$,
~$X'(v)=(2p_1,r)$, ~$X''(v)=(2q_1,r)$ and $Y(v)=(x,r)$ for some $2p_1\le x\le
2q_1$. For $i=0,1,\ldots,2\delta$, let $v_i$ denote the vertex in $S$ such that
$w(v_i)$ is the H-worm with $X'(v_i)=X'(v)$, $X''(v_i)=X''(v)$ and
$Y(v_i)=(2p_1+i,r)$. These worms have the same domain, and $v=v_j$ for some
$j$.

If $j>\delta$, then $\Vert X'(v)Y(v)\Vert=j>2\delta-j=\Vert Y(v)X''(v)\Vert$;
therefore (cf.~\refeq{edges1}) $\tilde\bfone$ acts at $w(v)$ and moves the $X'$
point by two units to the right. Then $R_u\subset R_v$, and we are done by
Claim~1(i).

Now let $j\le\delta$. We rely on the following claim; it will be proved in the
Appendix.\medskip

   \noindent\textbf{Claim 2}
~{\em When $j\le\delta$, operator $\bar\bfone$ acts at $v$ if and only if $q_1<
c_1$.}
\medskip

If $q_1=c_1$ (and therefore $X''(v)$ lies on the right boundary of the entire
region $R(c_1,c_2)$), then $\tilde\bfone$ does not act at $w(v)$. By Claim~2,
$\bar\bfone$ does not act at $v$ as well, and we obtain~(i) in the proposition.

Thus, we may assume that $q_1<c_1$. This and $\Vert X'(v)Y(v)\Vert=j \le
2\delta-j=\Vert Y(v)X''(v)\Vert$ imply that $\tilde\bfone$ acts at $w(v)$ and
moves the $X''$ point by two units to the right. Then $w(u)=\tilde\bfone w(v)$
is the H-worm with $X'(u)=X'(v)$, ~$Y(u)=Y(v)$ and
$X''(u)=X''(v)+(2,0)=(2q_1+2,\,r)$. In particular, $R_u\supset R_v$.

By Claim~2, $\bar\bfone$ acts at $v$; let $v':=\bar\bfone v$. Notice that the
assertion in this claim depends on $\delta-j$ and $c_1-q_1$, but not on
$p_1,q_1,r,c_1$. So we can apply it to the subgraph $B_u$ of $\frakB$ (the
interval between the principal vertices $\prv[(p_1,r,p_1)]$ and
$\prv[(q_1+1,r,q_1+1)]$), appealing to Corollary~\ref{cor:Bv}. This implies the
important fact that the \emph{vertex $v'$ belongs to $B_u$}, whence $w(v')$ is
an H-worm with $R_{v'}\subseteq R_u$. Let
  $$
  X'(v')=(2p'_1,r),\quad X''(v)=(2q'_1,r) \quad\mbox{and}\quad Y(v')=(2p_1+j',r).
  $$
We have to show that $v'=u$, i.e.
  \begin{equation} \label{eq:ppqpjp}
  p'_1=p_1,\quad q'_1=q_1+1 \quad\mbox{and}\quad j'=j.
  \end{equation}

To show this, we will construct certain routes in $B_u$ and appeal to the fact
that $B_u$ is graded (i.e. for any route $P$ in $B_u$ and each color $\zeta$,
the difference $k_\zeta(P)$ between the number of forward and backward edges of
color $\zeta$ in $P$ depends only on the beginning and end of $P$; cf.~1 in
Section~\ref{sec:Bn}). (Hereinafter by a \emph{route} we mean a path with
possible backward edges.)

The case $\eta(v')<\eta(v)$ is impossible, by Claim~1(ii) and the inclusion
$R_u\supset R_v$. Hence either (a) $p'_1=p_1$ and $q'_1=q_1$, or (b)
$p'_1=p_1+1$ and $q'_1=q_1+1$, or (c $p'_1=p_1$ and $q'_1=q_1+1$. First of all
we exclude (a) and (b). \medskip

1) Suppose $p'_1=p_1$ and $q'_1=q_1$. Then the worm $w(v')$ is obtained from
$w(v)$ by moving the $Y$ point by $j'-j$ units (to the right or left depending
on the sign of $j'-j$). Hence $w(v')=\tilde\bftwo^{j'-j}w(v)$, implying
$v'=\bar\bftwo^{j'-j}v$ (since the ``operator'' $w$ and the second color
commute, by Proposition~\ref{pr:S-worm}). The latter is impossible since
$v'=\bar\bfone v$ and $B_u$ is graded. \medskip

2) Suppose $p'_1=p_1+1$ and $q'_1=q_1+1$. We can find in $B_u$ a route $P_1$
from $\check v_0:=\prv[(p_1,r,p_1)]$ to $v$ and a route $P_2$ from $v$ to
$\check v_1:= \prv[(q_1+1,r,q_1+1)]$ such that
   \begin{equation} \label{eq:kP1P2}
k_{\bar\bfone}(P_1)=\delta,\quad k_{\bar\bftwo}(P_1)=j,\quad
k_{\bar\bfone}(P_2)=\delta+2,\quad k_{\bar\bftwo}(P_2)=2\delta+2-j.
  \end{equation}
To get $P_1$, we construct a sequence of worms as follows. Starting with the
principal worm $w(\check v_0)$, we first move the $X''$ point (by two units per
move) $\delta$ times, obtaining $X''=(2p_1+2\delta,\, r)$. Then move $Y$ (by
one unit per move) $j$ times, obtaining $Y=(2p_1+j,r)$. And to get $P_2$, we
start with $w(v)$ and first make $2\delta-j$ moves with $Y$ (obtaining
$Y=X''=(2q_1,r)$), then $q_1$ moves with $X'$ (obtaining $X'=(2q_1,r)$), and
finally one move with $X''$, two moves with $Y$ and one move with $X'$. This
results in $X'=X''=Y=(2q_1+2,\,r)$, giving the worm $w(\check v_1)$. Each of
these moves on worms induces passing through the corresponding edge in $B_u$:
this is clear when operator $\tilde\bftwo$ applies and follow by the induction
when $\tilde\bfone$ applies (since at least one of the two worms involved in
the operation has the length less than $\eta(v)$). This gives $P_1,P_2$
satisfying~\refeq{kP1P2}. (Note that the construction of $P_2$ remains correct
when $\eta(v)=0$, i.e. $p_1=q_1$ and $v=\check v_0$. In this case there is a
unique directed path $P$ from $\check v_0$ to $\check
v_1=\prv[(p_1+1,r,p_1+1)]$ in $\frakB$; see the illustrations of $K(1,0,1)$ and
$B(1,0)$ in Section~\ref{sec:Bn}. It corresponds to the relation $w(\check
v_1)=\tilde\bfone\tilde\bftwo\tilde\bftwo \tilde\bfone w(\check v_0)$ in $W$.)

Using a similar procedure, we construct a route $Q_2$ from $v'$ to $\check v_1$
such that
  $$
k_{\bar\bfone}(Q_2)=\delta \quad\mbox{and}\quad
k_{\bar\bftwo}(Q_2)=2\delta+2-j'.
  $$
Then concatenating the route $P_1$, the move by $\bar\bfone$ from $v$ to $v'$,
and the route $Q_2$, we obtain a route $P'$ from $\check v_0$ to $\check v_1$
for which $k_{\bar\bfone}(P')=\delta+1+\delta$. On the other hand, the path
from $\check v_0$ to $\check v_1$ that is the concatenation of $P_1$ and $P_2$,
denoted as $P_1\cdot P_2$, gives $k_{\bar\bfone}(P_1\cdot P_2)=2\delta+2$. A
contradiction with the gradedness of $B$.
\medskip

3) Now let $p'_1=p_1$ and $q'_1=q_1+1$. To show the desired equality $j'=j$,
consider the vertex $z$ in $B_u$ such that $w(z)$ is the H-worm with
   $$
   X'(z)=(2p_1+2,\,r),\quad X''(z)=(2q_1+2,\,r)\quad\mbox{and}\quad
   Y(z)=(2q_1+2-j,\,r).
   $$
Then the isomorphism between $B_u^{rev}$ and $B_u$ swaps $z$ and $v$ (as well
as $\check v_0$ and $\check v_1$). This implies that $\bar\bfone^{-1}$ acts at
$z$ and the worm of $z':=\bar\bfone^{-1}z$ should be ``symmetric'' to the worm
of $v'$, namely:
  $$
  X'(z')=(2p_1,\,r),\quad X''(z')=(2q_1+2,\,r)\quad\mbox{and} \quad
  Y(z')=(2q_1+2-j',\,r).
  $$
Then $w(v')$ is transformed into $w(z')$ by moving the $Y$ point by
$(2\delta+2-j')-j'=2\delta+2-2j'$ units. Now concatenating the above route
$P_1$ from $\check v_0$ to $v$, the $\bar 1$-edge from $v$ to $v'$, the
corresponding route from $v'$ to $z'$, the $\bar 1$-edge from $z'$ to $z$, and
the route from $z$ to $\check v_1$ that is ``symmetric'' to $P_1$, we obtain a
route $Q$ from $\check v_0$ to $\check v_1$ such that
  $$
  k_{\bar\bftwo}(Q)=j+(2\delta+2-2j')+j=2\delta+2+2j-2j'.
  $$
Since $B$ is graded, we must have $k_{\bar\bftwo}(Q)=k_{\bar\bftwo}(P_1)
+k_{\bar\bftwo}(P_2)=2\delta+2$ (cf.~\refeq{kP1P2}). Hence $j'=j$,
yielding~\refeq{ppqpjp}. Thus, we obtain~(ii) in the proposition. \medskip

 \noindent
 \underline{\em Case 2}: ~$w(v)$ is a V-worm with $X(v)\ne Y''(v)$.
Then $R_v$ is the vertical segment connecting the points $Y'(v)=(f,p_2)$ and
$Y''(v)=(f,q_2)$, where $f:=p_1=q_1$. Also $X(v)=(f,p_2+j)$ for some $0\le
j<h:=q_2-p_2$. For $i=0,\ldots,h$, let $z_i$ denote the vertex of $\frakB$ such
that $R_{z_i}=R_v$ and $X(z_i)=(f,p_2+i)$; then $v=z_j$.

By~\refeq{edges1}, for $i=0,\ldots,h-1$, the worm $w(z_i)$ is transformed by
$\tone$ into the worm $w(z_{i+1})$. Define $z'_{i}:=\bone z_i$. Our goal is to
show that $z'_i=z_{i+1}$  for $i=0,\ldots,h-1$. We rely on the following claim
which will be proved in the Appendix. \medskip

   \noindent\textbf{Claim 3}
~{\em Operator $\bone$ acts at each $z_i$ with $0\le i<h$.}
\medskip

This claim does not impose any conditions on $f,p_2,c_1,c_2$; so it is
applicable to the subgraph $B_v$ of $\frakB$ (taking into account
Corollary~\ref{cor:Bv}). This implies $R_{z'_{i}}\subseteq R_{z_i}=R_v$ for
each $i<h$. The strict inclusion here is excluded by Claim~1(ii). Hence
$R_{z'_{i}}=R_v$, implying $z'_i=z_{i'}$ for some $i'\in\{0,\ldots,h\}$.

To show the desired equality $i'=i+1$ for each $i$, we use the next claim,
denoting the principal vertices $\prv[(f,p_2,f)]$ and $\prv[(f,q_2,f)]$ by
$\check v_0$ and $\check v_1$, respectively.
\medskip

   \noindent\textbf{Claim 4}
~{\em Let $i\ne h/2$. Let $P_1$ be a route (in $\frakB$) from $\check v_0$ to
$z_i$, and $P_2$ a route from $z_i$ to $\check v_1$. Then $k_{\bone}(P_1)=i$
and $k_{\bone}(P_2)=h-i$. In particular, $k_{\bone}(P_1\cdot P_2)=h$.}
\medskip

 \begin{proof}
~We construct a sequence of worms starting with $w(z_i)$ and ending with
$w(\check v_1)$ (aiming to obtain $P_2$ as above). Two cases are possible: (a)
$i>h/2$, and (b) $i<h/2$. In case~(a), we first apply to $w(z_i)$ operator
$\ttwo$ which, in view of $\Vert Y'(z_i)X(z_i)\Vert=i>h-i=\Vert
X(z_i)Y''(z_i)\Vert$, moves $Y'$ by one unit up (obtaining $Y'=(f,p_2+1)$).
Then we make $h-i$ moves with $X$ (obtaining $X=(f,q_2)$), followed by $h-1$
moves with $Y'$. This results in $Y'=Y''=X=(f,q_2)$. And in case~(b), we first
apply operator $\ttwo^{-1}$ which, in view of $\Vert Y'(z_i)X(z_i)\Vert< \Vert
X(z_i)Y''(z_i)\Vert$, moves $Y''$ by one unit down (obtaining $Y''=(f,q_2-1)$).
Then we make $h-i-1$ moves with $X$ (obtaining $X=(f,q_2-1)$), followed by
$h-1$ moves with $Y'$ (obtaining $Y'=(f,q_2-1)$), and finally make one move
with each of $Y'',X,Y'$ (in this order); this results in $Y'=Y''=X=(f,q_2)$.

Since every time we either apply operator $\ttwo$ or $\ttwo^{-1}$, or apply
$\tone$ to a worm $w(z')$ with $R_{z'}\subset R_v$, we can conclude that in
each case the constructed sequence of worms induces a route $P$ from $z_i$ to
$\check v_1$ in $B_v$. The equality $k_{\bone}(P)=h-i$ follows by counting the
number of applications of $\tone$ in the sequence.

A route $P'$ from $\check v_0$ to $z_i$, giving $k_{\bone}(P')=i$, is
constructed in a similar way (by applying the above procedure to $z_i$ and
$\check v_0$ in $B_v^{rev}$). \hfill
\end{proof}

Consider $i\ne h/2,h$. If $i'\ne h/2$, take a route $P_1$ from $\check v_0$ to
$z_i$ and a route $P'_2$ from $z_{i'}$ to $\check v_1$. Let $Q$ be the
concatenation of $P_1$, the $\bar 1$-edge from $z_i$ to $z_{i'}$, and $P'_2$.
By Claim~4, $k_{\bone}(P_1)=i$ and $k_{\bone}(P'_2)=h-i'$; therefore,
$k_{\bone}(Q)=i+1+h-i'$. On the other hand, $k_{\bone}(Q)$ must be equal to $h$
(since $\frakB$ is graded and $k_{\bone}(P_1\cdot P_2)=h$). This implies
$i'=i+1$, as required.

Now let $i=h/2-1$ (when $h$ is even). Then the case $i'=h/2$ ($=i+1$) is only
possible (since $i'\ne h/2$ would imply $i'\ne i+1$, leading to a contradiction
with the argument above). Finally, let $i=h/2$. Considering $B_v^{rev}$ and
$i''=h/2+1$, we have $\bone^{-1} z_{i''}=z_{h/2}$. Hence $i'=i''=h/2=i+1$, as
required.

Thus, (ii) in the proposition is valid for each $v=z_i$ with $i<h$. \medskip

\noindent \underline{\em Case 3}: ~$w(v)$ is a proper VH-worm or a proper
HV-worm or a V-worm with $X(v)=Y''(v)\ne Y'(v)$. (This is the simplest case in
our analysis.) When $w(v)$ is an HV-worm, we have $p_1<q_1$,
~$X'(v)=(2p_1,p_2)$ and $X''(v)=Y'(v)=(2q_1,p_2)$. Then operator $\tone$ moves
$X'$ to $(2p_1+2,\,p_2)$; this gives $R_u\subset R_v$, and (ii) in the
proposition follows from Claim~1(i).

So we may assume that $w(v)$ is the VH-worm with $Y'(v)=(2p_1,p_2)$,
~$Y''(v)=X'(v)=(2p_1,q_2)$ and $X''(v)=(2q_1,q_2)$; it degenerates into a
V-worm when $p_1=q_1$. The following claim will be proved in the Appendix.
\medskip

   \noindent\textbf{Claim 5}
~{\em When $w(v)$ is a VH-worm, $\bone$ acts at $v$ if and only if $q_1<c_1$.}
\medskip

In case $q_1=c_1$, we have (i) in the proposition.

Let $q_1<c_1$. Then $\tone$ acts at $w(v)$ and moves $X''$ by two units to the
right, making the VH-worm $w(u)$ with $Y'(u)=Y'(v)$ and
$X''(u)=(2q_1+2,\,q_2)$. Claim~5 is applicable to $B_u$, whence the vertex
$v':=\bone v$ satisfies $R_{v'}\subseteq R_u$. The case $R_{v'}\ne R_u$ is
excluded (by repeating some reasonings from Case~1 and using Claim~1(ii)).

Let $R_{v'}=R_u$. Then $w(v')$ is either the VH-worm $w(u)$ (yielding (ii) in
the proposition), or the HV-worm with $X'(v')=(2p_1,p_2)$ and
$Y''(v')=(2q_1+2,q_2)$. One easily shows that the latter is impossible.

Thus, the proposition is valid in all cases. \hfill\qed
   \end{proof}

This completes the proofs of Theorems~\ref{tm:B3-B4} and~\ref{tm:A-B}. \medskip

  \noindent
{\bf Remark 5.} Return to a symmetric $A_{2n-1}$-crystal $K=K(c)$ and its
symmetric extract $B=B(\bar c)$. Let $S\Bscr(c)$ be the set of tuples
$a\in\Bscr(c)$ satisfying $a_{2n-i}=a_i$. Then $S\Pi:=\{\prv[a]\colon a\in
S\Bscr(c)\}$ is the set of self-complementary principal vertices in $K$. The
projection $a\mapsto (a_1,\ldots,a_n)=:\bar a$  gives a bijection between
$S\Pi$ and the integer $n$-box $\Bscr(\bar c)$. Also for any $a,a'\in
S\Bscr(c)$ with $a\le a'$, the interval $I(a,a')$ of $K$ between the principal
vertices $\prv[a]$ and $\prv[a']$ is isomorphic to the symmetric
$A_{2n-1}$-crystal $K(a'-a)$; this implies that the symmetric extract from
$I(a,a')$ is isomorphic to the $B_n$-crystal $B(\bar a'-\bar a)$. Thus, $S\Pi$
possesses properties similar to~(P1)--(P2) mentioned in the Introduction for
A-crystals, due to which this set can be regarded as the \emph{principal
lattice} of the $B_{n}$-crystal $B$. Note, however, that $S\Pi$ need not
satisfy property~(P3). This is seen already for $n=2$. Indeed, in this case any
principal vertex $\prv$ is represented by a single point $(x,y)$ (a principal
worm) in the worm model; therefore, in the subcrystal of color 1 (color 2)
containing $\prv$, all vertices correspond to horizontal (resp. vertical) worms
covering $(x,y)$, and none of proper VH- or HV-worms can be used.

A similar construction of principal lattices can be given for C-crystals.

%---------------------------  SEC. 9
 \section{Deriving $C_n$-crystals from symmetric $A_{2n}$-crystals} \label{sec:A2n}

%\refstepcounter{section}

By an analogue with the construction and results in
Sections~\ref{sec:Bn}--\ref{sec:proofB3-B4}, we can construct and examine the
$n$-colored graphs being symmetric extracts from symmetric $A_{2n}$-crystals.
In this section we show that these graphs are $C_n$-crystals.

Using terminology and notation similar to those in Section~\ref{sec:Bn}, we
consider an $A_{2n}$-crystal $K=K(c)=(V(K),\, E_1\sqcup\ldots \sqcup E_{2n})$
with a parameter $c$ satisfying $c_i=c_{2n+1-i}$, denote by $\bar i$ the pair
of complementary colors $(i,\,i':=2n+1-i)$ for $i=1,\ldots,n$, and consider the
corresponding complementarity involution $\sigma: V(K)\to V(K)$ and the set $S$
of self-complementary vertices $v=\sigma(v)$ of $K$. When $i<n$, the colors $i$
and $i'$ are not neighboring, and, as before, we draw an edge of color $\bar i$
from a vertex $u\in S$ to a vertex $v\in S$ if and only if $v=F_iF_{i'}(u)$
($=F_{i'}F_i(u)$). The edges of color $\bar n=(n,\,n'=n+1)$ are assigned in a
different way: we draw an $\bar n$-edge from $u\in S$ to $v\in S$ if and only
if $v=F_nF_{n'}F_{n'}F_n(u)$ (see an explanation in Remark~6 below).

Let $\frakC=(S,\,E_{\bar 1}\sqcup\ldots\sqcup E_{\bar n})$ be the obtained
$n$-colored graph. One can see that a majority of reasonings of
Section~\ref{sec:Bn} remain applicable to our case, yielding common properties
of $\frakC$ and $\frakB$.
  \begin{prop} \label{pr:Gn}
The symmetric extract $\frakC$ from a symmetric $A_{2n}$-crystal $K(c)$
satisfies axioms (BC1) and (BC2) (from Section~\ref{ssec:typeBC}).
  \end{prop}

Thus, like the $A_{2n-1}$ case, the key problem is to characterize the
components of the two-colored subgraph of $\frakC$ with colors $\overline{n-1}$
and $\bar n$. This is equivalent to characterizing the extract $(S,E_{\bar
1}\sqcup E_{\bar 2})$ from a symmetric $A_4$-crystal $K(c)$ with colors
1,2,3,4. Our goal is to show the following
  \begin{theorem} \label{tm:C3-C4}
Let $K(c)$ be a symmetric $A_4$-crystal. Then the symmetric extract $\frakC$
from $K$ is the $C_2$-crystal with parameter $\bar c=(c_1,c_2)$ respecting the
Cartan coefficients $m_{12}=-1$ and $m_{21}=-2$.
  \end{theorem}
(Cf. Theorem~\ref{tm:B3-B4}.) The proof will consist of several stages, which
are analogous, to some extent, to those in Section~\ref{sec:proofB3-B4}. We
will identify the vertices of $\frakC$ with certain quadruples (somewhat
different from those in the $A_3\to B_2$ reduction), describe the polytope
spanned by these quadruples and show their one-to-one correspondence to the
vertices of the worm graph $W(c_2,c_1)$. This gives a counterpart of
Proposition~\ref{pr:S-worm} (but with colors $\tilde 1,\tilde 2$ swapped).
Theorem~\ref{tm:C3-C4} will be obtained by comparing the edges in both graphs.
As a result, the desired relation between A- and C-crystals follows.
  \begin{corollary} \label{cor:A-C}
\rm{(Cf.~\cite[Th.~3.2.4]{NS}.)} The symmetric extract from a symmetric
$A_{2n}$-crystal is a $C_n$-crystal, and any $C_n$-crystal is obtained in this
way.
  \end{corollary}

Figure~\ref{fig:C-cryst} illustrates the ``simplest''  symmetric $A_4$-crystals
$K(1,0,0,1)$ and its symmetric extract $C(1,0)$.

  \begin{figure}[hbt]                  % Fig.6
  \begin{center}
  \unitlength=1mm
  \begin{picture}(120,55)
                         %              (left)
   \put(10,0){\begin{picture}(70,60)
  \put(0,0){\circle*{1}}
  \put(12,-4){\circle*{1}}
  \put(12,4){\circle*{1}}
  \put(24,0){\circle*{1}}
  \put(16,8){\circle*{1}}
  \put(8,16){\circle*{1}}
  \put(28,12){\circle*{1}}
  \put(23,27){\circle*{1}}
  \put(20,12){\circle*{1}}
  \put(25,27){\circle*{1}}
  \put(20,36){\circle*{1}}
  \put(28,36){\circle*{1}}
  \put(24,48){\circle*{1}}
  \put(36,44){\circle*{1}}
  \put(36,52){\circle*{1}}
  \put(32,40){\circle*{1}}
  \put(40,32){\circle*{1}}
  \put(11,25){\circle*{1}}
  \put(20,22){\circle*{1}}
  \put(35,17){\circle*{1}}
  \put(12.5,18.5){\circle*{1}}
  \put(27.5,23.5){\circle*{1}}
  \put(36.5,26.5){\circle*{1}}
  \put(48,48){\circle*{1}}
  \put(0,0){\circle{2.5}}
  \put(24,0){\circle{2.5}}
  \put(24,48){\circle{2.5}}
  \put(48,48){\circle{2.5}}
  \put(0,0){\vector(3,-1){12}}
  \put(0,0){\vector(3,1){12}}
  \put(12,-4){\vector(3,1){11}}
  \put(12,4){\vector(3,-1){11}}
  \put(12,-4){\vector(1,3){4}}
  \put(12,4){\vector(-1,3){4}}
  \put(24,0){\vector(1,3){4}}
  \put(24,0){\vector(-1,3){4}}
  \put(16,8){\vector(3,1){12}}
  \put(8,16){\vector(3,-1){12}}
  \put(28,12){\vector(-1,3){5}}
  \put(23,27){\vector(-1,3){3}}
  \put(20,12){\vector(1,3){5}}
  \put(25,27){\vector(1,3){3}}
  \put(20,36){\vector(1,3){3.7}}
  \put(28,36){\vector(-1,3){3.7}}
  \put(24,48){\vector(3,-1){12}}
  \put(24,48){\vector(3,1){12}}
  \put(36,44){\vector(3,1){11}}
  \put(36,52){\vector(3,-1){11}}
  \put(20,36){\vector(3,1){12}}
  \put(32,40){\vector(1,3){4}}
  \put(28,36){\vector(3,-1){12}}
  \put(40,32){\vector(-1,3){4}}
  \put(8,16){\vector(1,3){3}}
  \put(11,25){\vector(3,-1){9}}
  \put(20,22){\vector(3,-1){15}}
  \put(35,17){\vector(1,3){5}}
  \put(16,8){\vector(-1,3){3.5}}
  \put(12.5,18.5){\vector(3,1){15}}
  \put(27.5,23.5){\vector(3,1){9}}
  \put(36.5,26.5){\vector(-1,3){4.5}}
  \put(-10,35){\vector(3,-1){9}}
  \put(-10,35){\vector(3,1){9}}
  \put(-10,35){\vector(1,3){3}}
  \put(-10,35){\vector(-1,3){3}}
  \put(-2.5,28.5){1}
  \put(0,37){4}
  \put(-7,43){2}
  \put(-14,45){3}
  \put(35,0){$K(1,0,0,1)$}
  \put(60,20){$\Longrightarrow$}
    \end{picture}}
%%%                                       right
    \put(90,5){\begin{picture}(40,40)
  \put(0,0){\circle*{1}}
  \put(20,0){\circle*{1}}
  \put(20,30){\circle*{1}}
  \put(40,30){\circle*{1}}
  \put(0,0){\circle{2.5}}
  \put(20,0){\circle{2.5}}
  \put(20,30){\circle{2.5}}
  \put(40,30){\circle{2.5}}
  \put(0,0){\vector(1,0){19}}
  \put(20,0){\vector(0,1){29}}
  \put(20,30){\vector(1,0){19}}
  \put(30,5){$C(1,0)$}
  \put(9,-4.5){$\bar 1$}
  \put(21.5,14){$\bar 2$}
  \put(29,31){$\bar 1$}
 \end{picture}}
%%%
 \end{picture}
 \end{center}
  \caption{Creation of $C(1,0)$.}
  \label{fig:C-cryst}
  \end{figure}
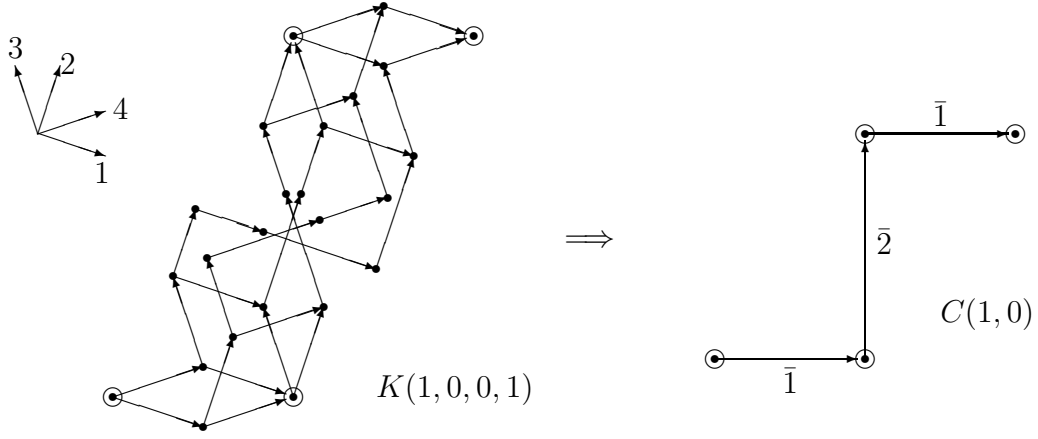

%------------------------Ssec.9.1
 \subsection{A polyhedral description for the vertices of $\frakC$ in case $n=2$}
 \label{ssec:C1}

Consider a self-complementary vertex $v\in S$. Take the corresponding objects
related to $v$: the upper subcrystal $\Kup[a]$ (with colors 1,\,2,\,3) and the
lower subcrystal $\Klow[b]$ (with colors 2,\,3,\,4), the corresponding
deviations $\Delta=(\Delta_1,\Delta_2,\Delta_3)$ and
$\nabla=(\nabla_2,\nabla_3,\nabla_4)$, the hearts $z:=\vup[a|\Delta]$ and
$z':=\vlow[b|\nabla]$, and the middle subcrystal $\Kmid$ (with colors 2,\,3).
(So $\Kmid$ contains $v$ and has the heart $z$ in $\Piup[a]$, and $z'$ in
$\Pilow[b]$.) Denoting $\delta:=\Delta_1$, we observe that
  \begin{equation} \label{eq:Delta-nabla4}
\Delta=(\delta,0,-\delta)\qquad\mbox{and} \qquad \nabla=(-\delta,0,\delta).
  \end{equation}
(Cf. Lemma~\ref{lm:relat}(iii).) Indeed, by the complementarity, we have
$(\nabla_2,\nabla_3,\nabla_4)= (\Delta_3,\Delta_2,\Delta_1)$. This together
with the relations $\nabla_i=-\Delta_{i-1}$ for $i=2,3,4$ (by~\refeq{nablai})
give $\Delta_1=\nabla_4=-\Delta_3=-\nabla_2$ and $\Delta_2=\nabla_3=-\Delta_2$,
yielding~\refeq{Delta-nabla4}.

We also have $b=(b_1,b_2,b_3,b_4)=(a_4,a_3,a_2,a_1)$ (by the complementarity).
Comparing this with the relations $b_1=a_1+\Delta_1^+$ and $b_2=a_2+\Delta_2^+
+\Delta_1^-$ (by~\refeq{bi}) and using~\refeq{Delta-nabla4}, we obtain
  \begin{equation} \label{eq:a4a3}
  a_4=a_1+\delta^+\qquad \mbox{and} \qquad a_3=a_2+\delta^-.
  \end{equation}
Thus, $a$ and $\Delta$ (as well as $b,\nabla$) are determined by
$a_1,a_2,\delta$. The latter triple obeys:
   \begin{gather}
   0\le a_1\le c_1;\qquad 0\le a_2\le c_2; \label{eq:c1-4} \\
   -a_2 \le\delta\le c_1-a_1. \label{eq:c2-4}
   \end{gather}
Here~\refeq{c1-4} is identical to~\refeq{constr1} (and is obvious),
but~\refeq{c2-4} is somewhat different from~\refeq{constr2}. To
see~\refeq{c2-4}, take the parameter $\parup$ and the heart coordinate
$\heartup$ of $\Kup[a]$; they are $\parup_i=c_i-a_i+a_{i+1}$ and
$\heartup_i=a_{i+1}$ for $i=1,2,3$ (cf.~\refeq{par_up},\refeq{heart_up}). Since
$\Delta=(\delta,0,-\delta)$, the evident relations $-\heartup_i\le\Delta_i\le
\parup_i-\heartup_i$, ~$i=1,2,3$, are reduced to
  $$
  -a_2\le\delta\le c_1-a_1\qquad \mbox{and}\qquad -a_4\le-\delta\le c_3-a_3.
  $$
The left expression is exactly~\refeq{c2-4} (and both inequalities in it are
essential), whereas the right expression is redundant (since~\refeq{a4a3}
implies $a_4\ge a_1+\delta\ge\delta$ and $a_3-\delta\le a_2\le c_2=c_3$).

(An additional fact is that~\refeq{a4a3},\refeq{c1-4},\refeq{c2-4} imply $0\le
a\le c$. This is because $0\le a_1\le a_4=a_1+\delta^+\le c_1=c_4$ and $0\le
a_2+\delta^-=a_3\le a_2\le c_2=c_3$.)

One more, fourth, ingredient in the description of $v$ comes up when we
consider the location of $v$ in the middle subcrystal $\Kmid$ (this is a
symmetric $A_2$-crystal, and our analysis becomes more involved compared with
the $A_3\to B_2$ case where the corresponding middle subcrystal is a path). We
rely on the following (using terminology from Section~\ref{ssec:A2}).
  \begin{lemma} \label{lm:v-diag}
~$\Kmid$ is symmetric, and a vertex $v$ of $\Kmid$ belongs to $S$ if and only
if $v$ lies on the diagonal $D=\{(i,i)\}$ of the critical lattice
$\Pi^{\uparrow\downarrow}$ of $\Kmid$. In particular, $\Kmid$ contains exactly
$|D|$ self-complementary vertices.
  \end{lemma}
\begin{proof} Take a path in $\Kmid$ going from $v$ to the sink
$\tilde t$ of $\Kmid$, and let $P'$ be the path in $K$ going from $v$ and
complementary to $P$. Then $P'$ is contained in $\Kmid$ (since the edges of
both $P,P'$ have colors only 2 and 3). Since the end $\tilde t$ of $P$ has no
outgoing edges of colors 2 and 3, so does the end $\sigma(\tilde t)$ of $P'$.
This implies $\sigma(\tilde t)=\tilde t$, i.e. $\tilde t\in S$. Similarly,
considering a path in $\Kmid$ going from the source $\tilde s$ of $\Kmid$ to
$v$ (or to $\tilde t$), we can conclude that $\tilde s\in S$.

Consider a vertex $u$ of $\Kmid$. If $u$ belongs to the critical lattice
$\Pi^{\uparrow\downarrow}$ and has coordinates $(\alpha,\beta)$ in it, then $u$
is expressed as $(\bftwo\bfthree)^\beta(\bfthree\bftwo)^\alpha(\tilde s)$. Then
the complementary vertex $\sigma(u)$ is
$(\bfthree\bftwo)^\beta(\bftwo\bfthree)^\alpha(\tilde s)$, and $u=\sigma(u)$
holds if and only if $\alpha=\beta$, i.e. $u$ lies on the diagonal $D$. Also
when $u=\tilde t$, the pair $(\alpha,\beta=\alpha)$ becomes the parameter of
$\Kmid$; so $\Kmid$ is symmetric. Finally, if $u$ is not in
$\Pi^{\uparrow\downarrow}$ and belongs to a right (left) sail, then $\sigma(u)$
must belong to the complementary left (resp. right) sail, whence $\sigma(u)\ne
u$. \hfill
\end{proof}

  \noindent
{\bf Remark 6.} For two consecutive elements $u=(i,i)$ and $w=(i+1,i+1)$ of
$D$, we have $w=\bftwo\bfthree\bfthree\bftwo(u)$. Also
$\bfthree\bftwo(u)\ne\bftwo\bfthree(u)$. This leads to the following
consequence of Lemma~\ref{lm:v-diag}: if $u,w$ are two self-complementary
vertices of a symmetric $A_{2n}$-crystal $K$ and if $w$ is obtained from $u$ by
applying a string of operators $F_n$ and $F_{n+1}$, then
$w=(F_nF_{n+1}^2F_n)^q(u)$ for some integer $q\ge 0$. This justifies the
definition of edges of color $\bar n$ in the $n$-colored extract $\frakC$.
\medskip

  \noindent
{\bf Remark 7.} For a vertex $u$ of the $A_4$-crystal $K$ and a color $i$,
define $\eps_{i}(u):=h_{i}(u)-t_{i}(u)$, where $h_{i}(u)$ and $t_{i}(u)$ are
the lengths of the maximal paths in $K$ having color $i$ and going from $u$ and
to $u$ , respectively (cf. Section~\ref{ssec:typeA}). For a vertex $u$ of
$\frakC$ and a color $\bar i$, the values $h_{\bar i},\, t_{\bar i}$ and
$\eps_{\bar i}(u)$ are defined in a similar way. Considering an edge $(u,w)$ of
color $\bar 2$ in $\frakC$, we have (using axiom (A2) for $K$):
  \begin{multline*}
 \qquad\eps_{\bar 1}(u)=\eps_{1}(u)=\eps_{1}(\bftwo(u))-1=\eps_{1}(\bfthree\bftwo(u))-1\\
=\eps_{1}(\bfthree\bfthree\bftwo(u))-1=\eps_{1}(\bftwo\bfthree\bfthree\bftwo(u))-2
 =\eps_{\bar 1}(w)-2,
\qquad
  \end{multline*}
in view of $w=\bftwo\bfthree\bfthree\bftwo(u)$ and $u,w\in S$. In its turn, an
edge $(u,w)$ of color $\bar 1$ in $\frakC$ gives
 $$
 \eps_{\bar 2}(u)=\eps_{2}(u)=\eps_{2}(\bfone(u))-1=\eps_{2}(\bffour\bfone(u))-1
 =\eps_{\bar 2}(w)-1.
 $$
As a consequence for $n$-colored extracts, we obtain the relation
$\eps_{\bar{n-1}}(u)-\eps_{\bar{n-1}}(w)=-2$ for an edge $(u,w)$ of color
$\bar{n}$, and the relation $\eps_{\bar{n}}(u)-\eps_{\bar{n}}(w)=-1$ for an
edge $(u,w)$ of color $\bar{n-1}$. This hints that the pair of colors
$\bar{n-1},\,\bar{n}$ behaves as prescribed by the Cartan coefficients
$m_{n-1,n}$ and $m_{n,n-1}$ for $C_n$-crystals figured in Axiom~(BC4$'$).
\smallskip

Return to $v$ as above. Using~\refeq{par_mid}, we express the parameter
$\parmid$ of $\Kmid$ by
  \begin{equation} \label{eq:c2c3mid}
  \parmid_2=c_2-\Delta_2+\Delta_1=c_2+\delta\qquad \mbox{and} \qquad
  \parmid_3=\parmid_2.
  \end{equation}

By Lemma~\ref{lm:v-diag}, the vertex $v$ is a point $(\rho,\rho)$ in the
diagonal $D$. Then~\refeq{c2c3mid} implies that the number $\rho$ satisfies
  \begin{equation} \label{eq:rho_ineq}
  0\le\rho\le c_2+\delta.
  \end{equation}

The obtained quadruple $(a_1(v),a_2(v),\delta(v),\rho(v))$ determining a vertex
$v\in S$ is just what we call the \emph{description} of $v$ in the graph
$\frakC$ with the ``first'' color $\bar 1=(1,\,4)$ and the ``second'' color
$\bar 2=(2,\,3)$. It satisfies \refeq{c1-4},\refeq{c2-4},\refeq{rho_ineq}. It
is straightforward to check the corresponding converse property, and we can
conclude with the following polyhedral result (cf. Theorem~\ref{tm:S-constr}
for the $A_3$ case).
  \begin{theorem} \label{tm:poly-4}
The above correspondence $v\mapsto (a_1,a_2,\delta,\rho)$ gives a bijection
between the set of self-complementary vertices of a symmetric $A_4$-crystal
$K(c)$ and the set of integer solutions to the linear system
{\rm\{\refeq{c1-4},\refeq{c2-4},\refeq{rho_ineq}\}}.
 \end{theorem}

%------------------------Ssec.9.2
 \subsection{Relation to the worm graph}
 \label{ssec:C2}

Based on the above linear system, we can associate the vertices of $\frakC$ to
worms, but now the colors $\tilde 1$ and $\tilde 2$ in the worm model turn out
to be ``swapped'': they become related to colors $\bar 2$ and $\bar 1$ in
$\frakC$, respectively (which is agreeable with Remark~7).

More precisely, we should deal with the reversed parameter
$c'=(c'_1,c'_2):=(c_2,c_1)$ and accordingly consider the worm graph $W(c')$
whose vertices (worms) live in the rectangle $R(c')=\{(\alpha,\beta)\colon 0\le
\alpha\le 2c_2,\, 0\le\beta\le c_1\}$. Given $v\in S$ and its corresponding
$a,\delta,\rho$, we construct the worm $w=w(v)=(X',X'',Y',Y'')$ as follows:
  \begin{gather}
  Y':=(a_2+a_3,\,a_1); \qquad Y'':=(a_2+a_3,\, a_4); \label{eq:C-YY} \\
  X':=(2\min\{\rho,a_3\},\, a_1+\min\{(\rho-a_3)^+,\delta^+\});
       \label{eq:C-Xp} \\
  X'':=(2a_2+2(\rho-a_2-\delta)^+,\, a_1+\min\{(\rho-a_3)^+,\delta^+\}).
       \label{eq:C-Xpp}
  \end{gather}

These settings look cumbersome, and to make them more comprehensible, let us
define the following points on $R(c')$:
  $$
  I_0:=(0,a_1),\quad J_0:=(2a_3,a_1), \quad J_1:=(2a_2,a_4), \quad I_1:=(2c_2,a_4).
  $$

Then, using relations~\refeq{a4a3}--\refeq{c2-4},\refeq{rho_ineq}, we can
rewrite~\refeq{C-YY}--\refeq{C-Xpp} in a more enlightening form, given
in~\refeq{C-hor} and~\refeq{C-vert}.

 \begin{numitem}
Let $\delta\le 0$. Then $a_1=a_4$, ~$a_3=a_2+\delta$, and $w$ is an H-worm such
that:
  \begin{itemize}
\item[(i)] $Y$ is the middle point of the horizontal segment $J_0J_1$;
\item[(ii)] if $\rho\le a_3$, then $X'$ is the point $(2\rho,\,a_1)$ occurring
in the horizontal segment $I_0J_0$, and $X''=J_1$;
\item[(iii)] if $\rho\ge a_3$, then $X'=J_0$, and $X''$ is the point
$(2a_2+2(\rho-a_3),\, a_1)$ occurring in the horizontal segment $J_1I_1$.
 \end{itemize}
  \label{eq:C-hor}
  \end{numitem}
(See Fig.~\ref{fig:wormC}(a),(b).) Here (i) is obvious. Part~(ii) follows from
$\min\{\rho,a_3\}=\rho$, ~$\min\{(\rho-a_3)^+,\delta^+\}=0$ and
$\rho-a_2-\delta=\rho-a_3\le 0$ (since $a_3=a_2+\delta^-$); cf.~\refeq{C-Xp}.
And~(iii) follows from $\min\{\rho,a_3\}=a_3$,
~$\min\{(\rho-a_3)^+,\delta^+\}=0$ and $(\rho-a_2-\delta)^+=\rho-a_3$;
cf.~\refeq{C-Xpp}.

  \begin{figure}[hbt]                  % Fig.7
  \begin{center}
  \unitlength=1mm
  \begin{picture}(145,50)
                         %              (a)
   \put(0,30){\begin{picture}(50,20)
   \put(0,10){\circle{1}}
   \put(50,10){\circle{1}}
   \put(12,10){\circle*{2}}
   \put(30,10){\circle*{2}}
   \put(40,10){\circle*{2}}
  \put(0,10){\line(1,0){50}}
  \put(0,8){\line(0,1){4}}
  \put(20,8){\line(0,1){4}}
  \put(50,8){\line(0,1){4}}
  \put(-1,12.5){$I_0$}
  \put(17,12.5){$J_0$}
  \put(40,12.5){$J_1$}
  \put(48,12.5){$I_1$}
  \put(12,5){$X'$}
  \put(29,5){$Y$}
  \put(37,5){$X''$}
  \put(23,14){$|\delta|$}
  \put(33,14){$|\delta|$}
  \put(6,14){$\rho$}
  \put(-2,5){$(0,a_1)$}
  \put(47,5){$2c_2$}
  \put(18,5){$2a_3$}
  \put(0,-2){(a)}
 \qbezier(20,10)(25,15)(30,10)
 \qbezier(30,10)(35,15)(40,10)
 \qbezier(0,10)(6,15)(12,10)
{\thicklines
  \put(12,10.1){\line(1,0){28}}
  \put(12,9.9){\line(1,0){28}}
 }
    \end{picture}}
 %%
                          %              (b)
   \put(0,5){\begin{picture}(50,20)
   \put(0,10){\circle{1}}
   \put(50,10){\circle{1}}
   \put(20,10){\circle*{2}}
   \put(30,10){\circle*{2}}
   \put(45,10){\circle*{2}}
  \put(0,10){\line(1,0){50}}
  \put(0,8){\line(0,1){4}}
  \put(40,8){\line(0,1){4}}
  \put(50,8){\line(0,1){4}}
  \put(18,12.5){$J_0$}
  \put(38,12.5){$J_1$}
  \put(18,5){$X'$}
  \put(28,5){$Y$}
  \put(43,5){$X''$}
  \put(36,5){$2a_2$}
  \put(0,0){(b)}
{\thicklines
  \put(20,10.1){\line(1,0){25}}
  \put(20,9.9){\line(1,0){25}}
 }
    \end{picture}}
 %%
                          %              (c)
   \put(80,25){\begin{picture}(40,25)
   \put(0,10){\circle{1}}
   \put(40,23){\circle{1}}
   \put(9,10){\circle*{2}}
   \put(20,10){\circle*{2}}
   \put(20,23){\circle*{2}}
  \put(0,10){\line(1,0){20}}
  \put(20,23){\line(1,0){20}}
  \put(0,8){\line(0,1){4}}
  \put(40,21){\line(0,1){4}}
  \put(-2,13){$I_0$}
  \put(9,12.5){$X'$}
  \put(18,5){$X''=Y'=J_0$}
  \put(4,22){$Y''=J_1$}
  \put(39,25){$I_1$}
  \put(24,15){$\delta$}
%  \put(33,14){$|\delta|$}
  \put(4,13.5){$\rho$}
  \put(-4,5){$(0,a_1)$}
  \put(38,18){$(2c_2,a_4)$}
  \put(20,25){$2a_3$}
  \put(50,10){(c)}
 \qbezier(20,10)(25,16)(20,23)
 \qbezier(0,10)(4.5,14)(9,10)
{\thicklines
  \put(9,10.1){\line(1,0){11}}
  \put(9,9.9){\line(1,0){11}}
  \put(19.9,10){\line(0,1){13}}
  \put(20.1,10){\line(0,1){13}}
 }
    \end{picture}}
 %%
                          %              (d)
   \put(60,-5){\begin{picture}(40,25)
   \put(0,10){\circle{1}}
   \put(40,23){\circle{1}}
   \put(20,16){\circle*{2}}
   \put(20,10){\circle*{2}}
   \put(20,23){\circle*{2}}
  \put(0,10){\line(1,0){20}}
  \put(20,23){\line(1,0){20}}
  \put(0,8){\line(0,1){4}}
  \put(40,21){\line(0,1){4}}
  \put(22,14){$X$}
  \put(20,6){$X''=Y'$}
  \put(14,23){$Y''$}
  \put(5,2){(d)}
{\thicklines
  \put(19.9,10){\line(0,1){13}}
  \put(20.1,10){\line(0,1){13}}
 }
    \end{picture}}
 %%
                          %              (e)
   \put(105,-5){\begin{picture}(40,25)
   \put(0,10){\circle{1}}
   \put(40,23){\circle{1}}
   \put(29,23){\circle*{2}}
   \put(20,10){\circle*{2}}
   \put(20,23){\circle*{2}}
  \put(0,10){\line(1,0){20}}
  \put(20,23){\line(1,0){20}}
  \put(0,8){\line(0,1){4}}
  \put(40,21){\line(0,1){4}}
  \put(28,18){$X''$}
  \put(21.5,7){$Y'$}
  \put(7,23.5){$X'=Y''$}
  \put(5,2){(e)}
{\thicklines
  \put(20,23.1){\line(1,0){9}}
  \put(20,22.9){\line(1,0){9}}
  \put(19.9,10){\line(0,1){13}}
  \put(20.1,10){\line(0,1){13}}
 }
    \end{picture}}
 \end{picture}
 \end{center}
  \caption{(a) $\delta<0$, ~$\rho<a_3$; ~(b) $\delta<0$, ~$\rho>a_3$;
~(c) $\delta>0$, ~$\rho<a_3$; ~(d) $\delta>0$, ~$a_3<\rho<a_3+\delta$; ~(e)
$\delta>0$, ~$\rho>a_3+\delta$.}
  \label{fig:wormC}
  \end{figure}
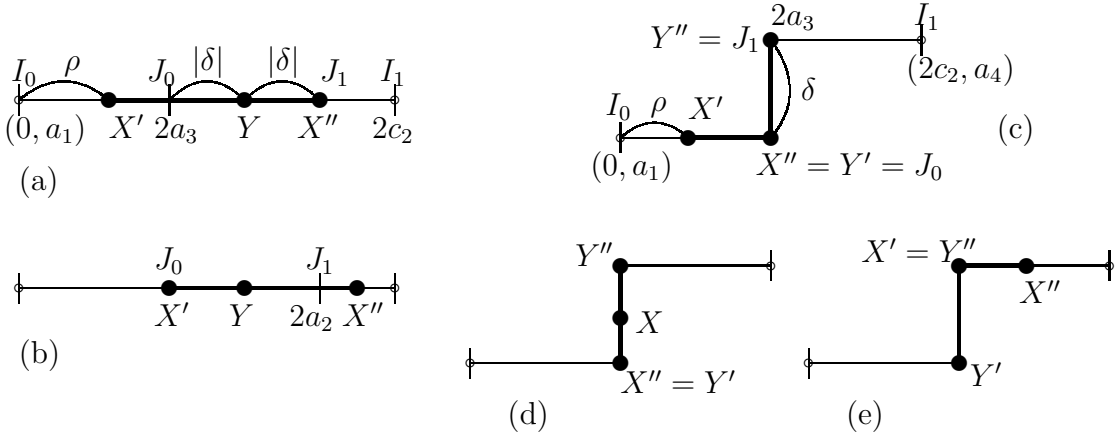
 \begin{numitem}
Let $\delta\ge 0$. Then $a_4=a_1+\delta$, ~$a_3=a_2$, and $w$ is an HV- or V-
or VH-worm such that:
  \begin{itemize}
\item[(i)] $J_0J_1$ is the vertical limb of $w$, ~$Y'=J_0=(2a_3,a_1)$ and
$Y''=J_1=(2a_3,a_4)$;
\item[(ii)] if $\rho\le a_3$, then $w$ is an HV-worm, $X'$ is the point
$(2\rho,\,a_1)$ occurring in the horizontal segment $I_0J_0$, and $X''=J_0$;
\item[(iii)] if $a_3\le\rho\le a_3+\delta$, then $w$ is a V-worm, and $X$ is the point
$(2a_3,\, a_1+\rho-a_3)$ occurring in the vertical limb $J_0J_1$;
\item[(iv)] if $\rho\ge a_3+\delta$, then $w$ is a VH-worm, $X'=J_1$, and $X''$ is the point
$(2a_3+2(\rho-a_3-\delta),\, a_4)$ in the horizontal segment $J_1I_1$.
 \end{itemize}
  \label{eq:C-vert}
  \end{numitem}
(See Fig.~\ref{fig:wormC}(c),(d),(e).) Again, (i) is obvious. Part~(ii) is
provided by $\min\{\rho,a_3\}=\rho$, ~$\min\{(\rho-a_3)^+,\delta^+\}=0$ and
$\rho-a_2-\delta\le 0$. Part~(iii) follows from $\min\{\rho,a_3\}=a_3$,
~$\min\{(\rho-a_3)^+,\delta^+\}=\rho-a_3$ and $\rho-a_2-\delta\le 0$. And~(iv)
follows from $\min\{\rho,a_3\}=a_3$,
~$a_1+\min\{(\rho-a_3)^+,\delta^+\}=a_1+\delta=a_4$ and
$(\rho-a_2-\delta)^+=\rho-a_3-\delta$.

Relations~\refeq{a4a3}--\refeq{c2-4},\refeq{rho_ineq} ensure that $w(v)$ is
indeed a feasible worm on $R(c')$, and one can see that any worm in $W(c')$ can
be obtained by the above construction from some $v\in S$. Furthermore, for an
edge $(u,v)$ of color $\bar 2$ in $\frakC$, we have $a(v)=a(u)$,
$\delta(v)=\delta(u)$ and $\rho(v)=\rho(u)+1$.
Considering~\refeq{C-hor},\refeq{C-vert}, one can realize that the worm $w(v)$
is obtained from $w(u)$ by applying the \emph{first} operator on $W(c')$ (which
updates $X',X''$ and preserves $Y',Y''$).

Thus, we can conclude with the following result analogous to
Proposition~\ref{pr:S-worm}.
  \begin{prop}  \label{pr:S-worm-C}
By the above construction, the correspondence $v\mapsto w(v)$ is a bijection
between the vertices of the symmetric extract $\frakC$ from an $A_4$-crystal
$K(c_1,c_2,c_2,c_1)$ and the vertices of the worm graph $W=W(c_2,c_1)$. Under
this bijection, the edges of color $\bar 2$ in $\frakC$ are transferred
one-to-one to the edges of color $\tilde 1$ in $W$.
  \end{prop}

%------------------------Ssec.9.3
 \subsection{Verification of edges of color 2}
 \label{ssec:C3}

It remains to prove that the edges of color $\bar 1$ in $\frakC$ are bijective
to the edges of color $\tilde 2$ in $W=W(c_2,c_1)$.

   \begin{prop} \label{pr:edg_C}
For each $v\in S$, the following properties hold:
  \begin{itemize}
\item[\rm(i)] if $\ttwo$ does not act at the worm $w(v)$, then
$\bone$ does not act at $v$;
 \item[\rm(ii)] if $\ttwo$ acts at $w(v)$, then $\bar\bfone$ acts at $v$
and $w(\bar\bfone v)=\tilde\bftwo w(v)$.
\end{itemize}
  \end{prop}

 \begin{proof}
Our approach is similar to used in Section~\ref{ssec:B3}; this allows us to
argue in a more concise manner, omitting details which can be restored by the
similarity.

Given $v\in S$, we define
  $$
  p_1(v):=a_3(v),\quad q_1(v):=a_2(v),\quad p_2(v):=a_1(v),\quad
  q_2(v):=a_4(v),
  $$
consider $\eta(v)$ and $R_v$ as before and proceed by induction on the length
$\eta(v)$ of the worm $w(v)$, using Corollary~\ref{cor:Bv} (where, instead of
$B_v$, one should consider the interval $C_v$ of $C$ between the principal
vertices $\check v[p_2(v),p_1(v),p_1(v),p_2(v)]$ and $\check
v[q_2(v),q_1(v),q_1(v),q_2(v)]$). When $\ttwo$ acts at $w(v)$, we define $u$ so
that  $w(u)=\ttwo w(v)$. As before, either $R_u\subset R_v$ or $R_u\supseteq
R_v$ takes place.

The role of Claim~1 is now played by the following claim whose proof is
similar.
 \medskip

  \noindent
\textbf{Claim 1$'$} ~\emph{{\rm(i)} If $R_u\subset R_v$, then $\bone$ acts at
$v$ and $u=\bone v$. ~{\rm(ii)} If $\bone$ acts at $v$ and if
$\eta(v')<\eta(v)$, where $v':=\bone v$, then $\ttwo$ acts at $w(v)$ and
$u=v'$. } \hfill\qed \medskip

Next we examine three cases (``symmetric'' to the ones in the proof of
Proposition~\ref{pr:edg_color1}); here $p_i,q_i,a_i$ concern $v$. \medskip

\noindent
 \underline{\em Case 1}: ~$w(v)$ is a V-worm, i.e. $p_1=q_1=:r$.
Then $q_2-p_2=\delta\ge 0$. Also $Y'(v)=(2r,\,p_2)$, $Y''(v)=(2r,\,q_2)$ and
$X(v)=(2r,\,y)$ for some $p_2\le y\le q_2$. Let $j:=y-p_2$ (then $0\le
j\le\delta$).

If $j>\delta/2$, then $\Vert Y'(v)X(v)\Vert=j >\delta-j=\Vert X(v)Y''(v)\Vert$.
Hence (cf.~\refeq{edges2}) operator $\ttwo$ acts at $w(v)$ and moves $Y'$ by
one unit up. This gives $R_u\subset R_v$, and we are done by Claim~1$'$(i).

Now let $j\le \delta/2$. The following claim will be proved in the Appendix.
 \medskip

  \noindent
\textbf{Claim 2$'$} ~\emph{When $j\le\delta/2$, operator $\bone$ acts at $v$ if
and only if $q_2<c_1$. } \medskip

In case $q_2=c_1$, neither operator $\bone$ acts at $v$ (by Claim~2$'$), nor
operator $\ttwo$ acts at $w(v)$ (as $Y''$ cannot be moved upward). So we
obtain~(ii) in the proposition.

Thus, we may assume that $q_2<c_1$. This and $\Vert Y'(v)X(v)\Vert=j
\le\delta-j=\Vert X(v)Y''(v)\Vert$ imply that $\ttwo$ acts at $w(v)$ and moves
$Y''$ by one unit up. Then  $w(u)$ is the V-worm with $Y'(u)=(2r,\,p_2)$,
~$Y''=(2r,\,q_2+1)$ and $X(u)=(2r,\,p_2+j)$.

Since Claim~2$'$ is applicable to the subgraph $C_u$ of $C$, the vertex
$v':=\bone v$ belongs to $C_u$, and $w(v')$ is a V-worm with $R_{v'}\subseteq
R_u$. Let
  $$
  Y'(v')=(2r,\,p'_2),\quad Y''(v')=(2r,\,q'_2)\quad \mbox{and}\quad
  X(v')=(2r,p_2+j').
  $$

To show that $p'_2=p_2$, ~$q'_2=q_2+1$ and $j'=j$ (yielding $v'=u$), we argue
as follows. Since $\eta(v')<\eta(v)$ is impossible (by Claim~1$'$(ii) and in
view of $\eta(u)>\eta(v)$), we are in one of the following three cases: (a)
$p'_2=p_2$ and $q'_2=q_2$; ~(b) $p'_2=p_2+1$ and $q'_2=q_2+1$; ~(c) $p'_2=p_2$
and $q'_2=q_2+1$.

Case (a) is impossible. For otherwise we would have $w(v')=\tone^{j'-j} w(v)$,
implying $v'=\btwo^{j'-j}v$ (since $\btwo$ on $C$ corresponds to $\tone$ on
$W$, by Proposition~\ref{pr:S-worm-C}.) But $C$ is graded and $v'=\bone v$.

In cases (b) and (c), acting as in the proof of
Proposition~\ref{pr:edg_color1}, we can construct four routes
$P_1,P_2,Q_1,Q_2$, respectively, from $\check v_0:=\check v[p_2,r,r,p_2]$ to
$v$, from $v$ to $\check v_1:=\check v[q_2+1,r,r,q_2+1]$, from $\check v_0$ to
$v'$, and from $v'$ to $\check v_1$. Moreover, these routes are consistent with
$W$, in the sense that each of their $\bar 1$-edges ($\bar 2$-edges) induces a
$\tilde 2$-edge (resp. $\tilde 1$-edge) in $W$. A direct count (using the
corresponding routes in $W$) gives
   \begin{equation} \label{eq:kP1P2-C}
k_{\bar\bfone}(P_1)=\delta,\quad k_{\bar\bftwo}(P_1)=j,\quad
k_{\bar\bfone}(P_2)=\delta+2,\quad k_{\bar\bftwo}(P_2)=\delta+1-j.
  \end{equation}

In case~(b), a similar count for $Q_2$ gives $k_{\bar\bfone}(Q_2)=q'_2-p'_2=
\delta$. Then concatenating the route $P_1$, the $\bar 1$-edge from $v$ to
$v'$, and the route $Q_2$, we obtain a route $P'$ from $\check v_0$ to $\check
v_1$ such that $k_{\bar\bfone}(P')=2\delta+1$. But $k_{\bar\bfone}(P_1\cdot
P_2)=2\delta+2$; a contradiction.

Thus, case~(c) is only possible. To show that $j=j'$, take the vertices $z,z'$
in $C_u$ whose worms are ``symmetric'' to $w(v),w(v')$, respectively, i.e.
  \begin{gather*}
  Y'(z)=(2r,\,p_2+1),\quad Y''(z)=(2r,\,q_2+1),\quad  X(z)=(2r,q_2+1-j); \\
  Y'(z')=(2r,\,p_2),\quad Y''(z')=(2r,\,q_2+1), \quad  X(z')=(2r,q_2+1-j').
  \end{gather*}
By the isomorphism between $C^{rev}_u$ and $C_u$, we have $z'={\bone}^{-1} z$.
Also $w(v')$ is transformed into $w(z')$ by moving the $X$ point by
$(q_2+1-j')-(p_2+j')=\delta+1-2j'$ points (regarding the upward direction as
positive). Now concatenating the route $P_1$ from $\check v_0$ to $v$, the
$\bar 1$-edge from $v$ to $v'$, the corresponding route from $v'$ to $z'$, the
$\bar 1$-edge from $z'$ to $z$, and the route from $z$ to $\check v_1$
``symmetric'' to $P_1$, we obtain a route $Q$ from $\check v_0$ to $\check v_1$
such that
   $$
  k_{\bar\bftwo}(Q)=j+(\delta+1-2j')+j=\delta+1+2j-2j'.
  $$
On the other hand, $k_{\bar\bftwo}(P_1\cdot P_2)=\delta+1$
(see~\refeq{kP1P2-C}). Hence $j'=j$, as required. \medskip

\noindent
 \underline{\em Case 2}: ~$w(v)$ is an H-worm with $Y(v)\ne X''(v)$.
Then $R_v$  is the horizontal segment connecting the points $X'(v)=(2p_1,f)$
and $X''(v)=(2q_1,f)$, where $f:=p_2=q_2$. Also $Y(v)=(2p_1+j,\,f)$ for some
$0\le j<2h$, where $h:=|\delta|=q_1-p_1$. For $i=0,\ldots,2h$, let $z_i$ denote
the vertex of $C$ such that $R_{z_i}=R_v$ and $Y(z_i)=(2p_1+i,\,f)$; then
$v=z_j$.

By~\refeq{edges2}, for $0\le i<2h$, operator $\ttwo$ transforms the worm
$w(z_i)$ into $w(z_{i+1})$. Define $z'_i:=\bone z_i$. Our goal is to show that
$z'_i=z_{i+1}$ for $i=0,\ldots,2h-1$. The following claim will be proved in the
Appendix.
 \medskip

  \noindent
\textbf{Claim 3$'$} ~\emph{Operator $\bone$ acts at each $z_i$ with $0\le
i<2h$.} \medskip

This claim is applicable to the subgraph $C_v$ of $C$, implying that
$R_{z'_i}\subseteq R_{z_i}=R_v$ for each $i<2h$. The strict inclusion here is
impossible by Claim~1$'$(ii); so $R_{z'_i}=R_v$ and $z'_i=z_{i'}$ for some
$i'\in\{0,\ldots,2h\}$.

The following claim is proved analogously to the proof of Claim~4 in
Section~\ref{ssec:B3}. Here $\check v_0:=\check v[f,p_1,p_1,f]$ and $\check
v_1:=\check v[f,q_1,q_1,f]$. \medskip

  \noindent
\textbf{Claim 4$'$} ~\emph{Let $i\ne h$. Let $P_1$ be a route (in $\frakC$)
from $\check v_0$ to $z_i$, and $P_2$ a route from $z_i$ to $\check v_1$. Then
$k_{\bone}(P_1)=i$ and $k_{\bone}(P_2)=2h-i$. In particular,
$k_{\bone}(P_1\cdot P_2)=2h$.} \hfill\qed
\medskip

This claim enables us to prove the desired equalities $i'=i+1$. If $i\ne h,2h$
and $i'\ne h$, then concatenating a route from $\check v_0$ to $z_i$, the $\bar
1$-edge from $z_i$ to $z_{i'}$, and a route $z_{i'}$ to $\check v_1$, we obtain
a route $Q$ from $\check v_0$ to $\check v_1$ with $k_{\bone}(Q)=i+1+2h-i'$ (by
Claim~4$'$ applied to $i$ and $i'$). This gives $i'=i+1$ (since $k_{\bone}(Q)$
must be equal to $2h$).

If $i=h-1$, then $i'=h$ ($=i+1$) is only possible. And if $i=h$, then $i'=h+1$
follows by considering $C_v^{rev}$ and $i'':=h+1$ (obtaining $\bone^{-1}
z_{i''} =z_h$, whence $i'=i''$).
\medskip

 \noindent
\underline{\em Case 3}: ~$w(v)$ is a proper VH-worm or a proper HV-worm or an
H-worm with $Y(v)= X''(v)\ne X'(v)$. This case is examined in a similar way as
Case~3 from the proof of Proposition~\ref{pr:edg_color1}, relying on the
following claim proved in the Appendix.
 \medskip

  \noindent
\textbf{Claim 5$'$} ~\emph{When $w(v)$ is an HV-worm, $\bone$ acts at $v$ if
and only if $q_2<c_1$.}
\medskip

\noindent This completes the proofs of Proposition~\ref{pr:edg_C},
Theorem~\ref{tm:C3-C4} and Corollary~\ref{cor:A-C}. \hfill\qed\qed
 \end{proof}

%---------------------------  Appendix
 \section*{Appendix. Proofs of claims}

\refstepcounter{section}

In this section we prove the claims from Propositions~\ref{pr:edg_color1}
and~\ref{pr:edg_C} that were left without verification, thus completing the
proofs of main theorems on B- and C-crystals from
Sections~\ref{sec:Bn},\ref{sec:proofB3-B4},\ref{sec:A2n}. In the proofs below
we will extensively use the explicit formulas on the parameters, heart
coordinates and etc. in A-crystals and their subcrystals.

 %-------------------- SSEC A.1
 \subsection*{A.1 ~~Proof of Claims 2,\,3,\,5 from Proposition~\ref{pr:edg_color1}}

Recall that in Section~\ref{sec:proofB3-B4} we considered a self-complementary
vertex $v$ of a symmetric $A_3$-crystal $K=K(c_1,c_2,c_1)$ and its symmetric
extract $B$. Let $a=(a_1,a_2,a_3)$,
$\Delta=(\Delta_1,\Delta_2)=(\delta,-\delta)$, $\ell$,
$\parup=(\parup_1,\parup_2)$, $\heartup=(\heartup_1,\heartup_2)$,
$\parupdown=(\parupdown_2)$, $\heartupdown=(\heartupdown_2)$ be the
corresponding objects concerning $v$ and its related upper subcrystal $\Kup[a]$
and middle subcrystal $\Kmid$ (path of color 2). They are subject to
relations~\refeq{a-delta}--\refeq{par-centr-P}.

The path $\Kmid$, further denoted as $P_2$, is the \emph{lower} subcrystal
containing $v$ in the $A_2$-crystal $\Kup[a]$ with colors 1,2. The heart $z$ of
$P_2$ (i.e. the common vertex of $P_2$ and the critical lattice $\Piup$ of
$\Kup[a]$) has the coordinate $\heartupdown_2$ in $P_2$ (counted from the
beginning of $P_2$) equal to $a_2+\delta$, by~\refeq{par-centr-P}. Hence the
deviation of $v$ from $z$ in $P_2$ is expressed as
   \begin{equation} \label{eq:Dupdown}
   \Dupdown_2:=\ell-\heartupdown_2=\ell-a_2-\delta.
   \end{equation}

The maximal path $P_1$ of color 1 that contains $v$ is an \emph{upper}
subcrystal in $\Kup[a]$. The location of $v$ in $P_1$ is crucial for the claims
that we are going to prove: operator $\bone$ acts at $v$ if and only if $v$ is
\emph{not the last vertex of} $P_1$ (taking into account that either none or
both of operators 1,\,3 act at $v$). In order to find this location, we will
compute the parameter (length) $\parupup_1$ of $P_1$, the coordinate
$\heartupup_1$ of the heart $z'$ of $P_1$ in $P_1$ itself, and the deviation
$\Dupup_1$ of $v$ from $z'$ in $P_1$. Then
\begin{numitem}
$\bone$ acts at $v$ if and only if ~$\heartupup_1+\Dupup_1<\parupup_1$.
  \label{eq:hDcupup}
  \end{numitem}

The values figured in ~\refeq{hDcupup} can be expressed as follows (relying on
the fact that $P_1$, $P_2$ and $v$ are interrelated upper, lower and middle
subcrystals in $\Kup[a]$). The coordinates (locus) $\bup=(\bup_1,\bup_2)$ of
$z$ in $\Piup$ are expressed as $\bup=\heartup+\Delta$. Then~\refeq{heart3} and
the relations $\Delta=(\delta,-\delta)$ and $a_3=a_1+\delta^+$ give
  \begin{equation}  \label{eq:bup12}
  \bup_1=a_2+\delta\quad\mbox{and} \quad \bup_2=a_3-\delta=a_1+\delta^+-\delta
  =a_1-\delta^-.
  \end{equation}

Applying Theorem~\ref{tm:mainA} to $\Kup[a]$, $P_1$, $P_2$, we observe that the
deviation $\Dupup_1$ is equal to $-\Dupdown_2$, and the coordinates (locus)
$\aup=(\aup_1,\aup_2)$ of $z'$ in $\Pi$ are expressed as
$\aup_1=\bup_1+{\Dupdown_2}^-$ and $\aup_2=\bup_2+{\Dupdown_2}^+$
(cf.~\refeq{ai}). Using~\refeq{Dupdown}, we have
   \begin{gather}
   \Dupup_1=-\Dupdown_2=a_2+\delta-\ell; \label{eq:Dupup} \\
   \aup_1=\bup_1+{\Dupdown_2}^-=a_2+\delta+(\ell-a_2-\delta)^-
                =\min\{\ell,\,a_2+\delta\}; \label{eq:aup1} \\
  \aup_2=\bup_2+{\Dupdown_2}^+=a_1-\delta^- +(\ell-a_2-\delta)^+. \label{eq:aup2}
    \end{gather}

Also (cf.~\refeq{par_up},\refeq{heart_up}):
  \begin{equation}  \label{eq:cch}
  \parup_1=c_1-a_1+a_2,\quad \parupup_1=\parup_1-\aup_1+\aup_2 \quad
  \mbox{and}\quad\heartupup_1=\aup_2.
  \end{equation}

Now~\refeq{bup12}--\refeq{cch} enable us to precisely compute the desired
quantity:
   \begin{multline*}
\parupup_1-\heartupup_1-\Dupup_1=(\parup_1-\aup_1+\aup_2) -\aup_2 -(a_2+\delta-\ell) \\
 =(c_1-a_1+a_2)-\min\{\ell,\,a_2+\delta\}-a_2-\delta+\ell \\
 =c_1-a_1-\delta+\ell-\min\{\ell,\,a_2+\delta\}=:\omega.
   \end{multline*}

Thus (by~\refeq{hDcupup}), $\bone$ acts at $v$ if and only if $\omega>0$. (Note
that simultaneously $\omega$ is equal to the length of the maximal $\bar
1$-colored path from $v$ in $B$, i.e. to $h_{\bar 1}(v)$, using notation from
Section~\ref{sec:prelim}.)

Now we are ready to prove Claims~2,\,3,\,5. \smallskip

1) The condition $j\le\delta$ in the hypotheses of Claim~2 is equivalent to
$\ell\le a_2+\delta$ (this is seen by considering the actions of operator
$\ttwo$ on $W$ described in Case~I of Section~\ref{ssec:B2}). Hence (in view of
$a_1+\delta=a_3$)
  $$
  \omega=c_1-a_1-\delta+\ell-\ell=c_1-(a_1+\delta)=c_1-a_3=c_1-q_1,
  $$
proving Claim~2. \smallskip

2) Similar conditions $\delta\ge 0$ and $\ell\le a_2+\delta$ hold in Claim~5,
and we again obtain $\omega=c_1-q_1$. \smallskip

3) Since $w(v)$ in Claim~3 is a V-worm, we have $\delta\le 0$ and $\ell\ge
a_2+\delta$ (see Case~II in Section~\ref{ssec:B2}), whence $\min\{\ell,\,
a_2+\delta\}=a_2+\delta$. Also $c_1\ge a_1$. Then
   $$
\omega=c_1-a_1-\delta+\ell-a_2-\delta\ge \ell-a_2-2\delta.
  $$
Observe that $X(v)$ is the point $(a_1,a_2)$ and $Y''(v)$ is the point
$(a_1,\,\ell-2\delta)$. Now $\omega>0$ follows from the condition that $X(v)$
lies below $Y''(v)$. \hfill\qed

 %-------------------- SSEC A.2
 \subsection*{A.2 ~~Proof of Claims~2$'$,\,3$'$,\,5$'$ from Proposition~\ref{pr:edg_C}}

In Section~\ref{ssec:C3} we considered a self-complementary vertex $v$ of an
$A_4$-crystal $K=K(c_1,c_2,c_2,c_1)$ and related $a,\delta,\rho$. Compared with
the previous case, we are now forced to handle more subcrystals of $K$ that
contain $v$, namely, $K',\,\Kpd,\,\Kpu,\, P_2,\, P_1$, where:

(i) $K'$ has colors 1,2,3 (it is just the upper subcrystal $\Kup[a]$ of $K$);

(ii) $\Kpd$ has colors 2,3; it is a lower subcrystal of $K'$;

(iii) $\Kpu$ has colors 1,2; it is an upper subcrystal of $K'$;

(iv) $P_2$ has color 2; this path is an upper subcrystal of $\Kpd$, a lower
subcrystal of $\Kpu$, and a middle subcrystal of $K'$;

(v) $P_1$ has color 1; this path is an upper subcrystal of $\Kpu$. \smallskip

\noindent Accordingly we denote:

(vi) the principal lattices of $K,\,K',\,\Kpd,\,\Kpu$ by $\Pi,\,\Pi',\,
\Pipd,\,\Pipu$, respectively;

(vii) the unique elements of $\Pi\cap K',\, \Pi'\cap \Kpd,\, \Pi'\cap \Kpu,\,
\Pipd\cap P_2,\, \Pipu\cap P_2,\, \Pipu\cap P_1$ by $z,\,
\zd,\,\zu,\,\zdu,\,\zud,\,\zuu$, respectively (these are the hearts of
corresponding subcrystals);

(viii) the parameters of $K',\,\Kpd,\,\Kpu,\,P_1$ by $c',\, \cpd,\, \cpu,\,
\cpuu$, respectively (each being a duly indexed vector; e.g.,
$c'=(c'_1,c'_2,c'_3)$, $\cpd=(\cpd_2,\cpd_3)$, $\cpuu=(\cpuu_2)$).
\smallskip

Considering one or another heart $z^\bullet$, we denote its coordinate in the
principal lattice of the smaller subcrystal by $\hslash'^\bullet$; e.g., $\hpd$
concerns $\zd$ in $\Pipd$, and $\hpud$ concerns $\zud$ in $P_2$. Notation for
additional objects (such as deviations, loci, et al.) will be specified on the
way.

Like the previous case (cf.~\refeq{hDcupup}), the following property is
evident:
\begin{numitem}
$\bone$ acts at $v$ if and only if ~$\cpuu_1-\hpuu_1-\Dpuu_1>0$,
  \label{eq:chDuu}
  \end{numitem}
where $\Dpuu_1$ is the deviation of $v$ from $\zuu$ in $P_1$. To express the
quantity figured in~\refeq{chDuu} in terms of $a_1,a_2,\delta,\rho$ takes some
technical efforts. We will use the following auxiliary values:
  \begin{equation} \label{eq:phi-psi}
  \phi:=\rho-a_2-\delta\qquad\mbox{and}\qquad \psi:=\rho-a_2-\delta^-.
  \end{equation}

Recall that the tuple $a=(a_1,a_2,a_3,a_4)$ (satisfying~\refeq{a4a3}) is the
locus of the heart of $K'$ in $\Pi$, and that the deviation in $\Pi'$ of the
heart $\zd$ of $\Kpd$ from the heart $v[a]$ of $K'$ is
$\Delta=(\Delta_1,\,\Delta_2,\,\Delta_3)= (\delta,0,-\delta)$
(by~\refeq{Delta-nabla4}).

Let $a',b',\apu,\bpu$ denote the loci of $\zu$ in $\Pi'$, of $\zd$ in $\Pi'$,
of $\zuu$ in $\Pipu$, of $\zud$ in $\Pipu$, respectively. By~\refeq{par_up}
and~\refeq{heart_up} (applied to appropriate subcrystals), we have
 \begin{gather}
 c'_1=c_1-a_1+a_2,\qquad \cpu_1=c'_1-a'_1+a'_2,\qquad
 \cpuu_1=\cpu_1-\apu_1+\apu_2;           \label{eq:cpuu} \\
 \hslash'_i=a_{i+1}\;\; (i=1,2,3),\qquad \hpu_i=a'_{i+1}\;\; (i=1,2), \qquad
                   \hpuu_1=\apu_2. \label{eq:hpuu}
  \end{gather}

Since $b'=\hslash'+\Delta$, the first relation in~\refeq{hpuu} gives
  \begin{equation} \label{eq:bp12}
  b'_1=\hslash'_1+\delta=a_2+\delta\quad \mbox{and}\quad
      b'_2=\hslash'_2+0=a_3=a_2+\delta^-.
  \end{equation}
By~\refeq{bp12} and~\refeq{heart_low} (applied to $K',\Kpd$), the locus $\hpd$
of $\zd$ in $\Pipd$ is computed as
   \begin{equation} \label{eq:hpd23}
   \hpd_2=b'_1=a_2+\delta\qquad \mbox{and}\qquad \hpd_3=b'_2=a_2+\delta^-.
   \end{equation}

Consider $K'$ and its lower, upper and middle subcrystals containing $v$,
namely, $\Kpd$, $\Kpu$ and $P_2$, respectively. Let $\Dpd$ denote the deviation
of $\zdu$ from $\zd$ in $\Pipd$, and $\Dpu$ the deviation of $\zud$ from $\zu$
in $\Pipu$. We know (cf. Lemma~\ref{lm:v-diag} and~\refeq{c2c3mid}) that the
subcrystal $\Kpd$ (with colors 2,3) is symmetric and has the parameter
$\cpd_i=c_2+\delta$ ($i=2,3$). Also the vertex $v$ is the point $(\rho,\rho)$
in the principal lattice $\Pipd$. Hence $v$ coincides with $\zdu$ (the heart of
the path $P_2$ w.r.t. $\Pipd$). These facts give (using~\refeq{phi-psi}
and~\refeq{hpd23}):
   \begin{equation} \label{eq:Dpd23}
   \Dpd_2=\rho-\hpd_2=\rho-a_2-\delta=\phi \quad\mbox{and} \quad
   \Dpd_3=\rho-\hpd_3=\rho-a_2-\delta^-=\psi.
   \end{equation}
This enables us to compute $\Dpu$ and $a'$. Namely (using~\refeq{nablai}
and~\refeq{ai}):
  \begin{equation} \label{eq:Dpu12}
  \Dpu_1=-\Dpd_2=-\phi\qquad \mbox{and} \qquad \Dpu_2=-\Dpd_3=-\psi;
  \end{equation}
and
  \begin{gather}
  a'_1=b'_1+{\Dpd_2}^-=a_2+\delta+\phi^-; \qquad\qquad\qquad \label{eq:ap12} \\
  a'_2=b'_2+{\Dpd_2}^+ +{\Dpd_3}^-=a_2+\delta^-+\phi^+ +\psi^-. \nonumber
  \end{gather}
Also the locus $\bpu$ of $\zud$ in $\Pipu$ is expressed (using~\refeq{hpuu}
and~\refeq{Dpu12}) as
  \begin{equation} \label{eq:bpu12}
  \bpu_1=\hpu_1+\Dpu_1=a'_2-\phi \quad\mbox{and}\quad
    \bpu_2=\hpu_2+\Dpu_2=a'_3-\psi.
    \end{equation}

Next we use the fact that $P_2$, $P_1$ and $v$ are the lower, upper and middle
subcrystals of $\Kpu$, respectively. The coordinate $\hpud_2$ of $\zud$ in
$P_2$ is equal to $\bpu_1$ (cf.~\refeq{heart_low}), and the coordinate of $v$
in $P_2$ is equal to $\rho$ (since the locus of $v=\zdu$ in $\Pipd$ is
$(\rho,\rho)$). Hence the deviation $\Dpud_2$ of $v$ from $\zud$ in $P_2$ is
$\rho-\bpu_1$, and we have (using~\refeq{Dpu12} and~\refeq{bpu12}):
   \begin{equation} \label{eq:Dpud2}
   \Dpud_2=\rho-\bpu_1=\rho-a'_2+\phi=\rho-a_2-\delta^--\phi^+-\psi^-+\phi
   =\phi^-+\psi^+.
   \end{equation}

Finally, we have
  \begin{equation} \label{eq:Dpuu1}
  \Dpuu_1=-\Dpud_2\qquad \mbox{and} \qquad \apu_1=\bpu_1+{\Dpud_2}^-
  \end{equation}
(cf.~\refeq{nablai} and~\refeq{ai}), where $\apu$ is the locus of $\zuu$ in
$\Pipu$.

The obtained formulas enable us to compute the desired quantity:
   \begin{gather*}
   \cpuu_1-\hpuu_1-\Dpuu_1=(\cpu_1-\apu_1+\apu_2)-\apu_2-\Dpuu_1 \qquad
             \qquad           \hfill\mbox{(by \refeq{cpuu},\refeq{hpuu})} \\
   =(c'_1-a'_1+a'_2)-(\bpu_1+{\Dpud_2}^-)+\Dpud_2 \qquad
               \qquad\qquad  \hfill\mbox{(by \refeq{cpuu},\refeq{Dpuu1})} \\
   =(c'_1-a'_1+a'_2)-(a'_2-\phi)+{\Dpud_2}^+ \qquad
               \qquad \qquad\qquad         \hfill\mbox{(by \refeq{bpu12})} \\
   =c'_1-a'_1+\phi+(\phi^-+\psi^+)^+    \qquad
               \qquad\qquad\qquad\qquad     \hfill\mbox{(by \refeq{Dpud2})} \\
   =(c_1-a_1+a_2)-(a_2+\delta+\phi^-)+\phi+(\phi^-+\psi^+)^+ \qquad
                                    \hfill\mbox{(by \refeq{cpuu},\refeq{ap12})} \\
   =c_1-a_1-\delta+\phi^++(\phi^-+\psi^+)^+
                               \qquad \qquad\qquad\qquad\qquad\qquad\qquad\\
   =c_1-a_1-\delta+(\rho-a_2-\delta)^++((\rho-a_2-\delta)^-
   +(\rho-a_2-\delta^-)^+)^+ =: \omega'.
       \end{gather*}

Now we are ready to prove Claims 2$'$,\,3$'$,\,5$'$. \smallskip

1) In the hypotheses of Claim~2$'$, ~$w(v)$ is a V-worm; therefore,
$\delta=a_4-a_1\ge 0$, ~$a_2=a_3$ and $a_2\le\rho\le a_2+\delta$
(cf.~\refeq{C-vert}(iii)). Then $\phi=\rho-a_2-\delta\le 0$ and $\psi=\rho-a_2
\ge 0$. Also $j=\rho-a_2\le\delta/2$. We have
   $$
   \omega'=c_1-a_1-\delta+(\rho-a_2-\delta+\rho-a_2)^+
          =c_1-a_4+(2\rho-2a_2-\delta)^+=c_1-a_4.
   $$
Since $a_4=q_2$, ~$\omega'>0$ if and only if $c_1>q_2$, as required (in view
of~\refeq{chDuu}). \medskip

2) In the hypotheses of Claim~3$'$, ~$w(v)$ is an H-worm; therefore, $\delta\le
0$. Moreover, $Y(v)\ne X''(v)$ implies $\delta<0$ (cf.~\refeq{C-hor}). Also
$a_1\le c_1$. Then $\omega'\ge c_1-a_1-\delta>0$.
\medskip

3) Since $w(v)$ in Claim~5$'$ is an HV-worm, ~$\delta=a_4-a_1\ge 0$ and
$\rho\le a_3=a_2$ (cf.~\refeq{C-vert}(ii)). Then $\phi\le 0$ and $\psi\le 0$.
We have $\omega'=c_1-a_1-\delta+(\phi^-+\psi^+)^+=c_1-a_4=c_1-q_2$, as
required. \hfill\qed

\end{document}